\documentclass{article}

\makeatletter
\def\input@path{{iclr2027/}{./}}
\makeatother

\usepackage{peppy}

\usepackage[bottom]{footmisc}
\usepackage{hyperref}
\usepackage{url}
\usepackage{graphicx}

\usepackage{amsmath}
\usepackage{amssymb}
\usepackage{mathtools}
\usepackage{amsthm}
\usepackage{booktabs}
\usepackage{float}
\usepackage{needspace}
\usepackage{tabularx}
\usepackage{colortbl}

\usepackage{tikz}
\usepackage[T1]{fontenc}
  \usepackage{times}

\usepackage[capitalize,noabbrev]{cleveref}

\usetikzlibrary{arrows.meta,positioning,calc,fit,backgrounds}

\definecolor{pepblue}{HTML}{2563EB}
\definecolor{pepteal}{HTML}{0F766E}
\definecolor{pepgreen}{HTML}{2F7D32}
\definecolor{peporange}{HTML}{D97706}
\definecolor{peppurple}{HTML}{6D3FA8}
\definecolor{pepnavy}{HTML}{172554}
\definecolor{pepslate}{HTML}{334155}
\definecolor{pepamber}{HTML}{F59E0B}

\let\peppytextcedilla\c
\let\peppytextcaron\v
\AtBeginEnvironment{thebibliography}{%
  \let\c\peppytextcedilla
  \let\v\peppytextcaron
}

\IfFileExists{iclr2027/commands.tex}{%
  \newcommand{\reals}                  {\mathbb R}

\newcommand{\inprod}[2]              {\langle #1, #2 \rangle}     %

\definecolor{uclablue}{RGB}{39,116,174}
\definecolor{ogreen}{RGB}{60,128,49}

\definecolor{myblue}{RGB}{0,114,178}
\definecolor{mygreen}{RGB}{0,158,115}
\definecolor{myorange}{RGB}{213,94,0}

\newcommand{\ie}{{\it i.e.}}
\newcommand{\etc}{{\it etc.}}

\newcommand{\A}{\boldsymbol{A}}
\newcommand{\B}{\boldsymbol{B}}

\newcommand{\D}{\boldsymbol{D}}

\newcommand{\Y}{\boldsymbol{Y}}

\renewcommand{\a}{\boldsymbol{a}}
\renewcommand{\b}{\boldsymbol{b}}
\renewcommand{\c}{\boldsymbol{c}}

\newcommand{\h}{\boldsymbol{h}}

\renewcommand{\k}{\boldsymbol{k}}

\renewcommand{\r}{\boldsymbol{r}}
\newcommand{\s}{\boldsymbol{s}}

\renewcommand{\v}{\boldsymbol{v}}

\newcommand{\x}{\boldsymbol{x}}
\newcommand{\y}{\boldsymbol{y}}

\newcommand{\cB}{\mathcal{B}}

\newcommand{\cS}{\mathcal{S}}
\newcommand{\cT}{\mathcal{T}}

\newcommand{\vc}{{\mathbf{c}}}
\newcommand{\vd}{{\mathbf{d}}}
\newcommand{\ve}{{\mathbf{e}}}

\newcommand{\vg}{{\mathbf{g}}}

\newcommand{\vu}{{\mathbf{u}}}
\newcommand{\vv}{{\mathbf{v}}}

\newcommand{\vx}{{\mathbf{x}}}

\newcommand{\vA}{{\mathbf{A}}}
\newcommand{\vB}{{\mathbf{B}}}
\newcommand{\vC}{{\mathbf{C}}}

\newcommand{\vG}{{\mathbf{G}}}

\newcommand{\vP}{{\mathbf{P}}}
\newcommand{\vQ}{{\mathbf{Q}}}

\newcommand{\vS}{{\mathbf{S}}}

\newcommand{\vV}{{\mathbf{V}}}

\newcommand{\oline}[1]{\mkern 1.5mu\overline{\mkern-1.5mu#1}}

\renewcommand{\hbar}{\oline{h}}

\makeatletter
\renewcommand*\env@matrix[1][c]{\hskip -\arraycolsep
	\let\@ifnextchar\new@ifnextchar
	\array{*\c@MaxMatrixCols #1}}
\makeatother

\usepackage{xparse}
\DeclareFontFamily{U}{ntxmia}{}
\DeclareFontShape{U}{ntxmia}{m}{it}{<-> ntxmia }{}
\DeclareFontShape{U}{ntxmia}{b}{it}{<-> ntxbmia }{}
\DeclareSymbolFont{lettersA}{U}{ntxmia}{m}{it}
\SetSymbolFont{lettersA}{bold}{U}{ntxmia}{b}{it}

\ExplSyntaxOn
\NewDocumentCommand{\varmathbb}{m}
{
	\tl_map_inline:nn { #1 }
	{
		\use:c { varbb##1 }
	}
}
\tl_map_inline:nn { ABCDEFGHIJKLMNOPQRSTUVWXYZ }
{
	\exp_args:Nc \DeclareMathSymbol{varbb#1}{\mathord}{lettersA}{\int_eval:n { `#1+67 }}
}
\exp_args:Nc \DeclareMathSymbol{varbbk}{\mathord}{lettersA}{169}
\ExplSyntaxOff

\newcommand{\gdstepone}{
\begin{examplebox}
\textbf{Setup:} $f$ is $L$-smooth convex function, characterized by %
\begin{equation} \label{eq:convex_smooth_ineq}
    \begin{aligned}
        \convexsmoothineq{x}{y} &:= f(y) - f(x) + \inner{\nabla f(y)}{x-y}
        + \frac{1}{2L}\norm{\nabla f(x)-\nabla f(y)}^2 \le 0, \quad \forall x,y.
    \end{aligned}
\end{equation}
\textbf{Performance metric:} $f(x_N) - f(x_\star)$ \\
\textbf{Initial condition:} $\|  x_0 - x_\star \| \le R$ \\
\textbf{Algorithm definition:}  \vspace{-3mm}
$$
x_{k+1} = x_k - \frac{1}{L} \nabla f(x_k).
$$
\end{examplebox}
}

\newcommand{\gdsteptwo}{
\begin{examplebox}
\vspace{-4mm}
\begin{align}
\label{eq:gd_full_pep}
    &\hspace{-1mm} f(x_N) - f(x_\star) - \frac{L}{4N+2} \| x_0 - x_\star \|^2 \\
    &\hspace{-1mm}=
\sum_{i=1}^{N} \lambda_{i-1,i} \convexsmoothineq{x_{i-1}}{x_{i}}
+\sum_{i=0}^{N} \lambda_{\star,i} \convexsmoothineq{x_\star}{x_i}
    - \sum_{i=0}^{N} d_i \norm{s_i}^2
     \nonumber
\end{align}
\end{examplebox}
}

\newcommand{\gdstepthree}{
\begin{examplebox}
\vspace{-4mm}
\begin{align*}
    \hspace{-1mm}V_k &\!:=
\!\sum_{i=1}^{k} \lambda_{i-1,i} \convexsmoothineq{x_{i-1}}{x_{i}}
+\!\sum_{i=0}^{k} \lambda_{\star,i} \convexsmoothineq{x_\star}{x_i}
     - \!\sum_{i=0}^{k} d_i \norm{s_i}^2
\end{align*}
\end{examplebox}
}

\newcommand{\gdstepfour}{
\begin{examplebox}
$$
\begin{aligned}
V_k &= \frac{k+1}{2N-k}\bigl(f(x_k)-f(x_{\star})\bigr)
-\frac{L}{4N+2}\|x_0-x_{\star}\|^2 \\
&\phantom{=}
-\frac{k+1}{2L(2N-k)}\|\nabla f(x_k)\|^2
+\frac{L(2N-2k-1)}{2(2N-k)^2}\|x_{k+1}-x_{\star}\|^2
\end{aligned}
$$
\end{examplebox}
}

\newcommand{\gdstepfivetheorem}{
\textbf{Theorem.}
Let $f$ be convex and $L$-smooth, and let $x_{\star}$ minimize $f$. If $\|x_{0}-x_{\star}\|^{2}\le R^2$, then the gradient descent iterates
$$
x_{k+1}=x_{k}-\frac{1}{L}\nabla f(x_{k})
$$
satisfy the following stronger Lyapunov statement. Set $V_{0}=0$ and, for $1\le k\le N-1$,
$$
\begin{aligned}
V_k &= \frac{k+1}{2N-k}\bigl(f(x_k)-f(x_{\star})\bigr)
-\frac{L}{4N+2}\|x_0-x_{\star}\|^2 \\
&\phantom{=}
-\frac{k+1}{2L(2N-k)}\|\nabla f(x_k)\|^2
+\frac{L(2N-2k-1)}{2(2N-k)^2}\|x_{k+1}-x_{\star}\|^2,
\end{aligned}
$$
and $V_N = f(x_{N})-f(x_{\star}) - \frac{L}{4N+2} \|x_0-x_{\star}\|^2$.
Then $V_{0}\ge V_{1}\ge\cdots\ge V_{N}$. Consequently, for every $N\ge 1$,
$$
f(x_{N})-f(x_{\star})\le \frac{LR^2}{4N+2}.
$$
}

\newcommand{\norm}[1]{\left\| #1 \right\|}
\newcommand{\inner}[2]{ \left\langle #1 ,  #2 \right\rangle }

\newcommand{\pr}[1]{ \left( #1 \right) }

\usepackage[most]{tcolorbox}

\newtcolorbox{examplebox}{
    breakable,
  colback=gray!5,
  colframe=gray!20,
  boxrule=0.5pt,
  arc=2pt,
  left=3pt,
  right=3pt,
  top=2pt,
  bottom=2pt,
  width=\textwidth,
  fontupper=\footnotesize,
}

\newtcolorbox{wideexamplebox}{
  breakable,
  colback=gray!5,
  colframe=gray!20,
  boxrule=0.5pt,
  arc=2pt,
  left=3pt,
  right=3pt,
  top=2pt,
  bottom=2pt,
  width=\textwidth,
  fontupper=\footnotesize,
}

\newtcolorbox{mdbox}{
  breakable,
  colback=gray!5,
  colframe=gray!20,
  boxrule=0.5pt,
  arc=2pt,
  left=3pt,
  right=3pt,
  top=2pt,
  bottom=2pt,
  width=\textwidth,
  fontupper=\sffamily\footnotesize,
}

\tcbuselibrary{listings, skins, breakable}

\newtcblisting{commandbox}[1][]{
  listing only,
  breakable,
  colback=gray!5,
  colframe=gray!30,
  boxrule=0.4pt,
  arc=2pt,
  left=3pt,
  right=3pt,
  top=2pt,
  bottom=2pt,
  listing options={
    basicstyle=\ttfamily\small,
    columns=fullflexible,
    breaklines=true,
    keepspaces=true,
    showstringspaces=false
  },
  #1
}

\newcommand{\convexsmoothineq}[2]{I(#1,#2)}

\newcommand{\SkillImplement}{\texttt{/pep-implement}}
\newcommand{\SkillFullProof}{\texttt{/pep-full-proof}}
\newcommand{\SkillLyapDefine}{\texttt{/lyap-define}}
\newcommand{\SkillLyapVectors}{\texttt{/lyap-vectors}}
\newcommand{\SkillLyapClosedForm}{\texttt{/lyap-closed-form}}

\usepackage{xspace}

\newcommand{\codeavailability}{%
      The code and accompanying examples are available at:
    \begin{center}
      {\color{blue}\url{https://github.com/pepflow-lib/PEPFlow/tree/peppy-v1}}
    \end{center}
}

\usepackage{wrapfig}
\usepackage{enumitem}

\definecolor{pepflowpurple}{RGB}{107,68,151}
\newcommand{\pepflowcolor}[1]{{\color{pepflowpurple}#1}}

\newcommand{\workflowname}{\texorpdfstring{\texttt{\textup{\color{pepflowpurple}Peppy}}}{Peppy}}

\newcommand{\pepflow}{{\texttt{PEPFlow}}\xspace}

\newcommand{\transpose}{\mathsf{T}}

\newcommand{\code}[1]{%
  \tcbox[
    on line,
    boxsep=1pt,
    left=1pt,
    right=1pt,
    top=0pt,
    bottom=0pt,
    arc=2pt,
    colback=gray!8,
    colframe=gray!18,
    boxrule=0.3pt
  ]{\ttfamily\small #1}%
}

\usepackage{array}
\newcolumntype{L}[1]{>{\raggedright\arraybackslash}p{#1}}

\usepackage{longtable}
}{%
}

\newtheorem{theorem}{Theorem}[section]
\newtheorem{lemma}[theorem]{Lemma}
\newcommand{\csinner}[2]{\left\langle #1,#2\right\rangle}

\title{\texorpdfstring{\raisebox{-0.15em}{\includegraphics[height=1.1em]{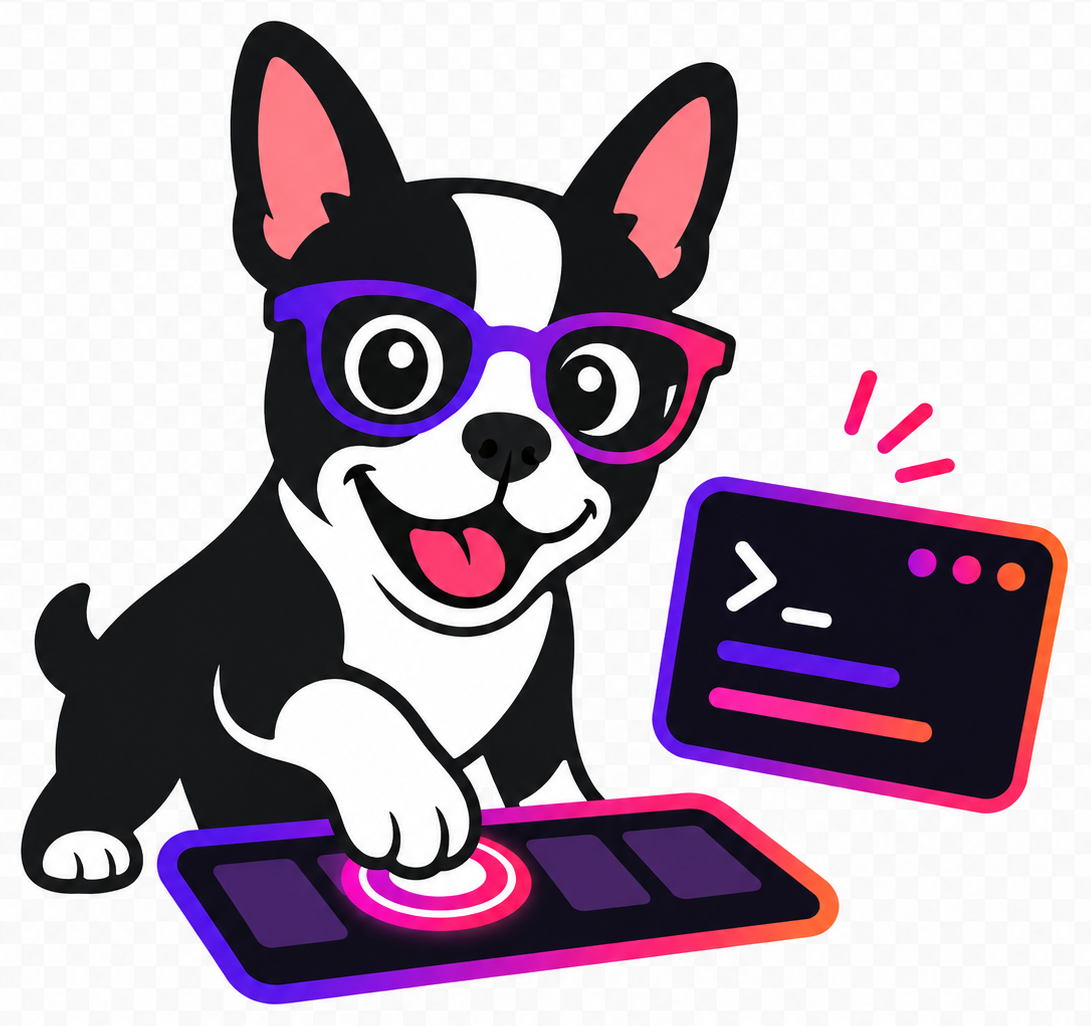}}\hspace{0.3em}}{}\workflowname: An AI-Assisted Workflow for \\Tight Convergence Analysis of Optimization Algorithms}

\author{%
\textbf{Jaewook J. Suh\textsuperscript{1,2}\quad
TaeHo Yoon\textsuperscript{2}\quad
Edward D. H. Nguyen\textsuperscript{3}}\\
\textbf{Bicheng Ying\textsuperscript{4}\quad
Shiqian Ma\textsuperscript{2,5}}\\[0.5em]
{\normalfont\small\textsuperscript{1}Department of Computational Applied Mathematics and Operations Research, Rice University}\\
{\normalfont\small\textsuperscript{2}Department of Applied Mathematics and Statistics, Johns Hopkins University}\\
{\normalfont\small\textsuperscript{3}Department of Electrical and Computer Engineering, Rice University}\\
{\normalfont\small\textsuperscript{4}Google Inc.}\\
{\normalfont\small\textsuperscript{5}Data Science and AI Institute, Johns Hopkins University}%
}

\hypersetup{hidelinks}
  \iclrfinalcopy
\begin{document}

\maketitle
  \fancyhead{}
  \renewcommand{\headrulewidth}{0pt}

\begin{abstract}
This paper presents \workflowname, an AI-assisted workflow for discovering tight, analytic convergence proofs for first-order optimization algorithms.
Generic approaches to using LLMs to conduct mathematical research target an unspecified, broad spectrum of problems and sometimes use the Lean 4 proof assistant for formalization. On the other hand, \workflowname\ leverages domain-specific knowledge more heavily and is thereby capable of constructing the proofs in a more structured manner, which allow a minimal and accessible verification through \texttt{SymPy}.
We experimentally demonstrate through examples that \workflowname\ provides a rigorous, practical, and reproducible paradigm for AI-assisted theorem synthesis in optimization. We further highlight its capability of closing several open problems on tight convergence analysis of first-order optimization algorithms, {including conjectures for Nesterov’s FGM}.
Overall, \workflowname\ is designed to turn the art of optimization algorithm analysis into a science.
\end{abstract}


\section{Introduction}

{Recent advances in large language models (LLMs) have fueled growing interest in AI-assisted research \citep{lu2024aiscientist,bubeck2025early}, including mathematical discoveries \citep{RomeraParedes2024_funsearch,georgiev2025mathematical}. In theorem proving, several systems pursue broad applicability through general-purpose proof assistants \citep{Yang2023_leandojo,thakur2024context,tsoukalas2026advancingmathematicsresearchaidriven}. However, such approaches are not necessarily tailored to structures and computational tools specific to individual mathematical subfields.
This motivates domain-specific workflows that integrate LLMs with a specialized set of tools and knowledge, enabling proof discovery to exploit the structure of the subfield at hand.
Such workflows can make the process of discovery more controlled and targeted.}

{This paper presents a complementary AI-assisted workflow \workflowname, focusing on searching for convergence proofs in optimization theory.
We newly and extensively combine AI with the mathematical framework developed by \citet{YoonSuhNguyenYingMa2026_systematic}, which
detects the Lyapunov function structure numerically identifiable via  Performance Estimation Problem (PEP) \citep{DT14, taylor2017smooth} and streamlines proofs alongside code.
This enables \workflowname\ to leverage AI agents to perform concrete, step-by-step guidelines, and produce novel results such as the first analytic proof of the tight convergence rate of Nesterov's FGM.

PEP itself is often viewed as a computer-assisted proof framework,
which has been proved useful for both finding tight convergence rates and designing novel optimization algorithms \citep{KimFessler2016_optimized, TaylorVanScoyLessard2018_lyapunov, TaylorBach2019_stochastic, RyuTaylorBergelingGiselsson2020_operator, Kim2021_accelerated, dragomir2022optimal, taylor2023optimal, DVR24, UpadhyayaBanertTaylorGiselsson2024_automated}.
However, PEP provides numerical---rather than analytic---certificates for a proof, and traditional research processes utilizing PEP typically involved a largely manual art of converting numerical data into analytic proofs, requiring significant effort from human experts.
While there exist packages such as \texttt{PESTO} \citep{TaylorHendrickxGlineur2017_performance}, \texttt{PEPit} \citep{pepit2024}, \texttt{AutoLyap} \citep{UpadhyayaTaylorBanertGiselsson2025_autolyap} or \texttt{PEPFlow} \citep{SuhYingJiangNguyen2025_pepflow} that support the computational side of helping users formulate PEP, workflows integrating LLMs with PEP for analytic proof synthesis remain scarce:
\citet{SuhYoonNguyenYingMa2026_peppy} and \citet{kim2026domainspecificharnessendtoendautomation} are among the few examples, the former of which being the preliminary version of this work.
}

{\workflowname\ guides AI to perform the procedure of \citet{YoonSuhNguyenYingMa2026_systematic}, which simulate the steps that were traditionally performed by humans.
Within the workflow, \pepflow~\citep{SuhYingJiangNguyen2025_pepflow} serves as the backend for implementing and solving PEPs while exposing the information needed for analytic extensions.
The completion of each step of the workflow can be verified with simple numerical or symbolic checks.
As a final output, \workflowname\ provides a rigorous \textit{theorem} stating an optimization algorithm's tight convergence rate with its proof and code scripts verifying it, with varying degrees of human guidance.

}

\subsection{Contributions}

We present \workflowname, an AI-assisted workflow for discovering analytic proofs of tight convergence bounds for first-order optimization algorithms.
Using \workflowname, we solve several open problems, including tight convergence rates for Nesterov's FGM. Details are provided below.

\begin{itemize}
    \item
    \workflowname\ has several advantages over existing approaches in ``AI4Math'' that typically combine LLMs with formal verification. We list them as follows:
        \begin{itemize}[label=$\circ$]
        \item \workflowname\ draws on the powerful
         \textbf{domain-specific machinery} of PEP to provide a tailored approach to AI-assisted optimization research.
        More specifically, AI agents are guided by concrete instructions based on the systematic framework of \citet{YoonSuhNguyenYingMa2026_systematic} for converting numerical PEP certificates into analytic Lyapunov analyses.
        Users can exploit the power of this well-established strategy %
        sharpened for optimization analysis while needing only a high-level understanding of how \workflowname~operates.

        \item
        \workflowname\ provides AI agents with the \textbf{domain-specific library} \pepflow, which abstracts the underlying PEP components behind an accessible interface.
        {Consequently, AI agents can efficiently handle PEP implementation from conceptual principles, based on an organized source of the necessary domain-specific language.}

        \item
        \workflowname\ uses \textbf{domain-specific proof structures for rapid verification,}
        {which boils down to simple checks of symbolic matrix/vector identities and positivity of algebraically expressed coefficients,
        accelerating the search for verified proofs.}\footnotemark

    \end{itemize}

    \item \textbf{Using \workflowname, we close several conjectures.}
    Specifically, \workflowname\ yields Lyapunov analyses that, to the best of our knowledge, provide the \textbf{first analytic proofs of tight convergence bounds for Nesterov's FGM}.
    These FGM analyses, together with our OGM result, resolve Conjectures~4--5 of \citet{taylor2017smooth}.
    Our analyses of FGM and FISTA with rational coefficients also resolve the corresponding conjectures in Table~1 of \citet{TaylorHendrickxGlineur2017_exacta}.
    Table~\ref{tab:conjecture-results} summarizes these results.

\end{itemize}

Notation and a detailed discussion of prior work are provided in Appendices~\ref{app:notation} and~\ref{app:prior-work}, respectively.
\codeavailability

  \begin{table}[H]
\centering
\caption{
    Conjectures closed by \workflowname. See Sections~\ref{sec:accelerated-case-study} and~\ref{sec:rational-coefficients-case-study} for algorithm definitions.
    }
\label{tab:conjecture-results}
\vspace{2pt}
\begin{tabularx}{\linewidth}{@{}>{\raggedright\arraybackslash}X@{\hspace{0.2em}}c@{\hspace{1.em}}@{}}
\toprule
\multicolumn{1}{c}{Conjecture} & \multicolumn{1}{c}{Corresponding theorem} \\
\midrule
\citet{taylor2017smooth}, Conjecture~4, \ref{eq:case-theta-recursive}, \(y_N\) & Theorem~\ref{thm:case-ofgm-primary} \\
\citet{taylor2017smooth}, Conjecture~5, \ref{eq:case-theta-recursive}, \(x_N\) & Theorem~\ref{thm:case-ofgm-secondary} \\
\citet{taylor2017smooth}, Conjecture~4, \ref{eq:case-ogm-parameters}, \(y_N\) & Theorem~\ref{thm:case-ogm} \\
\citet{TaylorHendrickxGlineur2017_exacta}, conjecture in Table~1, \ref{eq:case-theta-linear}, \(y_N\) & Theorem~\ref{thm:case-sfgm} \\
\citet{TaylorHendrickxGlineur2017_exacta}, conjecture in Table~1, \ref{eq:case-fista}, \(y_N\) & Theorem~\ref{thm:case-fista} \\
\bottomrule
\end{tabularx}
\end{table}

\begin{figure}[!b]
    \centering
    \resizebox{\textwidth}{!}{%
        \begingroup
\tikzset{
  stage/.style={rounded corners=8pt, minimum width=2.38cm, minimum height=2.75cm,
    align=center, line width=0.9pt, inner xsep=1pt, inner ysep=3pt},
  io/.style={rounded corners=5pt, draw=pepslate!45, fill=white, line width=.7pt,
    align=center, inner xsep=8pt, inner ysep=7pt, font=\sffamily\Large\bfseries, text=black},
  flow/.style={-{Latex[length=2.0mm,width=1.6mm]}, line width=0.85pt, draw=pepslate!75},
  dashflow/.style={-{Latex[length=1.8mm,width=1.4mm]}, dashed, line width=0.75pt, draw=pepamber!85!black},
  badge/.style={circle, minimum size=6.6mm, inner sep=0pt, text=white, font=\sffamily\bfseries\normalsize},
  layer/.style={rounded corners=4pt, draw=pepnavy!45, fill=pepnavy!3,
    line width=.75pt, align=center, inner xsep=7pt, inner ysep=4pt,
    font=\sffamily\normalsize\bfseries, text=pepnavy}
}

\newcommand{\StageCard}[8]{%
  \node[stage, draw=#2, fill=#7] (#1) at (#3,0) {%
    \begin{tabular}{@{}c@{}}
      \tikz\node[badge, fill=#8]{#4};\\[1.5mm]
      {\Large #5}\\[.5mm]
      {\parbox[c][1.0cm][c]{2.6cm}{\centering\large\bfseries #6}}
    \end{tabular}
  };
}

\begin{tikzpicture}[font=\sffamily, >=Latex]

\StageCard{s1}{pepslate!45}{0.00}{1}{$I(x,y)$}{Implement PEP}{white}{black}
\StageCard{s2}{pepslate!45}{3.15}{2}{$\lambda,S$}{Full-PEP certificate}{white}{black}
\StageCard{s3}{peppurple}{6.30}{3}{$V_k$}{Defined via\\partial sums}{peppurple!7}{peppurple}
\begin{scope}[stage/.append style={inner xsep=4pt}]
\StageCard{s4}{peppurple}{9.45}{4}{$\{z_i\}$}{\makebox[\linewidth][c]{\shortstack{Identify\\basis pattern}}}{peppurple!7}{peppurple}
\end{scope}
\StageCard{s5}{peppurple}{12.60}{5}{$V_{k+1}\leq V_k$}{Analytic theorem}{peppurple!7}{peppurple}

\draw[draw=black!65, dashed, line width=.7pt]
  ($(s1.north west)+(-.08,.10)$) -- ($(s1.north west)+(-.08,.92)$)
  -- ($(s5.north east)+(.08,.92)$) -- ($(s5.north east)+(.08,.10)$);
\draw[draw=black!65, dashed, line width=.7pt]
  ($(s4.north east)!0.50!(s5.north west)+(0,.92)$)
  -- ($(s4.north east)!0.50!(s5.north west)+(0,.10)$);
\node[font=\sffamily\Large\bfseries, text=black, fill=white, inner sep=2pt]
  at ($(s2.north)!0.50!(s3.north)+(0,.92)$) {Numerical discovery};
\node[font=\sffamily\Large\bfseries, text=black, fill=white, inner sep=2pt]
  at ($(s5.north)+(0,.92)$) {\shortstack{Symbolic\\synthesis}};

\node[io, left=1.15cm of s1] (input) {Algorithm\\\&\\Problem class};
\node[io, right=1.15cm of s5] (output) {Theorem\\\&\\Proof};
\draw[flow] (input.east) -- (s1.west);
\foreach \a/\b in {s1/s2,s2/s3,s3/s4,s4/s5}{\draw[flow] (\a.east) -- (\b.west);}
\draw[flow] (s5.east) -- (output.west);

\node[layer, minimum width=7.8cm] (pepflow) at (6.30,-2.55)
  {Programmable layer supported by \pepflow \quad {\normalsize Abstract vectors $\leftrightarrow$ Python vector objects}};
\foreach \s in {s1,s2,s3,s4,s5}{
  \draw[densely dotted, line width=.6pt, draw=pepnavy!35] (\s.south) -- (pepflow.north -| \s.south);
}

\foreach \s in {s1,s2,s3,s4,s5}{
  \fill[pepamber] ($(\s.south)+(0,-.72)$) circle[radius=1.7pt];
  \draw[densely dotted, line width=.6pt, draw=pepamber!80!black]
    ($(\s.south)+(0,-.10)$) -- ($(\s.south)+(0,-.67)$);
}

\end{tikzpicture}
\endgroup%
    }
    \vspace{-10pt}
    \caption{A schematic description of \workflowname's 5-stage workflow.}
    \label{fig:pepflow-workflow}
    \vspace{-3pt}
\end{figure}
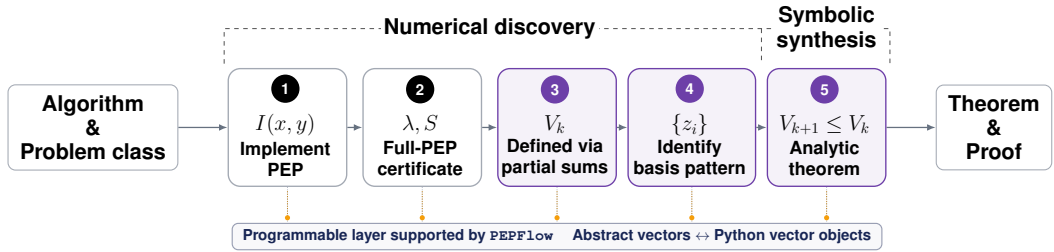
\footnotetext{
{As the main requirement for verification is establishing symbolic identities, we present them through Jupyter notebooks which contain the outputs based on \texttt{SymPy}, a symbolic calculation engine.
Nevertheless, we provide an auxiliary skill \texttt{\textbackslash{}lean-from-ipynb} for users who wish to proceed with Lean formalization.}
}


\section{\workflowname: Design principle with an example}
\label{section:workflow-theory}

The goal of \workflowname~is to produce a convergence theorem with its proof based on Lyapunov analysis.
Given a user input of the form described in \Cref{sec:workflow_step1}, a successful run produces a theorem and its proof as in \Cref{sec:workflow_step5}, with varying degrees of human intervention.
Figure~\ref{fig:pepflow-workflow} summarizes the five stages of \workflowname.
Below, we outline the theoretical design principle of each step of the workflow via a representative example of gradient descent (GD) for smooth convex minimization.

\subsection{PEP implementation of input algorithm} \label{sec:workflow_step1}

In the first stage of \workflowname, the user feeds in the problem class, algorithm's update rule, performance metric and initial condition into the coding agent.
Input format is flexible, as agents can handle various inputs; we provide a standardized example format in \Cref{sec:basic-usage}.
Once the agent understands all this information, it writes \pepflow~code for setting up the PEP construction, which can be numerically solved in the next stage.\footnote{The characterization of the function class by inequality \eqref{eq:convex_smooth_ineq} is due to \citet{taylor2017smooth}.}

\gdstepone

\subsection{Solving PEP and determining certificates}

Given the PEP implementation, the second stage actually builds and solves the PEP problem to obtain numerical certificates, which provide guidance toward the analytical convergence proof.
The PEP certificates obtained here can be understood as the numerical version of the following generic identity:
\begin{align*} \nonumber
    &\texttt{(Performance metric)} - \tau_N
    \times \texttt{(Initial condition)} \\ %
    &= \texttt{(Weighted sum of  $I(\cdot,\cdot)$'s)}
    - \texttt{(Sum of squares)} .
\end{align*}
Specifically for GD, we can write
\gdsteptwo

If the above identity is analytically determined, with forms of PEP certificates $\lambda_{i,i+1}, \lambda_{\star,i}, d_i \ge 0$ and $s_i$ clarified, then from \eqref{eq:convex_smooth_ineq} we conclude that $f(x_N) - f(x_\star) \le \frac{L}{4N+2} \| x_0 - x_\star \|^2$.
However, the PEP certificates are only numerical and searching for their analytical form traditionally required human effort \citep{taylor2017smooth, KimFessler2016_optimized}.
In \workflowname, this is often successfully carried out by the coding agent instead.

{
Because the AI agent works with the identity~\eqref{eq:gd_full_pep} for some fixed value of $N$, it has to be converted into a pattern that can be generalized to arbitrary $N$ for an analytic proof.
Our strategy is to employ the Lyapunov-style proof approach, outlined in the subsequent sections.
Notably, the analytic formulas for certificates $\lambda_{i,i+1}, \lambda_{\star,i}, d_i \ge 0$ and $s_i$ may or may not be easily inferred from the numerical values.
In our workflow, this is \emph{not a requirement}—one may proceed to the next stage using only numerical information, which is sufficient for inferring a right form of Lyapunov function to use.
Missing closed forms can be determined symbolically in the final stage.
}

\subsection{Defining Lyapunov function}

In this stage, we decompose the PEP proof~\eqref{eq:gd_full_pep} into partial sums $V_k$ for $k=1,\dots,N$, as:
\gdstepthree
By definition, due to \eqref{eq:convex_smooth_ineq}, \(V_k\) is a nonincreasing \textit{Lyapunov function}.
This construction and the subsequent steps of \workflowname\ follow \citet{YoonSuhNguyenYingMa2026_systematic}, which provided a systematic procedure for utilizing the Lyapunov function construction of the above form and deriving a concise and interpretable analytic proof from it.
Notably, \citet{GoujaudDieuleveutTaylor2023_fundamental} presented a similar idea of using a partial sum of nonpositive terms as a Lyapunov function, while \citet{YoonSuhNguyenYingMa2026_systematic} can be viewed as an extension and completion of the theory surrounding it.

\subsection{Finding a summation-free expression for \texorpdfstring{\(V_k\)}{V\_k}} \label{sec:workflow_block4}

In this stage, we systematically simplify $V_k$ into a summation-free form.
This is done via nontrivial linear algebraic techniques %
that we have implemented.
These core functions are packaged into the new \pepflowcolor{\texttt{lyapunov\_utils.py}} component, built upon the \pepflow library.
{\Cref{appendix:workflow_detail} explains the mathematical construction and its implementation.}

Utilizing them, we can extract the basis vectors that can express $V_k$ as a simple (low-rank) quadratic form.
This is the most novel component introduced by \workflowname, which enables the conversion of~\eqref{eq:gd_full_pep} into an analytical, $N$-independent proof.
Applying this step to GD, in particular, we obtain:
\gdstepfour

Note that this stage, on its own, used to be a research-level task requiring human experts, as done in prior work \citep{ParkParkRyu2023_factorsqrt2,dAspremontScieurTaylor2021_acceleration,LPR21}.

\subsection{Analytic proof by Lyapunov analysis} \label{sec:workflow_step5}
In this final step, we verify $V_{k+1} \le V_k$ through symbolic calculation.
In a successful use case, we let \workflowname\ additionally produce LaTeX code for the corresponding proof and the resulting theorem statement, as shown below.
Examples of proofs generated by \workflowname\ are provided in \Cref{sec:experiments}.
\begin{examplebox}
\gdstepfivetheorem
\end{examplebox}


\section{\workflowname: Agentic implementation}
\label{section:workflow-implementation}

\workflowname\ is implemented as a sequence of reusable commands, each written in a Markdown file.
Here, a \emph{command} is a prompt specification for executing each modular block of the theoretical workflow described in \Cref{section:workflow-theory}.

The user feeds \workflowname\ with a prompt including the command name and the required input arguments.
Below, we briefly outline the function of each command at a high level, deferring detailed explanation to \Cref{appendix:implementation_detail}.
See \Cref{sec:basic-usage} for a concrete prompt example.

{
\begin{enumerate}[label=(\arabic*)]
    \item \textbf{\SkillImplement{}:}
    Turns the %
    algorithm description into an executable \pepflow\ setup module.

    \item \textbf{\SkillFullProof{}:}
    Solves PEP with a fixed \(N\) and extracts proof certificates, e.g., \(\lambda\), \(d\), and \(s\) in \eqref{eq:gd_full_pep}, and guesses their closed-form expressions.

    \item \textbf{\SkillLyapDefine{}:}
    Constructs a candidate Lyapunov function \(V_k\) as partial sums of weighted inequalities and square terms from the PEP proof.

    \item \textbf{\SkillLyapVectors{}:}
    Identifies the vectors that are used for expressing \(V_k\), and thereby simplifies it into a summation-free expression.

    \item \textbf{\SkillLyapClosedForm{}:}
    Finalizes the analytic coefficients in $V_k$, symbolically verifies the Lyapunov analysis, and drafts a theorem with a proof outline. %
\end{enumerate}}
The role of AI assistance in \workflowname\ is two-fold:
\begin{itemize}
    \item
    AI generates \pepflow-based codes for structured tasks such as solving PEPs, constructing Lyapunov groupings, and searching for basis vectors.

    \item
    AI assists in automating tasks that require mathematical insight, including selecting sparsity patterns, guessing closed-form rates and symbolic formulas, and analytically identifying \(V_k\). Here, to prevent proceeding with hallucinated outputs, \workflowname\ performs specific numerical or symbolic verification steps on the AI agent's outputs.
\end{itemize}

At the end of each block, the outputs are saved into a \texttt{.json} state file, and a human-readable summary is written into a Jupyter notebook,
which enables both \workflowname\ and the user to keep track of intermediate results.
The user may inspect the progress and provide additional guidance to \workflowname\ when needed.

\subsection{How to use}
\label{sec:how_to_use}

We describe direct use of the five-stage workflow.

\label{sec:basic-usage}

The user specifies an algorithm, a problem class, an initial condition,
and a performance metric, then runs the five commands in sequence.

While the input prompts can be flexible, we provide a standardized example format that we found successful, for the case of GD for smooth convex minimization.
We start by running the first command, \SkillImplement{}, together with the required arguments.

\begin{commandbox}
/pep-implement

Function: f is convex and L-smooth
Parameters: L, R
Initial condition: ||x_0 - x_star|| <= R, where grad f(x_star) = 0
Performance metric: f(x_N) - f(x_star)
Algorithm: Gradient descent with fixed step size 1/L

For k >= 0:
x_{k+1} = x_k - (1/L) * grad f(x_k)

Conjectured rate: unknown
\end{commandbox}

Alternatively, one can choose to input a screenshot of the algorithm's information, e.g., the box in \Cref{sec:workflow_step1}.

In successful cases, \SkillImplement{} produces an output that can be used as an input to the next block, \SkillFullProof{}.
This can be triggered by prompting ``command name + algorithm name'', e.g., the prompt ``\texttt{/pep-full-proof gd}''
will let \workflowname\ proceed to the next stage.
{A simpler alternative for interactive use is: ``\texttt{proceed with the next step}''.}
This is repeated until the final block is reached.

{%
\section{Evaluations on Known Results}
\label{sec:experiments}

\begin{table}[!b]
\centering
{%
\caption{Evaluation results. Block times (minutes) and generation tokens (millions) are means $\pm$ one sample standard deviation over five attempts.
}
\label{tab:experiment-costs}
\small
\setlength{\tabcolsep}{3pt}
\begin{tabular}{@{}lccccccc@{}}
\toprule
Condition & Verified & Block 1 & Block 2 & Block 3 & Block 4 & Block 5 & Tokens (M) \\
\midrule
OGM & 5/5 & $7.7 \pm 0.8$ & $11.3 \pm 1.8$ & $6.8 \pm 0.8$ & $7.8 \pm 1.6$ & $12.5 \pm 1.7$ & $11.5 \pm 2.4$ \\
BPPM & 5/5 & $7.4 \pm 1.1$ & $10.7 \pm 2.6$ & $8.1 \pm 1.3$ & $6.8 \pm 1.2$ & $9.5 \pm 2.1$ & $9.9 \pm 1.3$ \\
FEG & 5/5 & $7.4 \pm 0.8$ & $11.1 \pm 1.5$ & $7.2 \pm 1.0$ & $5.6 \pm 0.4$ & $9.2 \pm 1.4$ & $10.7 \pm 1.1$ \\
Dual-FEG & 5/5 & $7.6 \pm 0.3$ & $13.1 \pm 2.8$ & $7.4 \pm 0.7$ & $7.4 \pm 1.7$ & $11.2 \pm 0.5$ & $12.5 \pm 1.7$ \\
PGM (no SOS) & 5/5 & $7.1 \pm 0.9$ & $10.5 \pm 1.5$ & $11.3 \pm 4.2$ & $8.0 \pm 1.4$ & $23.5 \pm 4.7$ & $15.7 \pm 2.2$ \\
PGM (given SOS) & 5/5 & $6.9 \pm 1.0$ & $12.7 \pm 0.7$ & $13.6 \pm 2.3$ & $10.2 \pm 1.1$ & $14.6 \pm 3.7$ & $15.4 \pm 1.0$ \\
\bottomrule
\end{tabular}
}
\end{table}

We provide a simple evaluation of \workflowname's ability to recover known convergence rates through
analytic Lyapunov proofs for PGM, OGM \citep{KimFessler2016_optimized}, BPPM, FEG \citep{LeeKim2021_fast}, and Dual-FEG \citep{YoonKimSuhRyu2024_optimal}.
PGM is tested with and without a supplied sum-of-squares (SOS) decomposition, yielding six experiment settings with five attempts each.
Proof generation and evaluation both use Codex with GPT-6~Astra, medium reasoning.

We design a controlled clean-room experiment protocol, where the AI agent is instructed not to ``cheat'' by consulting external material or using outputs of other runs,
and to strictly adhere to a specified Python environment.
The purpose of the experiments is to demonstrate the capability and consistency of \workflowname\ in simulating an optimization-specific workflow that has been performed by domain experts.
All 30 trials generated verifiable proofs establishing prescribed convergence bounds, streamlined through Jupyter Notebooks.\footnote{
A clean-room protocol issue was flagged in one trial each for BPPM, FEG, and PGM (without SOS),
during final notebook execution within Block 5 \emph{after} a Jupyter notebook containing a correct proof had been generated.
The incidents involved generic Python/Jupyter dependencies outside the designated environment.
No prohibited mathematical exposure was identified, and subsequent verification confirmed the submitted proofs.
}
We report the time and generation tokens consumed in \cref{tab:experiment-costs}.

Notably, while these runs consider algorithms whose tight rates are already known for the purpose of evaluating the workflow,
the \emph{Lyapunov proof} has not appeared in the literature for the case of PGM, which we present below.

\subsection{Proximal gradient descent method (PGM)}\label{sec:pgm-five-run-example}
Let $f$ be convex and $L$-smooth, let $g$ be closed proper convex, and let
$x_\star$ minimize $h=f+g$. Consider
\[
x_{k+1}=\operatorname{prox}_{g/L}\left(x_k-L^{-1}\nabla f(x_k)\right).
\]
Both PGM conditions recover
$h(x_N)-h(x_\star)\le\frac{L}{4N}\|x_0-x_\star\|^2$
for every $N\ge1$ \citep{TeboulleVaisbourd2023_elementary}.

To illustrate the mathematical output, the first
attempt without supplied SOS yields, for $1\le k<N$,
\begin{examplebox}
{%
\begin{equation*}
\begin{aligned}
V_k&=
\frac{k+1}{2N-k}\left(f(x_k)-f(x_\star)\right)
+\frac{k}{2N-k}\left(g(x_k)-g(x_\star)\right)
-\frac{1}{2N-k}\langle x_k-x_\star,\nabla f(x_k)\rangle \\
&\quad-\frac{L}{4N}\|x_0-x_\star\|^2
+\frac{L(N-k)}{(2N-k)^2}\|x_k-x_\star\|^2
-\frac{k}{2L(2N-k)^2}\|\nabla f(x_k)-\nabla f(x_\star)\|^2.
\end{aligned}
\end{equation*}
}
\end{examplebox}
With separately defined endpoints, the proof establishes
$V_N\le\cdots\le V_0=0$ and the desired bound. The first supplied-SOS attempt
has a different gradient-square coefficient; both constructions
are reproduced in \Cref{appendix:pgm_output}.

}


{\section{Case Studies on Open Problems}}
\label{sec:case-study}

In this section, we show that \workflowname\ can obtain novel convergence analyses for settings in which analytic analyses had not previously been developed, including some problems that had remained open for about a decade.

\textbf{We close Conjecture~4 and Conjecture~5 of \citet{taylor2017smooth}} by providing the first analytic proofs of tight convergence bounds for Nesterov's Fast Gradient Method (FGM) \citep{Nesterov1983_method}.
FGM is a classical accelerated method dating back to 1983, yet its tight worst-case bounds had not been established analytically. These bounds were conjectured in \citet{taylor2017smooth}, based on numerical values obtained by PEP roughly a decade ago, and had remained open since then. We resolve these conjectures by providing Lyapunov analyses obtained by \workflowname. Additionally, we close related conjectures for OGM raised in \citet[Conjecture~4]{taylor2017smooth} and for FGM and FISTA with rational coefficients raised in \citet[Table~1]{TaylorHendrickxGlineur2017_exacta}.

For the conjectures considered in this section, we used the same input format as in Section~\ref{sec:how_to_use}, setting only the conjectured-rate field to the rate stated in the corresponding conjecture. We used \texttt{GPT-6 Astra} with \texttt{Ultra} reasoning effort.

{
As the proofs grew more complex, the
Codex-generated drafts required human editing to make them readable.
In particular, the drafts introduced many unnecessary substitutions;
we had to simplify these and choose variable names and notation that
made the arguments easier to follow. Detailed proofs are provided in
Appendix~\ref{app:case-study-proofs}.}

\subsection{Tight convergence analysis of FGM using \workflowname: %
Closing Conjectures~4 and~5 of \texorpdfstring{\citet{taylor2017smooth}}{Taylor et al. (2017)}}
\label{sec:accelerated-case-study}

For a parameter sequence \(\theta_k\) and a scalar \(\eta\), FGM, OGM,
and FGM with rational coefficients start from \(x_0=y_0=z_0\) and use, for \(0\leq k<N\),
the following updates:
\begin{equation}
\begin{aligned}
 y_{k+1}&=x_k-\frac1L\nabla f(x_k),\\
 z_{k+1}&=z_k-\frac{\eta\theta_k}{L}\nabla f(x_k),\\
 x_{k+1}&=\left(1-\frac1{\theta_{k+1}}\right)y_{k+1}
          +\frac1{\theta_{k+1}}z_{k+1}.
\end{aligned}
 \label{eq:case-ofgm}
\end{equation}

We first consider FGM, which selects
\begin{align}
 \theta_0&=1,\qquad
 \theta_{k+1}=\frac{1+\sqrt{1+4\theta_k^2}}2
 \quad(k\geq0),\qquad \eta=1.
 \tag{FGM}
 \label{eq:case-theta-recursive}
\end{align}
\begin{theorem}[{FGM at \(y_N\); Conjecture~4 in \citet{taylor2017smooth}}]
\label{thm:case-ofgm-primary}
Let \(f\in\mathcal F_L\).
Fix an integer \(N\geq1\) and \(R\geq0\), and assume
\(\norm{x_0-x_\star}\leq R\).
Then \ref{eq:case-theta-recursive} satisfies
\begin{equation}
 f(y_N)-f(x_\star)\leq
 \frac{LR^2\theta_{N-1}^2}{2\left(2\sum_{j=0}^{N-1}\theta_j^3+\theta_{N-1}^2\right)}.
 \label{eq:case-ofgm-primary-rate}
\end{equation}
\end{theorem}

\begin{proof}[Proof outline]
Set \(D_j=2C_N+1-C_j\) and \(\tau_N=1/(2D_0)\). For the
interior iterates, set
\begin{equation*}
 \boldsymbol w_k=
 \begin{bmatrix}z_{k+1}-x_\star\\y_{k+1}-z_{k+1}\end{bmatrix}
\end{equation*}
and use \(\vQ_k\), whose diagonal entries are rational
expressions in \(a_k,C_{k+1},D_0,\theta_k\), and
\(\theta_{N-1}\). Define \(V_0=0\),
\begin{equation*}
 V_k=a_k\bigl(f(x_k)-f(x_\star)\bigr)
 -\tau_NL\norm{x_0-x_\star}^2
 -\frac{a_k}{2L}\norm{\nabla f(x_k)}^2
 +L\boldsymbol w_k^\top \vQ_k\boldsymbol w_k
\end{equation*}
for the interior iterates, and
\begin{equation*}
 V_N=f(y_N)-f(x_\star)-\tau_NL\norm{x_0-x_\star}^2.
\end{equation*}
Each decrement is an exact positive weighted sum of smooth convex
interpolation residuals minus one or two nonnegative squares. The explicit
coefficients and their positivity are verified in
Appendix~\ref{app:case-ofgm-primary-proof}. Hence
\(V_{k+1}\leq V_k\), so \(V_N\leq V_0=0\), which gives
\eqref{eq:case-ofgm-primary-rate}.

\end{proof}

\begin{theorem}[{FGM at \(x_N\); Conjecture~5 in \citet{taylor2017smooth}}]
\label{thm:case-ofgm-secondary}
Let \(f\in\mathcal F_L\).
Fix an integer \(N\geq1\) and \(R\geq0\), and assume
\(\norm{x_0-x_\star}\leq R\).
Then \ref{eq:case-theta-recursive} satisfies
\begin{equation}
 f(x_N)-f(x_\star)\leq
 \frac{LR^2\theta_N^2}{2\left(2\sum_{j=0}^{N}\theta_j^3-\theta_N^2\right)}.
 \label{eq:case-ofgm-secondary-rate}
\end{equation}
\end{theorem}

\begin{proof}
See Appendix~\ref{app:case-ofgm-secondary-proof} for the proof.
\end{proof}

{The bounds in Theorems~\ref{thm:case-ofgm-primary}
and~\ref{thm:case-ofgm-secondary} are tight, matching the lower bounds
established in Lemma~\ref{lem:case-smooth-lower} using the examples of
\citet[Section~4.2, Conjectures~4--5]{taylor2017smooth}.
Lemma~\ref{lem:case-fgm-conjecture-relation} shows that the bounds proved by Theorems~\ref{thm:case-ofgm-primary} and \ref{thm:case-ofgm-secondary} indeed correspond to their conjectured rates.
}

\subsubsection{OGM at \texorpdfstring{\(y_N\)}{y\_N}}
\label{sec:ogm-yN}

For the bound at \(y_N\), OGM uses the same recursive
parameters as FGM through \(\theta_{N-1}\) and doubles the coefficient
in the \(z\)-update:
\begin{align}
 \theta_0&=1,\qquad
 \theta_{k+1}=\frac{1+\sqrt{1+4\theta_k^2}}2
 \quad(0\leq k<N-1),\qquad \eta=2.
 \tag{OGM}
 \label{eq:case-ogm-parameters}
\end{align}

\begin{theorem}[{OGM at \(y_N\); Conjecture~4 in \citet{taylor2017smooth}}]
\label{thm:case-ogm}
Let \(f\in\mathcal F_L\).
Fix an integer \(N\geq1\) and \(R\geq0\), and assume
\(\norm{x_0-x_\star}\leq R\).
Then \ref{eq:case-ogm-parameters} satisfies
\begin{equation}
 f(y_N)-f(x_\star)\leq\frac{LR^2}{2(2\theta_{N-1}^2+1)}.
 \label{eq:case-ogm-rate}
\end{equation}
\end{theorem}

\begin{proof}
See Appendix~\ref{app:case-ogm-proof} for the proof.
\end{proof}

The bound in Theorem~\ref{thm:case-ogm} matches the lower bound in
\citet[Proposition~4.1]{KimFessler2017_convergence} and is therefore tight.

\subsection{FGM and \texorpdfstring{FISTA}{FISTA} with rational coefficients}
\label{sec:rational-coefficients-case-study}

We now consider the conjectures in
\citet[Table~1]{TaylorHendrickxGlineur2017_exacta}
for FGM and FISTA with rational coefficients.
For FGM, the common update \eqref{eq:case-ofgm} uses the parameters
\begin{align}
 \theta_k&=\frac{k+2}{2}\quad(k\geq0),\qquad \eta=1.
 \tag{FGM-rational}
 \label{eq:case-theta-linear}
\end{align}
It is worth noting that
\citet[Section~4.2]{ThomsenUpadhyayaGoujaudDieuleveutTaylor2026}
investigate sparsity patterns in numerical PEP proof certificates for this FGM bound.
Theorem~\ref{thm:case-sfgm} establishes this conjectured rate.

\begin{theorem}[{FGM with rational coefficients; conjecture in Table~1 of \citet{TaylorHendrickxGlineur2017_exacta}}]
\label{thm:case-sfgm}
Let \(f\in\mathcal F_L\).
Fix an integer \(N\geq1\) and \(R\geq0\), and assume
\(\norm{x_0-x_\star}\leq R\).
Then \ref{eq:case-theta-linear} satisfies
\begin{equation}
 f(y_N)-f(x_\star)\leq\frac{2LR^2}{(N+2)(N+3)}.
 \label{eq:case-sfgm-rate}
\end{equation}
\end{theorem}

\begin{proof}
See Appendix~\ref{app:case-sfgm-proof} for the proof.
\end{proof}

{Combining this theorem with Lemma~\ref{lem:case-smooth-lower}
shows that \eqref{eq:case-sfgm-rate} is attained and therefore sharp.}

We next consider FISTA with rational coefficients for a composite objective
\(F=f+g\). Starting from \(x_0=y_0\), the method uses, for \(0\leq k<N\),
\begin{equation}
\begin{aligned}
 y_{k+1}&=\operatorname{prox}_{g/L}
   \left(x_k-\frac1L\nabla f(x_k)\right),\\
 x_{k+1}&=y_{k+1}+\frac{k}{k+3}(y_{k+1}-y_k),
\end{aligned}
\tag{FISTA-rational}
\label{eq:case-fista}
\end{equation}
We report \(y_N\). This is the FPGM1 variant in
\citet[Table~1]{TaylorHendrickxGlineur2017_exacta}; the following theorem
resolves its nonsmooth composite conjecture.

\begin{theorem}[{FISTA with rational coefficients; nonsmooth FPGM1 conjecture
in Table~1 of \citet{TaylorHendrickxGlineur2017_exacta}}]
\label{thm:case-fista}
Let \(f\in\mathcal F_L\), and let
\(g:\mathcal H\to\mathbb R\cup\{+\infty\}\) be proper, closed, and convex.
Set \(F=f+g\). Fix an integer \(N\geq1\) and \(R\geq0\), and suppose
\(0\in\nabla f(x_\star)+\partial g(x_\star)\) and
\(\norm{x_0-x_\star}\leq R\).
Then \ref{eq:case-fista} satisfies
\begin{equation*}
 F(y_N)-F(x_\star)\leq\frac{2LR^2}{N^2+5N+2}.
\end{equation*}
\end{theorem}

\begin{proof}
See Appendix~\ref{app:case-fista-proof} for the proof.
\end{proof}

{The bound in Theorem~\ref{thm:case-fista} is tight, matching the
lower bound established in Lemma~\ref{lem:case-fista-lower} using the
constrained linear example of
\citet[Section~4.2.2, Table~1 and the following paragraph]{TaylorHendrickxGlineur2017_exacta}.}


\section{Conclusion}
\label{sec:conclusion}

{
We present \workflowname, an AI-assisted workflow for first-order optimization, built on the domain-specific machinery of PEP and the systematic framework of \citet{YoonSuhNguyenYingMa2026_systematic} for converting numerical PEP certificates into analytic Lyapunov analyses.
\workflowname\ equips AI agents with the \pepflow library, which abstracts the domain-specific context and language into an accessible interface.
Our particular Lyapunov proof structures support simple and rapid verification of the proposed proofs. %
This enables users to efficiently search for verified proofs with \workflowname\ using only a high-level understanding of the workflow, as demonstrated by our resolution of conjectures on Nesterov's FGM and other methods.
Future work could extend \workflowname\ to support a broader range of proof structures and thereby cover more algorithms and forms of convergence analysis.
We believe that \workflowname\ is an important step toward exploiting AI for making systematic and human-understandable advances in optimization.
}

\IfFileExists{iclr2027/peppy_refs.bib}{%
  \bibliography{iclr2027/peppy_refs,iclr2027/peppy}%
}{%
  \bibliography{peppy_refs,peppy}%
}
\IfFileExists{iclr2027/peppy.bst}{%
  \bibliographystyle{iclr2027/peppy}%
}{%
  \bibliographystyle{peppy}%
}

\clearpage
\appendix

\section{Notation}
\label{app:notation}

We write \(\mathcal H\) for a nonzero real Hilbert space, equipped with
inner product \(\langle\cdot,\cdot\rangle\) and its induced norm
\(\norm{\cdot}\).
For \(L>0\), \(\mathcal F_L\) denotes the class of differentiable convex
functions \(f:\mathcal H\to\mathbb R\) that are \(L\)-smooth, i.e.,
\(\norm{\nabla f(x)-\nabla f(y)}\leq L\norm{x-y}\) for all
\(x,y\in\mathcal H\).
For a proper, closed, convex function
\(g:\mathcal H\to\mathbb R\cup\{+\infty\}\) and \(\gamma>0\), we write
\[
\operatorname{prox}_{\gamma g}(v)
:=\operatorname*{arg\,min}_{w\in\mathcal H}
\left\{\gamma g(w)+\frac12\norm{w-v}^2\right\}.
\]

{For positive integers \(m,n\), \(\mathbb R^{m\times n}\) denotes the space of real \(m\times n\) matrices. We write \(\mathbb S^n\) and \(\mathbb S_+^n\) for the spaces of real symmetric and real symmetric positive semidefinite \(n\times n\) matrices, respectively. For a matrix \(A\), \(A^\transpose\), \(\operatorname{Tr}(A)\), \(\operatorname{rank}(A)\), and \(\operatorname{col}(A)\) denote its transpose, trace, rank, and column space. The \(i\)-th standard basis vector in a Euclidean space of dimension at least \(i\) is denoted by \(\ve_i\).}

For minimization problems, we assume that the objective attains its
minimum and denote a minimizer by \(x_\star\).
{We use \(\star\) as an index for an optimal point and reserve \(N\) for the terminal iteration, or horizon, at which the performance measure is evaluated.}
For endpoint coefficient formulas, we adopt the convention \(0/0=0\).


\section{Prior Work}
\label{app:prior-work}

\subsection{Performance estimation problems}
\label{app:prior-work-pep}

The performance estimation problem (PEP) framework formulates the worst-case
performance of a prescribed optimization algorithm as an optimization problem
over admissible functions and iterates. \citet{DT14} introduced this approach,
and \citet{taylor2017smooth} developed exact smooth strongly convex interpolation conditions
that yield finite-dimensional semidefinite formulations.
Extensions cover composite convex optimization
\citep{TaylorHendrickxGlineur2017_exacta}, Bregman geometry
\citep{dragomir2022optimal}, operator splitting
\citep{RyuTaylorBergelingGiselsson2020_operator}, and methods involving linear
operators \citep{BousselmiHendrickxGlineur2024_interpolation}.

Dual feasible solutions supply weights on interpolation inequalities and
positive semidefinite matrix certificates that together certify a performance
bound. Numerical solutions thus provide concrete guidance for analytic proof
construction. Turning these certificates into formulas valid for arbitrary
iteration horizons remains a separate mathematical task. This is the step
that \workflowname\ addresses through a sequence of numerical, structural,
and symbolic computations.

\subsection{Discoveries enabled by PEP}
\label{app:prior-work-pep-discoveries}

PEP has supported both sharper analyses of existing methods and the design of
new methods. Early examples include tight bounds for constant-step GD
\citep{DT14} and the optimized gradient method (OGM)
\citep{KimFessler2016_optimized},
{whose worst-case objective bound for the horizon-dependent
secondary output \(x_N\) was later matched by an oracle lower bound
\citep{Drori2017_exact}.}
Subsequent developments include OGM-G for reducing the gradient norm
\citep{KimFessler2021_optimizing}, the information-theoretic exact method
(ITEM) for smooth strongly convex minimization \citep{taylor2023optimal},
and OptISTA for composite optimization
\citep{JangGuptaRyu2025_computerassisted}.
The branch-and-bound PEP framework \citep{DVR24} extends this design process
to nonconvex quadratically constrained quadratic formulations in which the
algorithm coefficients themselves are optimization variables.

Beyond minimization, this line of work includes the accelerated proximal point
method \citep{Kim2021_accelerated}, tight Halpern iteration bounds
\citep{Lieder2021_convergence,ParkRyu2022_exact}, and accelerated methods for
minimax problems such as EAG and FEG
\citep{YoonRyu2021_accelerated,LeeKim2021_fast}.
Numerical relations between algorithm coefficients have also motivated H-duality,
connecting function-value and gradient-norm guarantees
\citep{KimOzdaglarParkRyu2023_timereversed}, with extensions to minimax and
fixed-point problems \citep{YoonKimSuhRyu2024_optimal}.
These results illustrate the mathematical insight required to interpret PEP
outputs, and motivate the structured proof-search tasks used in \workflowname.
\par

\subsubsection{Tight analysis of fixed algorithms}
Performance estimation has also shown that obtaining a tight rate for a
prescribed algorithm can be difficult even when that algorithm is not minimax
optimal.  For gradient descent (GD), \citet{DT14} derived a tight analytical
bound in the smooth convex setting, and \citet{taylor2017smooth} subsequently used exact
smooth strongly convex interpolation to formulate finite-dimensional PEPs and
numerically conjecture worst-case rates under several performance criteria.
It was only recently that \citet{kim2025proofexactconvergencerate} proved the exact
function-value convergence rate of GD for smooth strongly convex optimization,
settling conjectures posed in the aforementioned works.  In parallel,
\citet{rotaru2024exactworstcaseconvergencerates} obtained a complete exact
analysis of the minimum gradient norm for all constant stepsizes across convex,
strongly convex, and weakly convex classes. Their proof requires inequalities
coupling nonconsecutive iterates.  This progression from numerical evidence to
analytic proofs for GD itself illustrates that the tight analysis of a fixed,
suboptimal method is a nontrivial problem. Conventional convergence rate analysis does not in
general give the sharp finite-horizon constant or a matching worst-case
instance.
\par

{%
\subsubsection{Recent evidence and exact analyses for accelerated methods}
In their study of implicit primal--dual guarantees,
\citet[Section~6.2]{GrimmerWang2026_implicit} note that an exactly optimal PEP
certificate for the Nesterov FGM variant they analyze had not been determined
analytically, and they instead use a feasible, nearly optimal certificate.
Using chained, iteration-dependent quadratic Lyapunov functions over prescribed
state and oracle-evaluation vectors,
\citet[Section~3.2 and Figure~3]{UpadhyayaTaylorBanertGiselsson2025_autolyap}
obtain numerical finite-horizon constants for Nesterov's FGM that are smaller
than the classical rates in their comparison.
Identifying Grimmer and Wang's \(t_n\) with our \(\theta_n\), their initialization of
\(z_1\) is the first \(z\)-update in \eqref{eq:case-ofgm}, and their subsequent
recurrences agree with \eqref{eq:case-ofgm} and
\eqref{eq:case-theta-recursive}. Their criterion
\(f(x_N)-f(x_\star)\), under \(\norm{x_0-x_\star}\leq D\), therefore
corresponds to the FGM output \(x_N\) in Conjecture~5 of
\citet{taylor2017smooth}, rather than the output \(y_N\) in Conjecture~4.
Theorem~\ref{thm:case-ofgm-secondary} gives an analytic Lyapunov proof of the
corresponding tight finite-horizon bound for every \(N\).
\citet{ThomsenUpadhyayaGoujaudDieuleveutTaylor2026} investigate the
sparsification of PEP multipliers to obtain simpler proof certificates.
\par
}

{%
For the primary OGM output \(y_N\),
\citet[Theorem~4.1]{KimFessler2017_convergence} prove the upper bound
\(LR^2/(4\theta_{N-1}^2)\), while their Proposition~4.1 constructs an
instance attaining \(LR^2/[2(2\theta_{N-1}^2+1)]\). These expressions agree
asymptotically but leave a finite-horizon gap. Theorem~\ref{thm:case-ogm}
proves the latter expression as an upper bound for the same output \(y_N\),
closing the finite-horizon gap.
This result is separate from the already exact guarantee
\(LR^2/(2\theta_N^2)\) for OGM's secondary last iterate \(x_N\), which uses
the horizon-dependent terminal parameter
\(\theta_N=(1+\sqrt{1+8\theta_{N-1}^2})/2\) and is attained in
\citet[Theorem~5.1]{KimFessler2017_convergence}.
\par
}

\subsection{Computer-assisted Lyapunov analysis}
\label{app:prior-work-lyapunov}

A Lyapunov proof expresses progress through a scalar quantity \(V_k\) whose
decrease implies the desired convergence guarantee. Identifying an appropriate
form of \(V_k\) is often the central difficulty.
Integral quadratic constraints \citep{LessardRechtPackard2016_analysis} and
semidefinite searches over potential functions
\citep{TaylorVanScoyLessard2018_lyapunov,TaylorBach2019_stochastic}
provide systematic approaches to this task.
\citet{UpadhyayaBanertTaylorGiselsson2024_automated} give necessary and
sufficient conditions for a Lyapunov inequality to exist within a prescribed
family. These approaches allow automated searches once the candidate state
vectors and functional form have been chosen.
\citet{ThomsenUpadhyayaGoujaudDieuleveutTaylor2026} use a candidate-lemma SDP to identify Lyapunov potentials and their one-step decrements for proximal point and accelerated proximal point methods. %

A complementary approach starts from a PEP certificate and seeks a concise
Lyapunov representation of it.
The history of accelerated methods illustrates
this need: matrix-based certificates for methods such as OGM and OGM-G were
subsequently developed into potential-function or geometric arguments
\citep{dAspremontScieurTaylor2021_acceleration,LPR21,ParkParkRyu2023_factorsqrt2}.
\citet[Section~6]{GoujaudDieuleveutTaylor2023_fundamental} discuss forming
potentials from partial sums of the terms in a PEP proof.
Building on this perspective,
\citet{YoonSuhNguyenYingMa2026_systematic} develop a systematic characterization
and procedure for converting PEP proofs into Lyapunov analyses while preserving
their certificates, including the search for a concise representation in
suitable basis vectors.

That conversion framework provides the theoretical basis for \workflowname.
The present paper organizes its application into tasks that an AI agent can
execute: implementing the PEP, inferring analytic certificate formulas,
constructing and simplifying \(V_k\), and checking its difference and boundary
conditions. The resulting workflow complements searches over predefined
Lyapunov families by using the structure of the PEP certificate to guide
the choice of representation.
\par

\subsection{Software for performance estimation}
\label{app:prior-work-software}

Software packages make these computational approaches accessible.
\texttt{PESTO} \citep{TaylorHendrickxGlineur2017_performance} provides a MATLAB
toolbox for PEP, while \texttt{PEPit} \citep{pepit2024} offers an open-source
Python interface for modeling and solving such problems.
\texttt{AutoLyap} \citep{UpadhyayaTaylorBanertGiselsson2025_autolyap} supports
computer-assisted searches over Lyapunov templates.
Its fixed-horizon numerical certificates do not give explicit formulas for the
certificate coefficients or a single analytic proof valid for every horizon. Numerical
agreement with a tight FGM rate sequence does not by itself identify the same
algorithm, reported iterate, or Lyapunov function as in the conjectures of
\citet{taylor2017smooth}, and therefore does not settle them. The approaches
are complementary: \texttt{AutoLyap} searches over a prescribed Lyapunov
family in a general-purpose state and evaluation basis, while \workflowname\
starts from a tight PEP certificate and seeks a compressed basis adapted to
the algorithm, summation-free formulas, and proofs of the required identities
and coefficient signs valid for every horizon.

\workflowname\ uses \pepflow\ \citep{SuhYingJiangNguyen2025_pepflow} as its
computational backend. Its tagging of iterates, function values, and
gradient or operator evaluations makes the resulting vectors and matrices
accessible in mathematical notation.
This is useful when manipulating numerical certificates and recovering
analytic expressions. The certificate-processing routines in
\texttt{lyapunov\_utils.py}, discussed in
Section~\ref{appendix:lyapunov_utils}, support the conversion steps.
\workflowname\ adds reusable agent instructions, intermediate state records,
and numerical and symbolic checks around these operations.
\par

\subsection{AI for mathematics}
\label{app:prior-work-ai-math}

AI-assisted mathematical research includes discovering conjectures, searching
for explicit constructions, and developing proofs.
\citet{Davies2021_advancing} combine machine learning with attribution methods
to identify relationships among mathematical objects and guide mathematicians
toward conjectures and theorems in topology and representation theory.
FunSearch \citep{RomeraParedes2024_funsearch} combines a language model with
program evaluation and evolutionary search, producing constructions for the
cap-set problem and heuristics for online bin packing.
\citet{feng2026autonomousmathematicsresearch} introduce Aletheia, a mathematical research agent that iteratively generates, verifies, and revises solutions and is capable of solving research-level problems in mathematics.
\citet{tsoukalas2026advancingmathematicsresearchaidriven} use agents to leverage LLMs to generate formal proofs in languages like Lean to successfully solve Erd{\char174}s problems.

A related direction couples learned search with symbolic reasoning.
AlphaGeometry \citep{Trinh2024_alphageometry} uses a language model to propose
auxiliary geometric constructions and a symbolic engine to derive their
consequences.
LeanDojo \citep{Yang2023_leandojo} provides tools for interacting with Lean
and develops a retrieval-augmented prover that selects relevant premises
from a mathematical library.
The feedback in these systems takes different forms, including numerical
evaluation, symbolic deduction, and proof-assistant checking, depending on
the task and its representation.

\workflowname\ applies this general approach to convergence analysis.
Its search is organized around numerical PEP certificates and their
conversion into analytic Lyapunov proofs.
The agent seeks formulas valid beyond the finite horizons used in numerical
experiments and checks the resulting identities symbolically, together with
the inequalities and boundary arguments required for the claimed rate.
The central proof objects are a Lyapunov function, its one-step difference,
and the endpoint estimates that together yield the convergence analysis.
\par

\citet{kim2026domainspecificharnessendtoendautomation} integrate the PEP
framework with LLMs. Their approach differs from \workflowname\ in its level
of granularity. Given a user-specified problem class, their 
domain-specific harness \texttt{AutoOpt} uses the branch-and-bound performance estimation
programming (BnB-PEP) framework to automatically construct a nonconvex
quadratically constrained quadratic program (QCQP). Solving this QCQP provides
numerical guidance about an optimal algorithm for the problem class and its
convergence-rate proof. The harness uses high-level prompts to an LLM to carry
out the typically human step of converting these numerical outputs into an
analytic result. The resulting convergence analysis is then formally verified
in Lean and left for humans to interpret and verify.

\subsection{AI for Optimization Theory}
\label{app:prior-work-ai-optimization}

There have been various works that have leveraged LLMs to produce research-level results in optimization, specifically first-order methods. \citet{jang2026pointconvergencenesterovsaccelerated} leverage GPT-5 Pro to assist in proving the point convergence of Nesterov's accelerated gradient method. {\citet{Ma2026MatrixScaling} uses a proof generated by ChatGPT 5.5 and verified by the author to establish convergence of BDRS under a verifiable condition when it is viewed as a matrix scaling algorithm.} \citet{ma2026lowerboundstepsizebasedacceleration} improve the lower bound for the last-iterate convergence rate of gradient descent with predetermined nonnegative stepsize schedules to show that stepsize schedules alone cannot accelerate GD to the convergence rate achieved by Nesterov's accelerated method through using GPT-5.6 Sol Pro. \citet{ye2026improvedgradientdescentlower} further improve the lower bound established by \citet{ma2026lowerboundstepsizebasedacceleration} for GD with assistance from GPT-5.6 Sol Pro. The aforementioned works are a non-exhaustive representation of works in convergence analysis of first-order methods that employ LLMs to achieve state-of-the-art results.


\section{Omitted details of Section~\ref{sec:workflow_block4}} \label{appendix:workflow_detail}

{
Here, we briefly explain Step~4, which is implemented under the hood in \workflowname, using GD as a representative example. For a more detailed explanation, we refer the readers to \cite{YoonSuhNguyenYingMa2026_systematic}.
}

\subsection{{Mathematical construction: the GD example}}
\label{appendix:workflow_detail_gd}

The PEP proof for gradient descent (GD) $x_{k+1} = x_k - \frac{1}{L} \nabla f(x_k)$ for convex and $L$-smooth $f$ is equivalent to establishing the \textit{Lyapunov analysis} $V_{k+1} \le V_k$ for each
$k=0,\dots,N-1$,
where we define
\begin{equation}
\label{eq:gd_lyap_definition}
    V_k = \sum_{i=1}^{k} \lambda_{i-1,i} \convexsmoothineq{x_{i-1}}{x_{i}} + \sum_{i=0}^{k} \lambda_{\star,i} \convexsmoothineq{x_\star}{x_i}
     - \sum_{i=0}^{k} d_i \norm{s_i}^2
\end{equation}
where
\begin{align*}
\lambda_{i-1,i} &= \frac{i}{2N + 1 - i}, \quad i = 1, \dots, N,
\qquad
\lambda_{\star,i} =
\begin{cases}
\lambda_{0,1} & i = 0 \\
\lambda_{i,i+1} - \lambda_{i-1,i} & i = 1, \dots, N-1 \\
1 - \lambda_{i-1,i} & i = N,
\end{cases} \\
d_i &= \frac{L}{2} \frac{4Ni+2N-2i^2+1}{(2N-i)^2}, \quad i = 0, \dots, N-1 \\
d_N &= \frac{L}{2}, \\
s_i &= \frac{1}{2N+1-i} (x_i-x_\star) - \frac{1}{L} \nabla f(x_i) \\
\convexsmoothineq{x_{i-1}}{x_{i}} &= \pr{ f(x_{i}) - f(x_{i-1}) +  \inner{\nabla f(x_{i}) }{x_{i-1} - x_{i}} + \frac{1}{2L} \norm{ \nabla f(x_{i-1}) - \nabla f(x_{i}) }^2 } \le 0 \\
\convexsmoothineq{x_\star}{x_i} &= \left( f(x_{i})-f(x_{\star}) + \langle \nabla f(x_{i}) , x_{\star} - x_{i} \rangle + \frac{1}{2 L} \| \nabla f(x_{i}) \|^2 \right) \le 0 .
\end{align*}
Chaining the inequalities, we obtain $V_N \le \cdots \le V_1 \le 0$, and this would imply $f(x_N) - f(x_\star) \le \frac{LR^2}{4N+2}$.
Stage~4 checks whether each $V_k$ can be simplified into a summation-free formula written in terms of a small number of vectors, not increasing with $k$.
Direct computation, either numerical or symbolic, quickly reveals that the function value terms in \eqref{eq:gd_lyap_definition} simplify to
\[
    \lambda_{k,k+1} (f(x_k) - f(x_\star)) = \frac{k+1}{2N-k} (f(x_k) - f(x_\star)) .
\]
However, systematically checking that the sum of all inner product and squared norm terms can be simplified is a more involved process.
Below, we outline how this is performed in our code \texttt{gd\_recover\_example\_lyap.ipynb}.

Define the sum of only the inner product terms, excluding the function values, within \eqref{eq:gd_lyap_definition} as
\begin{align}
\label{eq:gd_tilde_V_k_definition}
    \begin{aligned}
            \widetilde{V}_k
             :&= \sum_{i=1}^{k} \lambda_{i-1,i}
             \pr{ \inner{\nabla f(x_{i}) }{x_{i-1} - x_{i}} + \frac{1}{2L} \norm{ \nabla f(x_{i-1}) - \nabla f(x_{i}) }^2 }  \\
             &\phantom{=}
        +\sum_{i=0}^{k} \lambda_{\star,i} \left( \langle \nabla f(x_{i}) , x_{\star} - x_{i} \rangle + \frac{1}{2 L} \| \nabla f(x_{i}) \|^2 \right) \\ &\phantom{=}
             - \sum_{i=0}^{k}
             d_i
             \left\| \frac{1}{2N+1-i} (x_i-x_\star) - \frac{1}{L} \nabla f(x_i) \right\|^2 .
        \end{aligned}
\end{align}
The Gram formulation used in PEP allows us to rewrite this quantity in terms of tangible matrices that we can handle with linear algebraic tricks.
Without loss of generality, assume $x_\star = \mathbf{0}$.
We let
\begin{align*}
    \vP = \begin{bmatrix} x_0 & \nabla f(x_0) & \cdots & \nabla f(x_N) \end{bmatrix} \in \reals^{d \times (N+2)} ,
\end{align*}
and define $\vx_0 = \ve_1$ and $\vg_i = \ve_{i+2}$ for $i=0,\dots,N$ so that $\vP\vx_0 = x_0$ and $\vP\vg_i = \nabla f(x_i)$.
Define $\vx_1, \dots, \vx_N$ by
\begin{align*}
    \vx_{i+1} = \vx_i - \frac{1}{L} \vg_i = \vx_0 - \frac{1}{L}\sum_{j=0}^i \vg_j
\end{align*}
for $i=0,\dots,N-1$ and let $\vx_\star = \mathbf{0}$, so that $\vP \vx_i = x_i$ for all $i\in \{0,\dots,N,\star\}$.
Now consider the Gram matrix $\vG = \vP^\transpose \vP \in \mathbb{S}_+^{N+2}$.
Then we can express the inner product terms as, for example, for $i=1,\ldots,k$,
\begin{align*}
\inprod{\nabla f(x_i)}{x_{i-1} - x_i}
&= \inprod{\vP\vg_i}{\vP(\vx_{i-1} - \vx_i)} \\&
 = \vg_i^\transpose \vP^\transpose \vP (\vx_{i-1} - \vx_i) \\
&= \vg_i^\transpose \vG (\vx_{i-1} - \vx_i)
 = \operatorname{Tr}\left( \vG \vg_i (\vx_{i-1} - \vx_i)^\transpose \right) .
\end{align*}
The last expression can be symmetrized: $\operatorname{Tr}\left( \vG \vg_i (\vx_{i-1} - \vx_i)^\transpose \right) = \operatorname{Tr}\left( \vG \cdot \frac{1}{2} \left( \vg_i (\vx_{i-1} - \vx_i)^\transpose + (\vx_{i-1} - \vx_i) \vg_i^\transpose \right) \right)$.
Applying the similar trick to all terms in \eqref{eq:gd_tilde_V_k_definition}, we can express $\widetilde{V}_k = \operatorname{Tr}(\vG\vV_k)$ with
\[
\begin{aligned}
\vV_k
:&=
\sum_{i=1}^{k}
\lambda_{i-1,i}
\left(
    \vB_{i-1,i}
    + \frac{1}{2L} \vC_{i-1,i}
\right) +\sum_{i=0}^{k}
\lambda_{\star,i}
\left(
    \vB_{\star,i}
    + \frac{1}{2L}\vC_{\star,i}
\right) \\
& -\sum_{i=0}^{k}
d_i
\left(
    \frac{1}{(2N+1-i)^2} \vA_{\star,i}
    + \frac{2}{L(2N+1-i)}\vB_{\star,i}
    + \frac{1}{L^2} \vC_{\star,i}
\right)
\end{aligned}
\]
where $\vA_{i,j}, \vB_{i,j}, \vC_{i,j} \in \mathbb{S}^{N+2}$ are the symmetric matrices
\begin{align*}
    \vA_{i,j} &= (\vx_i - \vx_j)(\vx_i - \vx_j)^{\transpose} \\
    \vB_{i,j} &= \frac{1}{2} \left( \vg_j(\vx_i - \vx_j)^{\transpose} + (\vx_i - \vx_j) \vg_j^{\transpose} \right) \\
    \vC_{i,j} &= (\vg_i - \vg_j)(\vg_i - \vg_j)^{\transpose} .
\end{align*}
Now, if the matrices $\vV_k$ have constant low rank $r$ across $k$, then this indicates that each $\widetilde{V}_k$ can be expressed in terms of $r$ vectors that are linear combinations of $x_0, \nabla f(x_0), \dots, \nabla f(x_N)$.
The process of identifying the low-rank structure of $\vV_k$ is based on the following linear algebraic argument.
If $\text{rank}(\vV_k) = r$ and $\mathcal{B}_k = \{\vv_{k,1},\ldots,\vv_{k,r}\}$ is a basis of the column space $\operatorname{col}(\vV_k)$, then there exists $\vC_k \in \reals^{r\times r}$ such that
\begin{align}
\label{eq:V_k_matrix_decomposition}
    \vV_k = \begin{bmatrix} \vv_{k,1} & \cdots \vv_{k,r} \end{bmatrix} \vC_k \begin{bmatrix} \vv_{k,1} & \cdots & \vv_{k,r} \end{bmatrix}^\transpose .
\end{align}
Then $\widetilde{V}_k$ can be expressed as inner products of $\vP\vv_{k,1}, \dots, \vP\vv_{k,r}$, and the associated coefficients are the entries of $\vC_k$.
The process of finding $\cB_k$ and $\vC_k$ is as follows.

\begin{enumerate}
\item Compute $r = \text{rank}(\vV_k)$. For GD, $r=3$.

\item Build the set of candidate vectors as
\[
    \cT = \{\vx_0, \dots, \vx_N, \vx_\star, \vg_0, \dots, \vg_N \} , \quad \cS = \cT \cup \{\vu - \vv: \vu, \vv \in \cT \} .
\]
Then, use the \texttt{vectors\_in\_column\_space} function to determine the candidate vectors that belong to $\operatorname{col}(\vV_k)$.
Denote the %
candidates as $\widetilde{\cB}_k=\{\vv\in\cS:\vv\in\operatorname{col}(\vV_k)\}$.

\item Use the
{\texttt{find\_basis\_with\_sparsest\_coefficients}}
function to search over linearly independent subsets of $\widetilde{\cB}_k$ with $r$ elements and select the one that yields the sparsest $\vC_k$ in the decomposition~\eqref{eq:V_k_matrix_decomposition}.
This will be our $\cB_k$.
In case of GD, we find $\cB_k = \{ \vx_0 - \vx_\star, \vx_{k+1} - \vx_\star, \vg_k \}$.

\item Identify the closed form of the Lyapunov coefficients, i.e., the entries of $\vC_k$.
This yields the analytic expression of $\widetilde{V}_k$, and thus completes the simplification of $V_k$.
For GD, we find
$$
\begin{aligned}
V_k &= \frac{k+1}{2N-k}\bigl(f(x_k)-f(x_{\star})\bigr)
-\frac{L}{4N+2}\|x_0-x_{\star}\|^2 \\
&
-\frac{k+1}{2L(2N-k)}\|\nabla f(x_k)\|^2
+\frac{L(2N-2k-1)}{2(2N-k)^2}\|x_{k+1}-x_{\star}\|^2
\end{aligned}
$$

\end{enumerate}


\subsection{{Implementation in \texorpdfstring{\texttt{lyapunov\_utils.py}}{lyapunov\_utils.py}}}
\label{appendix:lyapunov_utils}

The role of \texttt{lyapunov\_utils.py} is to expose the
linear-algebraic structure of a numerical candidate Lyapunov function in terms of
mathematical vectors that an agent and a user can inspect. The
inner-product part of a PEP scalar is stored as a matrix in a common Gram
coordinate system. The module accesses this matrix and the coordinates of
candidate vectors through \texttt{ExpressionManager} and then returns a short
basis and a coefficient matrix for the same quadratic form. In other words, the module
connects the quadratic part of \(V_k\) to expressions involving recognizable
iterates, gradients, and auxiliary vectors.

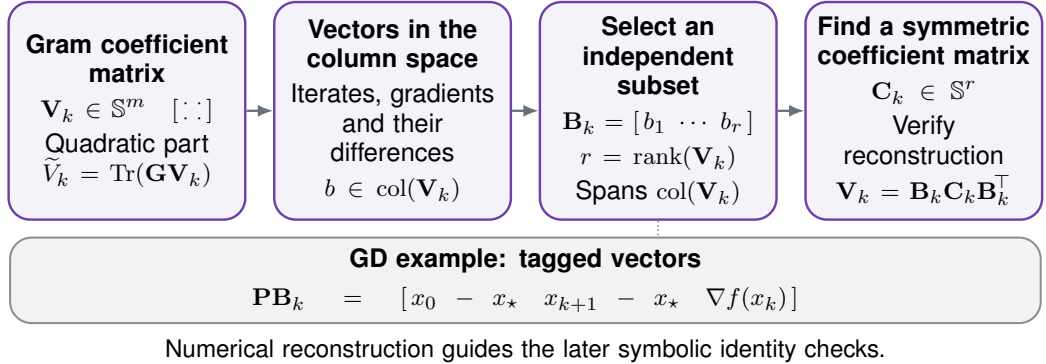
\begin{figure}[H]
\centering
\resizebox{\linewidth}{!}{\begingroup
\definecolor{basisfigpurple}{HTML}{6D3FA8}
\definecolor{basisfigslate}{HTML}{334155}
\begin{tikzpicture}[
  x=1cm,y=1cm,
  font=\sffamily\fontsize{9.5}{11}\selectfont,
  text=black,
  every node/.append style={execute at begin node={\hyphenpenalty=10000\exhyphenpenalty=10000}},
  card/.style={
    rounded corners=8pt,draw=basisfigpurple,fill=basisfigpurple!7,
    line width=0.9pt,minimum width=3.1cm,minimum height=2.85cm,
    text width=2.86cm,align=center,inner xsep=3pt,inner ysep=3pt
  },
  flow/.style={-{Latex[length=2.0mm,width=1.6mm]},
    line width=0.85pt,draw=basisfigslate!75},
  band/.style={
    rounded corners=8pt,draw=basisfigpurple,fill=basisfigpurple!7,
    line width=0.75pt,text width=13.15cm,align=center,
    inner xsep=5pt,inner ysep=4pt
  }
]
  \node[card] (matrix) at (0,0) {%
    {\bfseries Gram coefficient\\matrix}\\[2pt]
    $\mathbf V_k\in\mathbb S^m$\quad
    $\left[\begin{smallmatrix}
      \cdot&\cdot\\
      \cdot&\cdot
    \end{smallmatrix}\right]$\\[2pt]
    Quadratic part\\
    $\widetilde V_k=\operatorname{Tr}(\mathbf G\mathbf V_k)$
  };
  \node[card] (filter) at (3.48,0) {%
    {\bfseries Vectors in the\\column space}\\[2pt]
    Iterates, gradients\\and their differences\\[2pt]
    $b\in\operatorname{col}(\mathbf V_k)$
  };
  \node[card] (basis) at (6.96,0) {%
    {\bfseries Select an\\independent subset}\\[2pt]
    $\mathbf B_k=[\,b_1\ \cdots\ b_r\,]$\\[2pt]
    $r=\operatorname{rank}(\mathbf V_k)$\\[2pt]
    Spans $\operatorname{col}(\mathbf V_k)$
  };
  \node[card] (coefficients) at (10.44,0) {%
    {\bfseries Find a symmetric\\coefficient matrix}\\[2pt]
    $\mathbf C_k\in\mathbb S^r$\\[2pt]
    Verify\\reconstruction\\[2pt]
    $\mathbf V_k=\mathbf B_k\mathbf C_k\mathbf B_k^{\!\top}$
  };

  \draw[flow] (matrix.east) -- (filter.west);
  \draw[flow] (filter.east) -- (basis.west);
  \draw[flow] (basis.east) -- (coefficients.west);

  \node[fit=(matrix)(filter)(basis)(coefficients),inner sep=0pt] (cards) {};
  \node[band,draw=black!45,fill=black!5,anchor=north,minimum height=1.12cm] (gd)
    at ($(cards.south)+(0,-0.20)$) {%
    {\bfseries GD example: tagged vectors}\\[3pt]
    $\mathbf P\mathbf B_k=
      [\,x_0-x_\star\quad x_{k+1}-x_\star\quad\nabla f(x_k)\,]$
  };
  \draw[draw=basisfigslate!60,densely dotted,line width=0.6pt]
    (basis.south) -- (gd.north -| basis.south);

  \node[anchor=north,font=\sffamily\fontsize{9}{10.5}\selectfont,
    align=center,text width=13.15cm,inner sep=0pt]
    at ($(gd.south)+(0,-0.20)$)
    {Numerical reconstruction guides the later symbolic identity checks.};
\end{tikzpicture}
\endgroup}
\caption{{Recovering an interpretable representation of the quadratic part.
Given candidates spanning the target column space, the routines select a basis
and recover a small symmetric coefficient matrix whose reconstruction is
checked numerically.
Function value and constant terms are retained in the complete Lyapunov function.
The GD example uses the notation of Section~\ref{appendix:workflow_detail_gd}.}}
\label{fig:basis_recovery}
\end{figure}

\subsubsection{Preserving mathematical names}
The vector-tagging and display mechanisms belong to the underlying
domain-specific library. The utilities preserve this interface by accepting
and returning the original \texttt{Vector} objects, rather than replacing
them with anonymous coordinate columns. Thus, when a basis is returned in a
notebook, its entries retain their user-facing \texttt{repr} strings, and
the rows and columns of the associated coefficient matrix can be read in
that same order. This link between Gram coordinates and tagged vectors is
what makes the numerical basis search useful for proof development.

\subsubsection{Finding an interpretable basis}
The main basis-search path begins with
\texttt{vectors\_in\_column\_space}. Given a quadratic scalar and a list of
mathematically motivated candidates, it uses an SVD projection test to keep
the vectors that lie in the column space of the scalar's inner-product
matrix. The user can then call
\texttt{find\_basis\_with\_sparsest\_coefficients}, which enumerates
independent subsets whose size is the numerical rank of the candidate
coordinates and chooses a basis whose
coefficient matrix has the largest number of numerically zero entries.
Specified \texttt{fixed\_vectors} can be required to appear in every tested
basis. The simpler \texttt{select\_independent\_subset} routine greedily
removes dependent vectors in input order. It is useful when the desired
candidate ordering is already known, but it does not perform the sparsity
search.

\subsubsection{Reconstructing the quadratic form}
For a proposed independent basis
\(\boldsymbol v=(v_1,\ldots,v_r)\),
\texttt{find\_symmetric\_coefficient\_matrix} checks that the column space
of the target matrix is contained in \(\operatorname{span}(\boldsymbol v)\)
and solves for the symmetric matrix \(C\) satisfying
\begin{equation*}
    V^{\mathrm{quad}}=\boldsymbol v^{\top}C\boldsymbol v.
\end{equation*}
This reconstruction is the basic numerical check used both inside the
sparse-basis search and after an agent proposes a basis directly. A second,
more specialized routine,
\texttt{complete\_basis\_with\_sparsifying\_last\_vector}, applies when all
but one basis direction have been fixed. It finds the missing numerical
direction in the column space and shifts it by the fixed directions, when
possible, so that the last off-diagonal block of \(C\) vanishes. This option
is useful when the candidate list does not already contain a convenient
final vector. The resulting coordinate vector still must be interpreted as
a mathematical expression.

\subsubsection{From fixed numerical instances to an analytic formula}
These routines analyze a specified numerical instance. They do not
automatically align bases across iteration indices or infer a closed-form,
\(k\)-dependent pattern. In the \workflowname\ workflow, the agent repeats
the basis and reconstruction calculations for several values of \(k\) and
\(N\), compares the tagged vectors and the entries of \(C\), and conjectures
their analytic formulas. The final stage then verifies the resulting
Lyapunov identity, boundary relations, and coefficient signs symbolically.
Keeping this division explicit is important: the module supplies
inspectable numerical evidence, while the proof establishes the identity
for arbitrary indices and horizons.

\subsubsection{Auxiliary LDL decomposition}
\texttt{ldl\_decompose\_with\_reversed\_basis} is a specialized helper for a
named slack matrix whose basis is ordered chronologically. It reverses the
basis, performs an LDL decomposition, propagates any factorization
permutation to the vector names, and can print the labeled factors. This is
useful for inspecting certain slack-matrix decompositions, but it is an
auxiliary diagnostic rather than the main basis-recovery mechanism used by
the current Lyapunov workflow.

\subsubsection{Role in the \texorpdfstring{\workflowname}{Peppy} architecture}
The module's
column-space, basis-selection, and coefficient-reconstruction routines alongside \pepflow's Gram representation and vector tags are crucial components of Stage 4 of \workflowname.
Without these components, the matrix of the quadratic part of \(V_k\) is opaque, and identifying a compact mathematical representation
requires case-specific coordinate calculations. Together they let the
agent test proposed vectors against this matrix, retain a basis
that a human can read, and recover the small matrix that must later be
expressed and verified symbolically.

The computations remain subject to numerical tolerances and to the
candidate vectors proposed by the agent. In particular, the sparse-basis
routine performs an exhaustive search within the supplied candidate family
and warns when the number of subsets becomes large. It neither guarantees a
globally preferred mathematical basis nor proves the inferred formulas.
Its purpose is to make the transition from the matrix of the quadratic
part to a readable candidate Lyapunov function systematic and understandable.
\par


\section{Additional details for Section~\ref{section:workflow-implementation}}
\label{appendix:implementation_detail}

This sections documents the details of the commands used in \workflowname\ to start with an
algorithm description and elicit analytic Lyapunov analysis.
{Each block specifies a mathematical object to construct, checks on
that object, and the evidence passed to the next block.}

\subsection{{Modular execution and state}}

The five commands are executed sequentially{, with recovery loops when needed}:
\[
\begin{gathered}
\texttt{/pep-implement}
\rightarrow
\texttt{/pep-full-proof}
\rightarrow
\texttt{/lyap-define}
\\
\rightarrow
\texttt{/lyap-vectors}
\rightarrow
\texttt{/lyap-closed-form}.
\end{gathered}
\]
{The Codex plugin provides skill entry points and accompanying workflow
instructions under \texttt{plugins/peppy-workflow/skills/}. For Claude,}
each command is written in a Markdown file in the directory \texttt{.claude/commands}.

The first command, \texttt{/pep-implement}, starts by parsing the information
provided by the user{ into the fields listed below. It creates the executable
setup and initializes the Lyapunov notebook.}
Each subsequent command typically reads the state generated by the previous
block, performs a specific discovery or verification task, writes a new state
file, and updates the same Lyapunov notebook.

\begin{tcolorbox}[
  enhanced,
  breakable,
  colback=gray!5,
  coltext=.,
  colframe=gray!30,
  boxrule=0.5pt,
  arc=2pt,
  left=10pt,
  right=10pt,
  top=10pt,
  bottom=10pt,
  fontupper=\sffamily
]

\renewcommand{\arraystretch}{1.45}
\setlength{\tabcolsep}{5pt}
\setlength{\parskip}{0pt}
\newcommand{\peppyparserow}[1]{%
  \par\penalty0\vskip-\arrayrulewidth\nointerlineskip
  \noindent
  \begin{tabularx}{\linewidth}{|>{\raggedright\arraybackslash}p{0.34\linewidth}|>{\raggedright\arraybackslash}X|}
  \hline
  #1
  \hline
  \end{tabularx}\par
}

\vbox{\hsize=\linewidth
{\large\bfseries Step 1 --- Parse}

\vspace{0.3em}
\hrule
\vspace{1.2em}
From \code{\$ARGUMENTS}, identify:

\vspace{1em}

\noindent
\begin{tabularx}{\linewidth}{|>{\raggedright\arraybackslash}p{0.34\linewidth}|>{\raggedright\arraybackslash}X|}
\hline
\rowcolor{gray!5}
\textbf{Field} & \textbf{Description} \\
\hline
\code{ALGO\_NAME}
&
\code{snake\_case} identifier, e.g. \code{heavy\_ball}
\\
\hline
\end{tabularx}
\par
}
\peppyparserow{
\rowcolor{gray!4}
\code{PROBLEM\_TYPE}
&
\code{smooth\_convex} / \code{smooth\_strongly\_convex} /
\code{monotone\_operator} / \code{composite}
\\
}
\peppyparserow{
\code{OBJ\_TAG}
&
\code{"f"} (functions) or \code{"A"} (operators)
\\
}
\peppyparserow{
\rowcolor{gray!4}
\code{PARAMS}
&
All symbolic parameters (e.g. \code{L}, \code{alpha}, \code{beta}, \code{R})
\\
}
\peppyparserow{
\code{PERFORMANCE\_METRIC}
&
e.g. \code{f(x\_N) - f(x\_star)} or \code{\textbar\textbar A(x\_N)\textbar\textbar\texttwosuperior}
\\
}
\peppyparserow{
\rowcolor{gray!4}
\code{INITIAL\_CONDITION}
&
e.g. \code{\textbar\textbar x\_0 - x\_star\textbar\textbar\texttwosuperior <= R\texttwosuperior}
\\
}
\peppyparserow{
\code{CONJECTURED\_RATE}
&
Known bound or \code{"unknown"}
\\
}
\end{tcolorbox}

The following table briefly summarizes the tasks performed by each command
and their main outputs, {together with the checks before completion.
Numerical checks apply to sampled instances. Block~5 establishes the
analytic Lyapunov proof.}

\begingroup
\small
\renewcommand{\arraystretch}{1.12}
\setlength{\tabcolsep}{3pt}
\setlength{\LTleft}{0pt}
\setlength{\LTright}{0pt}
\begin{longtable}{@{}>{\raggedright\arraybackslash}p{\dimexpr.22\linewidth-\tabcolsep\relax}
>{\raggedright\arraybackslash}p{\dimexpr.28\linewidth-2\tabcolsep\relax}
>{\raggedright\arraybackslash}p{\dimexpr.23\linewidth-2\tabcolsep\relax}
>{\raggedright\arraybackslash}p{\dimexpr.27\linewidth-\tabcolsep\relax}@{}}
\toprule
\textbf{{Block / }command} & \textbf{Tasks} & \textbf{Main output} & \textbf{{Checks before completion}} \\
\midrule
\endfirsthead
\toprule
\textbf{{Block / }command} & \textbf{Tasks} & \textbf{Main output} & \textbf{{Checks before completion}} \\
\midrule
\endhead
\midrule
\multicolumn{4}{r}{Continued on the next page}\\
\endfoot
\bottomrule
\endlastfoot

\leavevmode{\textbf{B1}} \newline \texttt{/pep-implement}
& Parse the algorithm information. Implement the algorithm.
Solve PEP problem for various $N=1,\dots,7$ and identify {a conjectured} convergence rate.
Initialize the Lyapunov notebook.
& Setup module. Numerical sweep data.
Block 1 state \texttt{\{ALGO\_NAME\}\allowbreak\_b1.json}.
& \leavevmode{Check problem encoding and successful solves. Label the candidate rate as numerical evidence.} \\
\addlinespace[7pt]

\leavevmode{\textbf{B2}} \newline \texttt{/pep-full-proof}
& Find sparsity patterns of the dual variables.
Find a decomposition of the PSD matrix corresponding to the sum-of-squares term in \eqref{eq:gd_full_pep}.
Try to find closed-form expressions for the dual variables{ and the PSD matrix}.
& Closed-form \(\lambda\){ and \(S\), when available}.
{Verified numerical PEP certificate}.
{Separate numerical and analytic verification statuses.}
Block 2 state \texttt{\{ALGO\_NAME\}\allowbreak\_b2.json}.
& \leavevmode{Check the full numerical identity, dual signs, PSD slack, decomposition, and dense/relaxed bound consistency. Record the initial formula search.} \\
\addlinespace[7pt]

\leavevmode{\textbf{B3}} \newline \texttt{/lyap-define}
& Build Lyapunov partial sums \(V_k\) from the full-PEP certificate.
{Compute the rank profile of the matrices representing the quadratic parts of \(V_0, \dots, V_N\).}
& Partial sums \(V_k\), stored in the list \texttt{lyap}.
Rank profile \texttt{[r\_0, r\_1, ..., r\_N]}.
{Rank evidence across horizons.}
Block 3 state \texttt{\{ALGO\_NAME\}\allowbreak\_b3.json}.
& \leavevmode{Check decomposition, coverage, numerical step identities, and required signs. Confirm consistent interior rank at two or more horizons with fixed boundary treatment.} \\
\addlinespace[7pt]

\leavevmode{\textbf{B4}} \newline \texttt{/lyap-vectors}
& Find special vectors that yield a summation-free expression of \(V_k\).
Find basis patterns: special vector patterns indexed by \(k\).
Extract {and verify} coefficient matrices{. Optionally} infer closed-form coefficient patterns.
& Basis patterns.
Coefficient matrices and {available formulas or explicit unknowns}.
Block 4 state \texttt{\{ALGO\_NAME\}\allowbreak\_b4.json}.
& \leavevmode{Check independence of basis vectors and reconstruction of quadratic vector expressions and function value terms at every index, including exceptions at boundaries.} \\
\addlinespace[7pt]

\leavevmode{\textbf{B5}} \newline \texttt{/lyap-\allowbreak closed-form}
& \leavevmode{Derive unresolved coefficients and} assemble the final closed-form Lyapunov proof.
Verify the result symbolically.
& Final notebook.
Block 5 state \texttt{\{ALGO\_NAME\}\allowbreak\_b5.json}.
LaTeX theorem/proof.
& \leavevmode{Establish exact step, initial, and terminal identities and sign/domain conditions. Verify that the proof yields the stated convergence bound. Complete notebook delivery checks.} \\

\end{longtable}
\endgroup

\subsection{{Discovery and recovery using structure}}

\subsubsection{{From numerical certificates to analytic coefficients}}
{After Block~2's initial formula search, verified numerical certificates can
support Blocks~3 and~4 even if analytic formulas for dual variables are not identified yet.
Block~4 passes a structured, summation-free expression for the Lyapunov function $V_k$ with possibly unknown coefficients to Block~5. 
Block~5 derives missing analytic information and performs necessary symbolic checks of identities. 
}

\subsubsection{{Structural search using rank}}
{Block~3 uses rank profiles to propose groupings of inequalities and square terms.
Confirmation requires the same rank patterns at two or more horizons from the same type of grouping. 
}
{If a candidate grouping is unhelpful, the agent tries alternative groupings or square decompositions.
If the failure persists
the agent may return to Block~2 and start over using another set of proof certificates. 
Block~4 identifies a set of linear independent directions pertaining to the algorithm that is sufficient for expressing $V_k$, and checks if the Lyapunov identity can be reconstructed. 
If these checks fail, the search may return to Block~3.}

\begin{figure}[H]
\centering
\resizebox{\linewidth}{!}{\begingroup
\definecolor{rankpurple}{HTML}{6D3FA8}
\definecolor{rankslate}{HTML}{334155}
\definecolor{rankamber}{HTML}{F59E0B}
\begin{tikzpicture}[
  x=1cm,y=1cm,font=\sffamily\fontsize{9.5}{11.3}\selectfont,
  text=black,
  card/.style={rounded corners=8pt,draw=rankpurple,line width=.85pt,
    fill=rankpurple!7,align=center,inner sep=5pt},
  neutral/.style={card,draw=rankslate!45,fill=white},
  badge/.style={circle,fill=rankpurple,text=white,minimum size=6mm,
    inner sep=0pt,font=\sffamily\bfseries\fontsize{10}{11}\selectfont},
  flow/.style={-{Latex[length=2mm,width=1.5mm]},draw=rankslate!75,line width=.8pt},
  recovery/.style={flow,draw=rankslate!60,line width=.65pt,rounded corners=3pt},
  label/.style={fill=white,inner sep=2pt,font=\sffamily\fontsize{8.5}{10}\selectfont},
  cell/.style={draw=rankpurple!55,fill=rankpurple!9,rounded corners=1.4pt,
    minimum width=.36cm,minimum height=.36cm,inner sep=0pt,
    font=\sffamily\fontsize{9}{10}\selectfont},
  dimcell/.style={cell,draw=rankslate!25,fill=rankslate!5,text=rankslate!60},
  note/.style={font=\sffamily\fontsize{8.5}{10}\selectfont,align=center},
  title/.style={font=\sffamily\bfseries\fontsize{10}{11.5}\selectfont,anchor=west}
]
\node[title] at (0,3.15) {(a) Search over partial sums};
\node[neutral,minimum width=2.9cm,minimum height=3.65cm] (b2) at (1.45,0) {};
\node[badge,fill=black] at (1.45,1.35) {2};
\node[align=center,font=\sffamily\bfseries\fontsize{9}{10.5}\selectfont] at (1.45,.60)
  {Verified numerical\\certificates};
\node at (1.45,-.06) {$\lambda,\ S$};
\node[note] at (1.45,-.94) {Analytic formulas\\may remain unknown};

\node[card,minimum width=6.5cm,minimum height=4.1cm] (b3) at (7,0) {};
\node[badge] at (7,1.62) {3};
\node[font=\sffamily\bfseries\fontsize{9.5}{11.3}\selectfont] at (7,1.08) {Regroup inequalities and squares};
\node[note,text=rankslate] at (7,.63) {Illustrative ranks of the quadratic part};

\node[align=center,font=\sffamily\bfseries\fontsize{9}{10.5}\selectfont]
  at (4.55,-.035) {Grouping\\A};
\node[align=center,font=\sffamily\bfseries\fontsize{9}{10.5}\selectfont,text=rankpurple]
  at (4.55,-1.255) {Grouping\\B};
\foreach \h/\yy in {6/.20,8/-.27,6/-1.02,8/-1.49}{
  \node[anchor=east,note] at (6.30,\yy) {$N=\h$};
}
\foreach \k/\r in {0/*,1/3,2/4,3/5,4/4,5/3,6/*}{
  \node[dimcell] at ({6.70+.38*\k},.20) {$\r$};
}
\foreach \k/\r in {0/*,1/3,2/4,3/5,4/6,5/5,6/4,7/3,8/*}{
  \node[dimcell] at ({6.70+.38*\k},-.27) {$\r$};
}
\foreach \k in {0,...,6}{
  \ifnum\k=0 \node[dimcell] at ({6.70+.38*\k},-1.02) {$*$};
  \else\ifnum\k=6 \node[dimcell] at ({6.70+.38*\k},-1.02) {$*$};
  \else \node[cell] at ({6.70+.38*\k},-1.02) {$3$};\fi\fi
}
\foreach \k in {0,...,8}{
  \ifnum\k=0 \node[dimcell] at ({6.70+.38*\k},-1.49) {$*$};
  \else\ifnum\k=8 \node[dimcell] at ({6.70+.38*\k},-1.49) {$*$};
  \else \node[cell] at ({6.70+.38*\k},-1.49) {$3$};\fi\fi
}

\node[card,minimum width=3cm,minimum height=3.65cm] (b4) at (12.95,0) {};
\node[badge] at (12.95,1.35) {4};
\node[align=center,font=\sffamily\bfseries\fontsize{9}{10.5}\selectfont] at (12.95,.60)
  {Search basis\\vectors and verify};
\node[note] at (12.95,-0.7) {Numerical coefficients\\are known, but\\analytic formulas\\may remain unknown};
\draw[flow] (b2.east) -- (b3.west);
\draw[flow] (b3.east) -- (b4.west);
\draw[recovery,dashed] ([xshift=-.60cm]b3.north) -- (6.4,2.53)
  -- node[label,above] {New certificate if needed} (1.45,2.53) -- (b2.north);
\draw[recovery] (b4.north) -- (12.95,2.53)
  -- node[label,above] {Basis search or reconstruction fails} (7.6,2.53)
  -- ([xshift=.60cm]b3.north);
\node[note,anchor=north] at (7.2,-2.27)
  {Illustration of different Lyapunov constructions.\\
   Grouping~B has constant rank except for boundary entries $*$,\\representing an admissible grouping.};

\begin{scope}[yshift=-.6cm]
\node[title] at (0,-3.21) {(b) From the recovered structure to an analytic proof};
\node[card,minimum width=3.08cm,minimum height=2.60cm] (template) at (1.54,-5.10) {};
\node[note,text=rankpurple,font=\sffamily\bfseries] at (1.54,-4.14) {From Block 4};
\node[align=center,font=\sffamily\bfseries] at (1.54,-4.91) {Verified structure\\(ansatz) of $V_k$};
\node[note] at (1.54,-5.77) {Unknown coefficient\\functions};

\node[card,minimum width=3.08cm,minimum height=2.60cm] (derive) at (5.34,-5.10) {};
\node[font=\sffamily\bfseries\fontsize{9}{10.5}\selectfont] at (5.34,-4.27) {Derive coefficients};
\node[note] at (5.34,-5.37) {Coefficient identities\\Recurrences\\Endpoint conditions\\etc.};

\node[card,minimum width=3.08cm,minimum height=2.60cm] (verify) at (9.14,-5.10) {};
\node[font=\sffamily\bfseries] at (9.14,-4.27) {Verify the proof};
\node[note] at (9.14,-5.37) {Descent identity at\\interior iteration\\and endpoints,\\Certificate signs};

\node[card,minimum width=3.08cm,minimum height=2.60cm] (result) at (12.94,-5.10) {};
\node[font=\sffamily\bfseries] at (12.94,-4.27) {Analytic theorem};
\node at (12.94,-4.84) {$V_{k+1}\le V_k$};
\node[note] at (12.94,-5.58) {Closed form of $V_k$\\and convergence\\bound};
\draw[flow] (template.east) -- (derive.west);
\draw[flow] (derive.east) -- (verify.west);
\draw[flow] (verify.east) -- (result.west);
\draw[draw=rankpurple!45,dashed,line width=.6pt]
  (3.80,-3.69) -- (14.48,-3.69);
\node[label,text=rankpurple,font=\sffamily\bfseries] at (9.14,-3.69) {Block 5};
\end{scope}
\end{tikzpicture}
\endgroup}
\caption{{Discovery of Lyapunov proof structure and analytic completion.
(a)~Stratifying the proof structure given by certificates via different groupings of inequalities and squares.
The displayed ranks are illustrative.
(b) Numerically verified Lyapunov structure can be passed to Block~5, where the unresolved analytic coefficients can be derived and verified to produce the final analytic theorem.}}
\label{fig:rank_structural_search}
\end{figure}
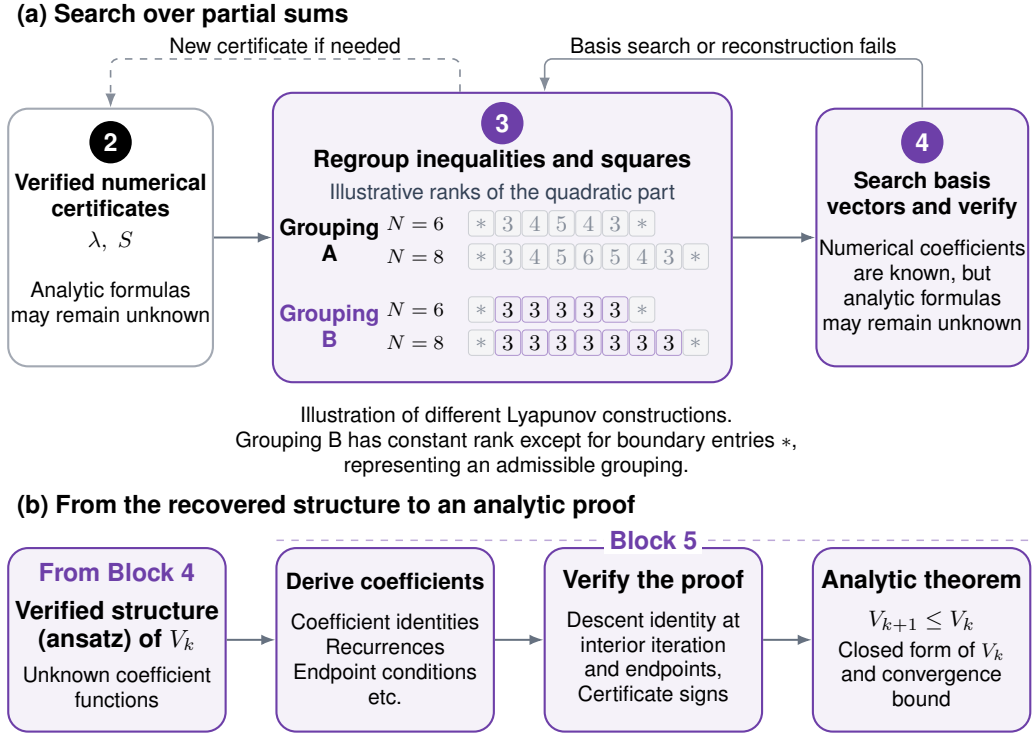

\subsection{{Verification and completion criteria}}

\subsubsection{{Numerical evidence and analytic proof}}
{While the final goal of the workflow is to derive an analytic Lyapunov proof, prior to Block~5 the proof may have only numerical evidence and not be completed in the symbolic form.
This completion occurs within Block~5.
At that point, $V_k$ is expressed in a simple form with possibly undetermined coefficients, and $V_k - V_{k+1}$ is written as a linear combination of specific inequalities and square terms, also with coefficients that are potentially not identified analytically.
Using this information, the AI agent determines the analytic formulas for the coefficients based on symbolic linear system solve, and when needed, checks for their well-definedness and positivity.
Together, these arguments establish $V_{k+1}\le V_k$, which then directly implies the final convergence bound.
}

\subsubsection{{Inspectable proof artifacts and delivery}}
The final result is a Jupyter notebook file that contains a LaTeX theorem
statement and proof outline, together with the detailed results of the whole
workflow.
Definitions/notations needed for the proof and executed symbolic checks remain visible in the notebook cells.
Details are as follows.

\begin{center}
\small
\renewcommand{\arraystretch}{1.12}
{%
\begin{tabularx}{\linewidth}{@{}>{\raggedright\arraybackslash}p{0.24\linewidth}>{\raggedright\arraybackslash}X@{}}
\toprule
\textbf{Notebook section} & \textbf{Contents} \\
\midrule
Problem and result & Algorithm and assumptions, followed by the theorem and proof outline. \\
Executable setup & Imports, function or operator definitions, and PEP construction. \\
Numerical evidence & Rate comparison, dense and relaxed solves, and certificate checks. \\
Lyapunov construction & Partial sums, rank profiles, and basis search. \\
Analytic verification & Coefficient reconstruction, closed-form $V_k$, and symbolic step, initial, and terminal identities. \\
\bottomrule
\end{tabularx}}
\end{center}

\subsection{{Artifacts and human interaction}}
{Block~1 creates} the output file
\[
\texttt{examples\_peppy/\{ALGO\_NAME\}/state/\{ALGO\_NAME\}\_b1.json},
\]
and initializes a Jupyter notebook
\[
\texttt{examples\_peppy/\{ALGO\_NAME\}/\{ALGO\_NAME\}\_example\_lyap.ipynb}.
\]
The \texttt{state/} directory contains information of block states and certificates, while 
additional helpers, exploratory evidence, and logs
are stored in \texttt{supplementary/},
with paths following the notebook destination.

{As additional guidance, \texttt{supplementary/} may accumulate
files that are unnecessary for distribution, so we used the auxiliary skill
\texttt{clean-peppy-artifacts} to clean up the released code.
Simple proofs can be read directly from the Jupyter notebook, while more
involved proofs can be easier to read as a LaTeX document and PDF generated
from the notebook using the optional \texttt{pdf-from-ipynb} skill.}
While the agent first attempts to handle the task autonomously, if it is unsuccessful,
the user may interact with the agent to supply it with alternative groupings or square decomposition.

\codeavailability
Examples of the theorem statements and proofs are provided in
\Cref{appendix:experiments}.

      \colorlet{experimentboxtext}{black}
\lstdefinestyle{experimentprompt}{
  basicstyle=\ttfamily\small\color{experimentboxtext},
  columns=fullflexible, breaklines=true, keepspaces=true,
  breakindent=0pt, breakautoindent=false, breakatwhitespace=true,
  showstringspaces=false
}
{%
\section{Omitted details of Section~\ref{sec:experiments}}\label{appendix:experiments}

We provide experimental details omitted from the main text, including the input prompts and generated proofs. 
For each tested algorithm, we report the complete theorem and proof outline produced by repetition~1 (the first run out of five), \emph{without editorial correction other than minimal typesetting and line breaks}.
The code, notebooks, and evaluation records for these experiments are available at:
\begin{center}
  {\color{blue}\url{https://github.com/pepflow-lib/PEPFlow/tree/e/sjw/peppy-supplementary-v1/examples_peppy/_evaluations}}
\end{center}

\subsection{Experimental setup and evaluation}\label{appendix:five-run-protocol}
\subsubsection{Execution settings}
We ran five repetitions for each of six conditions: PGM with and without
supplied SOS, OGM, BPPM, FEG, and Dual-FEG. Generation and independent evaluation
both used GPT-6 Astra with medium reasoning. Up to three AI agents (workers)
generated proofs concurrently, with each repetition completed before the next began.
Each generation attempt and evaluation had a limit of two hours and no token cap.
The experiment was launched with the following request:
}

\begin{commandbox}[unbreakable, coltext=experimentboxtext, listing options={style=experimentprompt}]
Use $evaluate-peppy with GPT-6 Astra and medium reasoning for both generation and evaluation.
Run all six conditions with 5 repetitions per condition, in parallel.
\end{commandbox}

{%
The recorded environment used Windows x86-64 and Python~3.11.13, with
\texttt{NumPy}~2.3.1, \texttt{SymPy}~1.14.0, and \texttt{CVXPY}~1.7.1. Numerical solves used \texttt{Clarabel}~0.11.1
with one numerical thread per worker.
Workers were assigned separate Python/Jupyter environments.

\subsubsection{Clean-room setting}
Here, a clean-room run starts in a fresh context with the selected problem
packet, approved generic framework code and workflow instructions, and
artifacts created within that run. Workers are instructed not to consult
existing target proofs, other repetitions, or other workers' results. BPPM receives its conjectured
rate, and PGM with supplied SOS receives a candidate SOS decomposition to verify.
The full inputs below specify this assistance. The following excerpt from
the worker instructions states the source restriction:
}

\begin{commandbox}[coltext=experimentboxtext, listing options={style=experimentprompt}]
Do not inspect, search for, open, quote, copy, or derive hints from any target-specific pre-existing result in this repository, another checkout, external sources, or any form of history.
\end{commandbox}

{%
The clean-room policy is based on instructions, rather than directly enforced filesystem isolation or control over prior model knowledge. Evaluators check each proof in a
separate context without other attempts' results.

\subsubsection{Success criteria}
A successful proof establishes the prescribed bound under the input assumptions for every integer $N\ge1$.
It need not use a particular Lyapunov function or SOS decomposition.
The evaluation checks the following evidence from each block:

\begin{center}
\small
\renewcommand{\arraystretch}{1.15}
\begin{tabularx}{\linewidth}{@{}l>{\raggedright\arraybackslash}X@{}}
\toprule
Block & Required checks \\
\midrule
B1 & Correct algorithm, assumptions, metric, and normalization. Valid numerical sweep and rate comparison. \\
B2 & Coherent full-PEP certificate, agreement of dense and relaxed bounds, identities for all scalar components, required dual signs and PSD conditions, and reconstruction of the slack matrix. \\
B3 & Decomposition into partial sums, coverage and increment identities, and consistent interior rank at two or more horizons with fixed boundary treatment. \\
B4 & Independent spanning bases and reconstruction of the full Lyapunov function, including function value terms, constants, and exceptional indices. \\
B5 & General step, initial and terminal identities, coefficient signs and domains, exceptional cases, and implication of the exact target bound. \\
\bottomrule
\end{tabularx}
\end{center}

Numerical verification on any fixed value of $N$ does not count as success.
Usage of an invalid identity or an unjustified inequality also prevents mathematical success. 
Missing verification and runtime errors are recorded
separately. 
Workflow completion additionally requires the notebook evidence, static checks, and execution from a clean kernel.

\subsubsection{Completion of interrupted verification}
Three runs were flagged during final notebook execution in Block~5 for
Python/Jupyter environment issues explained in the main text: BPPM repetition~2, FEG repetition~3,
and PGM without SOS repetition~2. Each run had already produced its proof before interruption.
After the experiment, the experimenter requested completion of
the omitted verification, treating these generic runtime issues separately
from mathematical correctness. The relevant excerpt from that request is:
}

\begin{commandbox}[coltext=experimentboxtext, listing options={style=experimentprompt}]
Do not rerun proof generation, alter submitted proofs, or supply missing mathematical arguments. Verification code may check the submitted expressions and reasoning.

Treat generic dependency-path and notebook-display issues as operational issues.
\end{commandbox}

{%
This post hoc evaluation retained 28 completed verifications, including
BPPM repetition~2, and completed the two remaining evaluations. The evaluators
replayed selected submitted proof cells and saved certificates in validation copies,
with recorded adjustments confined to execution.
No proofs were regenerated or repaired. The mathematical criteria and
tolerances were unchanged. All 30 submitted proofs passed, and 27 runs
completed the original workflow. No prohibited
mathematical exposure was identified in the available records, although
the access logs were not exhaustive.

\subsection{Computation and observed outputs}\label{appendix:five-run-resources}
\subsubsection{Time and token measurements}
Table~\ref{tab:experiment-costs} reports means and sample standard deviations
over all five attempts per condition, including time consumed by the three
incomplete deliveries. Generation time runs from dispatch to stopping and
includes service waits and recovery. Evaluation time is not included here. Token counts
are input plus output, including reused context, with cached input and
reasoning counted within those totals.
Across all attempts, generation used 24.51 summed hours and 378.47 million
tokens. The original evaluations used 2.87 summed hours and 43.67 million tokens.
Summed worker times are not elapsed experiment time because workers ran in
parallel.

The two additional evaluations used 8.23 minutes and 1.75 million tokens for
PGM without SOS repetition~2, and 6.68 minutes and 1.40 million tokens for
FEG repetition~3. These are evaluation costs, accounted for separately from
generation and excluding orchestration.

\subsubsection{Variation in the PGM outputs}
Overall generation times were similar: $57.9\pm4.2$ minutes with supplied
SOS and $60.4\pm9.7$ minutes without it (mean $\pm$ sample standard deviation).
The five supplied-SOS runs produced the same Lyapunov function after notation and
normalization were aligned. Without supplied SOS, the five runs produced
three distinct forms: one in repetition~1, a shared form in repetitions~2--4,
and a third in repetition~5. These differences persist beyond notation or
overall scaling, while every proof establishes the same target bound.
The outputs from repetitions~1 and~5 below illustrate this variation.
}

{%
\subsection{PGM}\label{appendix:pgm_output}
The two input conditions share the algorithm and assumptions. The second
additionally supplies a candidate SOS decomposition. The archived packets below
make this difference explicit.
}
\subsubsection{Input prompt: PGM without supplied SOS}
\begin{commandbox}[coltext=experimentboxtext, listing options={style=experimentprompt}]
# Common assumptions

All cases use real finite-dimensional Euclidean spaces, positive `L` where present, nonnegative `R`, and integer horizons `N >= 1`. Proximal updates and stated evaluations are assumed well-defined. Supply these common assumptions and only the selected case subsection to its worker.


### `pgm`


- Problem: composite convex minimization. `f` is convex and `L`-smooth; `g` is closed, proper, and convex. Write `h = f + g`.
- Initial condition: `x_star` minimizes `h`, and `||x_0 - x_star||^2 <= R^2`.
- Metric: `h(x_N) - h(x_star)`.
- For `k = 0, ..., N-1`, set `x_{k+1} = prox_{g/L}(x_k - grad f(x_k)/L)`.
- Conjectured rate: unknown.

### `pgm_without_sos`


- No fixed SOS aid is supplied. Derive the certificate and Lyapunov analysis within this attempt. You may discover and use an SOS decomposition yourself.
- Do not read the supplied fixed aid, its source, either condition's pre-existing answers, sibling outputs, or any other attempt's artifacts, including the paired repetition.
\end{commandbox}
\subsubsection{Input prompt: PGM with supplied SOS}
\begin{commandbox}[coltext=experimentboxtext, listing options={style=experimentprompt}]
# Common assumptions

All cases use real finite-dimensional Euclidean spaces, positive `L` where present, nonnegative `R`, and integer horizons `N >= 1`. Proximal updates and stated evaluations are assumed well-defined. Supply these common assumptions and only the selected case subsection to its worker.


### `pgm`


- Problem: composite convex minimization. `f` is convex and `L`-smooth; `g` is closed, proper, and convex. Write `h = f + g`.
- Initial condition: `x_star` minimizes `h`, and `||x_0 - x_star||^2 <= R^2`.
- Metric: `h(x_N) - h(x_star)`.
- For `k = 0, ..., N-1`, set `x_{k+1} = prox_{g/L}(x_k - grad f(x_k)/L)`.
- Conjectured rate: unknown.

### `pgm_with_sos`


- Authorized assistance: the fixed SOS aid below, supplied from the start. Record it as supplied assistance. Derive the remaining certificate and Lyapunov analysis within this attempt.
- The aid is a candidate, not a verified certificate or proof. Verify it against this attempt's certificate and all applicable block gates.

## Fixed PGM SOS Aid


Only a `pgm_with_sos` worker receives this section. This is a candidate decomposition of the nonnegative full-PEP slack. It does not provide the convergence bound, dual multipliers, partial sums, or final Lyapunov function.

For `i = 1, ..., N`, write `a_i = grad f(x_i)`, `a_star = grad f(x_star)`, and choose the proximal subgradient `s_i in partial g(x_i)` given by `s_i = L*(x_{i-1}-x_i) - a_{i-1}`. Define:

```text
c1_i = N / ((2*N+1-i)*(2*N-i)*L)
u1_i = (i/(2*N))*a_i + ((2*N-i)/(2*N))*a_{i-1} - a_star

c2_i = (((i-1)/(2*N-i))*((2*N+2)/(2*N+1))
        + (2*N)/((2*N-i)^2*(2*N+1))) / (2*L)
u2_i = a_{i-1} + s_i - (L/(2*N+1-i))*(x_{i-1}-x_star)

c3_i = (i/(2*N+1-i))*((2*N)/(2*N+1))/(2*L)
u3_i = ((2*N+1)/(2*N))*(a_i+s_i) - (a_{i-1}+s_i)/(2*N)
        - (L/(2*N-i))*(x_i-x_star)

S_candidate = sum(c1_i*||u1_i||^2 + c2_i*||u2_i||^2 + c3_i*||u3_i||^2
                  for i = 1, ..., N)
```

These are mathematical formulas, not directly executable Python. Define `a_0 = grad f(x_0)`. Use exact arithmetic when implementing them. Verify their role against the certificate obtained in this attempt, accounting explicitly for any overall certificate normalization.
\end{commandbox}

\subsubsection{{Generated output: PGM without supplied SOS (repetition 1)}}
\begingroup
\tcbset{coltext=experimentboxtext,breakable,ignore nobreak=true,lines before break=4}
\begin{wideexamplebox}
{%
\textbf{Theorem}

For the stated PGM recurrence with convex \(L\)-smooth \(f\), closed proper convex \(g\), \(L>0\), \(R\ge0\), integer \(N\ge1\), a minimizer \(x_{\star}\) and \(\|x_{0}-x_{\star}\|^{2}\le R^{2}\), the following potential is nonincreasing. All stated evaluations and proximal steps are well-defined.

Put \(d_{k}=x_{k}-x_{\star}\) and \(p_{k}=(\nabla f(x_{k})-\nabla f(x_{\star}))/L\).
For \(1\le k<N\), define
\[\begin{aligned}
&V_{k}=-\frac{L}{4N}\|d_{0}\|^{2}
+\frac{L(N-k)}{(2N-k)^{2}}\|d_{k}\|^{2}
-\frac{\langle d_{k},\nabla f(x_{k})\rangle}{2N-k}
\\
&\quad -\frac{Lk}{2(2N-k)^{2}}\|p_{k}\|^{2}
+\frac{k+1}{2N-k}(f(x_{k})-f(x_{\star}))
+\frac{k}{2N-k}(g(x_{k})-g(x_{\star})).
\end{aligned}\]
The endpoints are \(V_{0}=0\) and
\[V_{N}=h(x_{N})-h(x_{\star})-\frac{L}{4N}\|d_{0}\|^{2}.\]
Consequently, \[h(x_{N})-h(x_{\star})\le\frac{LR^{2}}{4N}.\]

\textbf{Proof outline}

Put \(d_{k}=x_{k}-x_{\star}\) and \(p_{k}=(\nabla f(x_{k})-\nabla f(x_{\star}))/L\).
For \(1\le k<N\), define
\[\begin{aligned}
&V_{k}=-\frac{L}{4N}\|d_{0}\|^{2}
+\frac{L(N-k)}{(2N-k)^{2}}\|d_{k}\|^{2}
-\frac{\langle d_{k},\nabla f(x_{k})\rangle}{2N-k}
\\
&\quad -\frac{Lk}{2(2N-k)^{2}}\|p_{k}\|^{2}
+\frac{k+1}{2N-k}(f(x_{k})-f(x_{\star}))
+\frac{k}{2N-k}(g(x_{k})-g(x_{\star})).
\end{aligned}\]
The endpoints are \(V_{0}=0\) and
\[V_{N}=h(x_{N})-h(x_{\star})-\frac{L}{4N}\|d_{0}\|^{2}.\]

Define the nonpositive interpolation residuals
\[\begin{aligned}
&I_{f}(u,v)=f(v)-f(u)+\langle\nabla f(v),u-v\rangle+\frac{\|\nabla f(u)-\nabla f(v)\|^{2}}{2L},\quad
\\
&\quad I_{g}(u,v)=g(v)-g(u)+\langle s_{v},u-v\rangle,
\end{aligned}\]
where \(s_{v}\in\partial g(v)\) is the selected oracle subgradient. At the minimizer choose \(s_{\star}=-\nabla f(x_{\star})\); at a new iterate use the proximal subgradient.
Write \(a_{k}=(k+1)/(2N-k)\) and \(b_{k}=k/(2N-k)\).
For a step, put \(m=2N-k\), \(p=p_{k}\), \(q=p_{k+1}\),
\(s=(s_{x_{k+1}}+\nabla f(x_{\star}))/L\), \(z=s-d_{k}/m\).
The recurrence gives \(d_{k+1}=d_{k}-p-s\).
Define
\[A_{k}=\frac{k+1}{2(m-1)^{2}}+\frac{k+2}{2(m-1)},\quad
B_{k}=\frac{k+1}{2(m-1)},\quad C_{k}=\frac{km+N-k}{(m-1)^{2}},\]
\[D_{k}=\frac{k+1}{2m(m-1)},\quad E_{k}=\frac{km+1}{2(m-1)^{2}},\quad
F_{k}=\frac{a_{k}}2-\frac{N-k-1}{(m-1)^{2}}-\frac{k}{2m^{2}}.\]
Set \(\alpha_{k}=A_{k}\), \(\beta_{k}=C_{k}-B_{k}^{2}/A_{k}\),
\(\eta_{k}=E_{k}-B_{k}D_{k}/A_{k}\),
\(\gamma_{k}=F_{k}-D_{k}^{2}/A_{k}-\eta_{k}^{2}/\beta_{k}\) and
\[\begin{aligned}
&Q_{k}=L\alpha_{k}\|q+(B_{k}/A_{k})z+(D_{k}/A_{k})p\|^{2}
\\
&\quad +L\beta_{k}\|z+(\eta_{k}/\beta_{k})p\|^{2}+L\gamma_{k}\|p\|^{2}.
\end{aligned}\]
The sign lemma below proves \(\alpha_{k},\beta_{k}>0\) and \(\gamma_{k}\ge0\) for \(0\le k\le N-2\).
For \(N\ge2\), the exact base identity is
\[\begin{aligned}
&V_{1}=a_{0}I_{f}(x_{\star},x_{0})+(a_{1}-a_{0})I_{f}(x_{\star},x_{1})
\\
&\quad +b_{1}I_{g}(x_{\star},x_{1})+a_{0}I_{f}(x_{0},x_{1})
-\frac{L}{4N}\|p_{0}\|^{2}-Q_{0}.
\end{aligned}\]
For \(1\le k\le N-2\), the exact recursion is
\[\begin{aligned}
&V_{k+1}-V_{k}=(a_{k+1}-a_{k})I_{f}(x_{\star},x_{k+1})
+(b_{k+1}-b_{k})I_{g}(x_{\star},x_{k+1})
\\
&\quad +a_{k}I_{f}(x_{k},x_{k+1})+b_{k}I_{g}(x_{k},x_{k+1})-Q_{k}.
\end{aligned}\]
At the final step put \(p=p_{N-1}\), \(q=p_{N}\), \(s=(s_{x_{N}}+\nabla f(x_{\star}))/L\), \(d=d_{N-1}\) and
\[T_{N}=\frac L2\left\|q+s+\frac{p-d}{N+1}\right\|^{2}
+\frac L2\left\|s+\frac{Np-d}{N+1}\right\|^{2}.\]
Then, for \(N\ge2\),
\[\begin{aligned}
&V_{N}-V_{N-1}=(1-a_{N-1})I_{f}(x_{\star},x_{N})+(1-b_{N-1})I_{g}(x_{\star},x_{N})
\\
&\quad +a_{N-1}I_{f}(x_{N-1},x_{N})+b_{N-1}I_{g}(x_{N-1},x_{N})-T_{N}.
\end{aligned}\]
For \(N=1\), use the separate identity
\[\begin{aligned}
&V_{1}=\tfrac12 I_{f}(x_{\star},x_{0})+\tfrac12 I_{f}(x_{\star},x_{1})+I_{g}(x_{\star},x_{1})
\\
&\quad +\tfrac12 I_{f}(x_{0},x_{1})-\tfrac L4\|p_{0}\|^{2}-T_{1}.
\end{aligned}\]
All residual multipliers are nonnegative: \(a_{k+1}-a_{k}=(2N+1)/[(2N-k-1)(2N-k)]\),
\(b_{k+1}-b_{k}=2N/[(2N-k-1)(2N-k)]\),
\(1-a_{N-1}=1/(N+1)\) and \(1-b_{N-1}=2/(N+1)\).
Thus \(V_{N}\le\cdots\le V_{1}\le V_{0}=0\), and
\[h(x_{N})-h(x_{\star})\le\frac{L}{4N}\|d_{0}\|^{2}\le\frac{LR^{2}}{4N}.\]

}
\end{wideexamplebox}
\endgroup

\subsubsection{{Generated output: PGM with supplied SOS (repetition 1)}}
\begingroup
\tcbset{coltext=experimentboxtext,breakable,ignore nobreak=true,lines before break=4}
\begin{wideexamplebox}
{%
\textbf{Theorem}

Let \(h=f+g\), with \(f\) convex and differentiable with \(L\)-Lipschitz gradient, \(L>0\), and \(g\) closed, proper and convex. Suppose \(x_{\star}\) minimizes \(h\) and \(\|x_{0}-x_{\star}\|^{2}\le R^{2}\), \(R\ge0\). For any integer \(N\ge1\), proximal gradient with step \(1/L\) satisfies
\[h(x_{N})-h(x_{\star})\le\frac{L\|x_{0}-x_{\star}\|^{2}}{4N}\le\frac{LR^{2}}{4N}.\]
Here is a nonincreasing Lyapunov certificate. Write \(a_{k}=\nabla f(x_{k})\), \(a_{\star}=\nabla f(x_{\star})\), \(d_{k}=x_{k}-x_{\star}\) and \(D_{k}=2N-k\). Define \(V_{0}=0\) and
\[\begin{aligned}
&V_{k}=\frac{k+1}{D_{k}}[f(x_{k})-f(x_{\star})]+\frac{k}{D_{k}}[g(x_{k})-g(x_{\star})]-\frac{\langle a_{k},d_{k}\rangle}{D_{k}}\\
&\quad +\frac{\|a_{k}-a_{\star}\|^{2}}{2LD_{k}}+\frac{L(N-k)}{D_{k}^{2}}\|d_{k}\|^{2}-\frac{L}{4N}\|d_{0}\|^{2},\qquad 1\le k<N.
\end{aligned}\]
The terminal case is defined separately:
\[V_{N}=h(x_{N})-h(x_{\star})-\frac{L}{4N}\|d_{0}\|^{2}.\]
Then \(V_{N}\le V_{N-1}\le\cdots\le V_{0}=0\). The statement also covers \(N=1\) and \(R=0\); it requires no finite value of \(g(x_{0})\).

\textbf{Proof outline}

Use \(a_{k}=\nabla f(x_{k})\), \(s_{i}=L(x_{i-1}-x_{i})-a_{i-1}\in\partial g(x_{i})\), \(s_{\star}=-a_{\star}\), and \(d_{k}=x_{k}-x_{\star}\). The interpolation residuals are
\[I_{f}(u,v)=f(v)-f(u)+\langle\nabla f(v),u-v\rangle+\frac{\|\nabla f(u)-\nabla f(v)\|^{2}}{2L}\le0,\]
\[I_{g}(u,v)=g(v)-g(u)+\langle s_{v},u-v\rangle\le0.\]
Define \(D_{k}=2N-k\), \(V_{0}=0\), \(V_{N}=h(x_{N})-h(x_{\star})-L\|d_{0}\|^{2}/(4N)\), and, for \(1\le k<N\),
\[\begin{aligned}
&V_{k}=\frac{(k+1)[f(x_{k})-f(x_{\star})]+k[g(x_{k})-g(x_{\star})]-\langle a_{k},d_{k}\rangle}{D_{k}}\\
&\quad +\frac{\|a_{k}-a_{\star}\|^{2}}{2LD_{k}}+\frac{L(N-k)\|d_{k}\|^{2}}{D_{k}^{2}}-\frac{L\|d_{0}\|^{2}}{4N}.
\end{aligned}\]
For \(1\le i\le N\), set \(Q_{i}=\sum_{r=1}^{3}c_{i,r}\|u_{i,r}\|^{2}\), where
\[c_{i,1}=\frac{N}{(2N+1-i)(2N-i)L},\qquad u_{i,1}=\frac{i}{2N}a_{i}+\frac{2N-i}{2N}a_{i-1}-a_{\star},\]
\[\begin{aligned}
&c_{i,2}=\frac{1}{2L}\left[\frac{i-1}{2N-i}\frac{2N+2}{2N+1}+\frac{2N}{(2N-i)^{2}(2N+1)}\right],\qquad \\
&\quad u_{i,2}=a_{i-1}+s_{i}-\frac{L}{2N+1-i}d_{i-1},
\end{aligned}\]
\[\begin{aligned}
&c_{i,3}=\frac{i}{2N+1-i}\frac{2N}{2N+1}\frac{1}{2L},\qquad \\
&\quad u_{i,3}=\frac{2N+1}{2N}(a_{i}+s_{i})-\frac{a_{i-1}+s_{i}}{2N}-\frac{L}{2N-i}d_{i}.
\end{aligned}\]
Let
\[\alpha_{i}=\begin{cases}\dfrac{2N+1}{(2N-i)(2N+1-i)}&i<N,\\ \dfrac{1}{N+1}&i=N,\end{cases}\quad
\beta_{i}=\begin{cases}\dfrac{2N}{(2N-i)(2N+1-i)}&i<N,\\ \dfrac{2}{N+1}&i=N.\end{cases}\]
The exact identity for every step is
\[\begin{aligned}
&V_{i}-V_{i-1}=\mathbf{1}_{\{i=1\}}\frac{I_{f}(x_{\star},x_{0})}{2N}+\alpha_{i}I_{f}(x_{\star},x_{i})+\frac{i}{2N+1-i}I_{f}(x_{i-1},x_{i})\\
&\quad +\beta_{i}I_{g}(x_{\star},x_{i})+\frac{i-1}{2N+1-i}I_{g}(x_{i-1},x_{i})-Q_{i}.
\end{aligned}\]
The last \(I_{g}\) term is \textbf{omitted} for \(i=1\), avoiding an evaluation of \(g(x_{0})\). For \(N=1\), use the terminal multipliers and the initial residual together. The notebook verifies the interior, base, terminal, and \(N=1\) identities separately, in all scalar components with symbolic parameters.

All multipliers and square coefficients are nonnegative: \(2N-i\ge N>0\), \(2N+1-i\ge N+1>0\), \(L>0\), and \(i-1\ge0\). Thus each displayed right-hand side is nonpositive. Telescoping yields \(V_{N}\le V_{0}=0\), so \(h(x_{N})-h(x_{\star})\le L\|d_{0}\|^{2}/(4N)\le LR^{2}/(4N)\). The use of a supplied SOS candidate is justified by these exact identities, not by numerical agreement.

}
\end{wideexamplebox}
\endgroup

\subsubsection{{Additional generated output: PGM without supplied SOS (repetition 5)}}
\label{appendix:pgm_r05_output}
{This additional example is selected to illustrate the distinct
Lyapunov function using vectors from the next step.
The outputs from repetition~1 above retain the uniform selection rule.}
\begingroup
\tcbset{coltext=experimentboxtext,breakable,ignore nobreak=true,lines before break=4}
\begin{wideexamplebox}
{%
\textbf{Theorem}

Let \(f\) be convex and \(L\)-smooth, \(g\) closed proper convex, \(L>0\), \(R\ge0\), and \(x_{\star}\in\arg\min(f+g)\). Assume the stated proximal updates are well-defined and \(\|x_{0}-x_{\star}\|^2\le R^2\). For every integer \(N\ge1\), PGM satisfies
\[h(x_{N})-h(x_{\star})\le \frac{LR^2}{4N}.\]
Here is an explicit finite-product Lyapunov certificate. Let \(s_i\in\partial g(x_i)\) be the proximal subgradient for \(i\ge1\), and set
\[d_i=x_i-x_{\star},\quad p_i=\frac{\nabla f(x_i)-\nabla f(x_{\star})}{L},\quad q_i=\frac{s_i+\nabla f(x_{\star})}{L},\]
\[\begin{aligned}
&F_i=\frac{f(x_i)-f(x_{\star})-\langle\nabla f(x_{\star}),d_i\rangle}{L},\qquad \\
&\quad G_i=\frac{g(x_i)-g(x_{\star})+\langle\nabla f(x_{\star}),d_i\rangle}{L}.
\end{aligned}\]
Only \(F_{0}\), not \(G_{0}\) or \(s_{0}\), is used. We have \(d_{i+1}=d_i-p_i-q_{i+1}\).

For \(0\le j\le N-2\), write \(m_j=2N-j\) and define the explicit matrices
\[T_j=\begin{pmatrix}
2(j+1)(m_j-1)[(j+2)m_j(m_j-1)-(j+1)] &(j+1)^2(m_j-2)(2N+1)\\
4m_j^2(m_j-1)[(2j+3)m_j-(3j+4)]&2m_j^2(m_j-2)(j+1)(j+2)
\end{pmatrix}.\]
Define \(z_k\) by the finite ordered product
\[\binom{P_k}{Q_k}=T_kT_{k+1}\cdots T_{N-2}\binom{N}{2(N+1)^2},\qquad z_k=\frac{P_k}{Q_k},\qquad 0\le k\le N-1.\]
The empty product is the identity, so \(z_{N-1}=N/[2(N+1)^2]\). This formula has no recursively unspecified coefficients.

For \(1\le k\le N-1\), let \(m=2N-k\) and
\[\begin{aligned}
&a_k=\frac{k+1}{m},\quad b_k=\frac{k}{m},\quad A_k=\frac{N-k-1+(k+1)/(2m)-z_k}{(m-1)^2},\quad \\
&\quad B_k=\frac1{2m}+\frac{z_k}{m-1},\quad C_k=-\frac{k-1}{2m}-z_k.
\end{aligned}\]
Then the nonincreasing sequence has the full formula
\[\begin{aligned}
&V_k=L\bigl[-\frac{\|d_0\|^2}{4N}+a_kF_k+b_kG_k-\frac{a_k}{2}\|p_k\|^2-b_k\langle p_k,q_{k+1}\rangle\\
&\quad +A_k\|d_{k+1}\|^2+2B_k\langle d_{k+1},q_{k+1}\rangle+C_k\|q_{k+1}\|^2\bigr],
\end{aligned}\]
\[V_0=0,\qquad V_N=h(x_N)-h(x_{\star})-\frac{L}{4N}\|d_0\|^2.\]
There are no interior indices when \(N=1\). The certificate uses one-step-ahead vectors, as in the numerically discovered Block 4 basis.

\textbf{Proof outline}

Use the coefficient definitions and finite product in the theorem. The potential being proved is, explicitly,
\[\begin{aligned}
&\frac{V_k}{L}=-\frac{\|d_0\|^2}{4N}+a_kF_k+b_kG_k-\frac{a_k}{2}\|p_k\|^2-b_k\langle p_k,q_{k+1}\rangle\\
&\quad +A_k\|d_{k+1}\|^2+2B_k\langle d_{k+1},q_{k+1}\rangle+C_k\|q_{k+1}\|^2\quad(1\le k<N),
\end{aligned}\]
with \(V_0=0\) and \(V_N=h(x_N)-h(x_{\star})-L\|d_0\|^2/(4N)\).

The nonpositive interpolation residuals, normalized by \(L\), are
\[I_f(i,j)=\frac{f(x_j)-f(x_i)+\langle\nabla f(x_j),x_i-x_j\rangle+\|\nabla f(x_j)-\nabla f(x_i)\|^2/(2L)}{L},\]
\[I_g(i,j)=\frac{g(x_j)-g(x_i)+\langle s_j,x_i-x_j\rangle}{L}.\]
The star index means \(x_{\star}\) and \(s_{\star}=-\nabla f(x_{\star})\). Smooth convex interpolation gives \(I_f\le0\), and convexity gives \(I_g\le0\).

For \(0\le k\le N-2\), put \(m=2N-k\), \(w=z_{k+1}\), and
\[\begin{aligned}
&H_k=2(m-1)^2w+(k+1)(m-2),\quad \\
&\quad J_k=2(m-1)[(2k+3)m-(3k+4)]w+(m-2)(k+1)(k+2),
\end{aligned}\]
\[\begin{aligned}
&\alpha_k=\frac{H_k}{2(m-2)^2},\quad \beta_k=\frac{J_k}{2(m-1)H_k},\quad \\
&\quad \gamma_k=\frac{2(m-1)w+(k+1)(m-2)}{H_k},\quad \delta_k=\frac{k+1}{2m\beta_k}.
\end{aligned}\]
Define the complete square contribution
\[\begin{aligned}
&\mathcal Q_k=\alpha_k\left\|q_{k+2}-\frac{d_{k+1}}{m-1}+\gamma_k p_{k+1}\right\|^2\\
&\quad +\beta_k\left\|p_{k+1}+\delta_k\left(q_{k+1}-\frac{d_{k+1}}{m-1}\right)\right\|^2.
\end{aligned}\]
The matrix product gives exactly \(z_k=(k+1)/(2m)-(k+1)^2/(4m^2\beta_k)\). Substitution of the PGM updates yields, for every \(1\le k\le N-2\),
\[\begin{aligned}
&\frac{V_{k+1}-V_k}{L}=a_k I_f(k,k+1)+(a_{k+1}-a_k)I_f(\star,k+1)\\
&\quad +b_k I_g(k,k+1)+(b_{k+1}-b_k)I_g(\star,k+1)-\mathcal Q_k.
\end{aligned}\]

For \(N\ge2\), define
\[\mathcal Q_{\mathrm{init}}=\frac{\|p_0\|^2}{4N}+\frac{z_0}{(2N-1)^2}\|d_0-p_0-2Nq_1\|^2.\]
The exact base identity is
\[\begin{aligned}
&\frac{V_1}{L}=a_0I_f(\star,0)+a_0I_f(0,1)+(a_1-a_0)I_f(\star,1)\\
&\quad +b_1I_g(\star,1)-\mathcal Q_0-\mathcal Q_{\mathrm{init}}.
\end{aligned}\]
The terminal identity is
\[\begin{aligned}
&\frac{V_N-V_{N-1}}{L}=\frac{N}{N+1}I_f(N-1,N)+\frac1{N+1}I_f(\star,N)\\
&\quad +\frac{N-1}{N+1}I_g(N-1,N)+\frac2{N+1}I_g(\star,N)-\mathcal Q_{\mathrm{term}},
\end{aligned}\]
\[\mathcal Q_{\mathrm{term}}=\frac12\left\|p_N+\frac{N}{N+1}q_N-\frac{d_N}{N+1}\right\|^2.\]
For \(N=1\), the separate identity is
\[\begin{aligned}
&\frac{V_1}{L}=\tfrac12 I_f(\star,0)+\tfrac12 I_f(0,1)+\tfrac12 I_f(\star,1)+I_g(\star,1)\\
&\quad -\tfrac14\|p_0\|^2-\tfrac18\|d_0-p_0-2q_1\|^2-\tfrac12\|p_1+q_1/2-d_1/2\|^2.
\end{aligned}\]

All four entries of \(T_j\) are strictly positive on \(0\le j\le N-2\): here \(m_j\ge j+4\), so both bracketed factors are positive. The terminal vector has positive entries, hence every \(z_k>0\) and every product denominator is nonzero. Also \(H_k,J_k,\alpha_k,\beta_k>0\) and all displayed denominators are positive. The weights \(a_k,b_k\) are nonnegative, with
\[a_{k+1}-a_k=\frac{2N+1}{(2N-k)(2N-k-1)}>0,\qquad b_{k+1}-b_k=\frac{2N}{(2N-k)(2N-k-1)}>0.\]
Every right side is nonpositive. Thus \(V_N\le V_{N-1}\le\cdots\le V_1\le V_0=0\), including the separate \(N=1\) case. Rearranging the explicit terminal potential and using \(L>0\) and \(\|d_0\|^2\le R^2\) proves the theorem. The notebook below verifies every Gram coefficient, function-value coefficient, and offset in these identities exactly with symbolic \(N,k,L\).

}
\end{wideexamplebox}
\endgroup

{%
\subsection{OGM}\label{appendix:ogm_output}
}
\subsubsection{Input prompt}
\begin{commandbox}[coltext=experimentboxtext, listing options={style=experimentprompt}]
# Common assumptions

All cases use real finite-dimensional Euclidean spaces, positive `L` where present, nonnegative `R`, and integer horizons `N >= 1`. Proximal updates and stated evaluations are assumed well-defined. Supply these common assumptions and only the selected case subsection to its worker.


### `ogm`


- Problem: smooth convex minimization. `f` is convex and `L`-smooth.
- Initial condition: `grad f(x_star) = 0`, `||x_0 - x_star||^2 <= R^2`, and `z_0 = x_0`.
- Metric: `f(x_N) - f(x_star)`.
- Define `theta_{-1} = 0`; for `0 <= k < N`, `theta_k = (1 + sqrt(1 + 4*theta_{k-1}^2))/2`; set `theta_N = (1 + sqrt(1 + 8*theta_{N-1}^2))/2`.
- For `k = 0, ..., N-1`:

```text
y_{k+1} = x_k - grad f(x_k)/L
z_{k+1} = z_k - (2*theta_k/L)*grad f(x_k)
x_{k+1} = (1 - 1/theta_{k+1})*y_{k+1} + z_{k+1}/theta_{k+1}
```

- Conjectured rate: unknown. Preserve the special terminal `theta_N` recurrence.
\end{commandbox}
\subsubsection{{Generated output (repetition 1)}}
\begingroup
\tcbset{coltext=experimentboxtext,breakable,ignore nobreak=true,lines before break=4}
\begin{wideexamplebox}
{%
\textbf{Theorem}

Let \(f\) be convex and \(L\)-smooth on a finite-dimensional Euclidean space, with \(L>0\), \(\nabla f(x_{\star})=0\), and \(\|x_{0}-x_{\star}\|^{2}\leq R^{2}\), where \(R\geq0\). For every integer \(N\geq1\), run the algorithm above with its special terminal recurrence. Write \(g_{k}=\nabla f(x_{k})\), \(f_{\star}=f(x_{\star})\), \(T=\theta_{N}\), and \(a=x_{0}-x_{\star}\).

A nonincreasing certificate is \(V_{0}=0\), with
\[V_{k}=\frac{2\theta_{k}^{2}}{T^{2}}(f(x_{k})-f_{\star})-\frac{\theta_{k}^{2}}{LT^{2}}\|g_{k}\|^{2}+\frac{L}{2T^{2}}\left(\|z_{k+1}-x_{\star}\|^{2}-\|a\|^{2}\right),\quad 1\leq k<N,\]
and
\[V_{N}=f(x_{N})-f_{\star}+\frac{L}{2T^{2}}\left(\left\|z_{N}-x_{\star}-\frac{T}{L}g_{N}\right\|^{2}-\|a\|^{2}\right).\]
The interior range is empty when \(N=1\). These expressions satisfy \(V_{N}\leq\cdots\leq V_{0}=0\), and therefore
\[\boxed{f(x_{N})-f(x_{\star})\leq\frac{LR^{2}}{2\theta_{N}^{2}}.}\]

\textbf{Proof outline}

For evaluated points \(u,v\), write \(g(u)=\nabla f(u)\) and define the nonpositive smooth convex interpolation residual
\[\mathcal I(u,v)=f(v)-f(u)+\langle g(v),u-v\rangle+\frac{1}{2L}\|g(u)-g(v)\|^{2}\leq0.\]
Use \(\mathcal I_{i,j}=\mathcal I(x_{i},x_{j})\) and \(\mathcal I_{\star,j}=\mathcal I(x_{\star},x_{j})\).

To verify the displayed \(V_{k}\), define auxiliary expressions
\[\Phi_{k}=2\theta_{k}^{2}(f(x_{k})-f_{\star})-\frac{\theta_{k}^{2}}{L}\|g_{k}\|^{2}+\frac{L}{2}\|z_{k+1}-x_{\star}\|^{2},\quad 0\leq k<N,\]
\[\Phi_{N}=T^{2}(f(x_{N})-f_{\star})+\frac{L}{2}\left\|z_{N}-x_{\star}-\frac{T}{L}g_{N}\right\|^{2},\qquad A=\frac{L}{2}\|a\|^{2}.\]
Thus \(V_{k}=(\Phi_{k}-A)/T^{2}\) for \(1\leq k\leq N\), while \(V_{0}=0\).

The exact base, ordinary-step, and terminal-step identities are
\[\Phi_{0}-A=2\mathcal I_{\star,0},\]
\[\Phi_{j+1}-\Phi_{j}=2\theta_{j+1}\mathcal I_{\star,j+1}+2\theta_{j}^{2}\mathcal I_{j,j+1},\quad 0\leq j\leq N-2,\]
\[\Phi_{N}-\Phi_{N-1}=T\mathcal I_{\star,N}+2\theta_{N-1}^{2}\mathcal I_{N-1,N}.\]
The ordinary identity uses \(\theta_{j+1}^{2}-\theta_{j+1}=\theta_{j}^{2}\); the terminal identity uses \(T^{2}-T=2\theta_{N-1}^{2}\). These are the two different recurrences in the input.

In particular, for \(N\geq2\),
\[V_{1}-V_{0}=\frac{2\mathcal I_{\star,0}+2\theta_{1}\mathcal I_{\star,1}+2\mathcal I_{0,1}}{T^{2}};\]
for \(1\leq k\leq N-2\),
\[V_{k+1}-V_{k}=\frac{2\theta_{k+1}\mathcal I_{\star,k+1}+2\theta_{k}^{2}\mathcal I_{k,k+1}}{T^{2}};\]
and, when \(N\geq2\),
\[V_{N}-V_{N-1}=\frac{T\mathcal I_{\star,N}+2\theta_{N-1}^{2}\mathcal I_{N-1,N}}{T^{2}}.\]
For \(N=1\), \(T=2\) and the direct identity is
\[V_{1}=\frac{2\mathcal I_{\star,0}+2\mathcal I_{\star,1}+2\mathcal I_{0,1}}{4}.\]
Every multiplier is nonnegative: \(\theta_{0}=1\), the ordinary recurrence preserves \(\theta_{j}\geq1\), and the terminal recurrence gives \(T\geq2\). All divisions are valid since \(L>0\) and \(T>0\). Thus all displayed increments are nonpositive.

Finally, the terminal formula gives
\[f(x_{N})-f_{\star}=V_{N}+\frac{L}{2T^{2}}\|a\|^{2}-\frac{L}{2T^{2}}\left\|z_{N}-x_{\star}-\frac{T}{L}g_{N}\right\|^{2}\leq\frac{LR^{2}}{2T^{2}}.\]
This uses the nonnegative terminal square and the initial radius bound, and includes \(R=0\). The notebook below verifies each identity in all symbolic scalar components. The proof is local and valid for arbitrary horizon; a separate global analytic multiplier-table verification is not claimed.

}
\end{wideexamplebox}
\endgroup

{%
\subsection{BPPM}\label{appendix:bppm_output}
}
\subsubsection{Input prompt}
\begin{commandbox}[coltext=experimentboxtext, listing options={style=experimentprompt}]
# Common assumptions

All cases use real finite-dimensional Euclidean spaces, positive `L` where present, nonnegative `R`, and integer horizons `N >= 1`. Proximal updates and stated evaluations are assumed well-defined. Supply these common assumptions and only the selected case subsection to its worker.


### `bppm`


- Problem: Bregman proximal point minimization. `f` is closed, proper, convex, and differentiable relative to the Bregman kernel `h`; `h` is closed, proper, convex, and differentiable.
- Define `D_h(u,v) = h(u) - h(v) - <grad h(v), u-v>`.
- Parameters: fixed `alpha > 0` and `R >= 0`.
- Initial condition: `x_star` minimizes `f`, and `D_h(x_star, x_0) <= R`. Here the radius parameter is `R`, not `R^2`.
- Metric: `f(x_N) - f(x_star)`.
- For `k = 0, ..., N-1`, `x_{k+1} = prox^h_{alpha*f}(x_k)`, meaning an exact minimizer of `alpha*f(x) + D_h(x, x_k)`.
- PEPFlow interface: `f.bregman_prox(x_k, alpha, h)`.
- Conjectured rate provided in the input: `R/(alpha*N)`.
\end{commandbox}
\subsubsection{{Generated output (repetition 1)}}
\begingroup
\tcbset{coltext=experimentboxtext,breakable,ignore nobreak=true,lines before break=4}
\begin{wideexamplebox}
{%
\textbf{Theorem}

Under the stated Bregman proximal point assumptions, let \(\alpha>0\), \(R\ge0\), \(N\ge1\) be an integer, and \(D_{h}(x_{\star},x_{0})\le R\). No strong convexity of \(h\) or Lipschitz gradient assumption is needed. Define
\[V_{0}=0,\qquad V_{k}=\frac{\alpha k\bigl(f(x_{k})-f(x_{\star})\bigr)+D_{h}(x_{\star},x_{k})-D_{h}(x_{\star},x_{0})}{\alpha N},\quad 1\le k\le N.\]
Then \(V_{k+1}\le V_{k}\) and
\[0\le f(x_{N})-f(x_{\star})\le\frac{R}{\alpha N}.\]
The normalization is the supplied Bregman radius \(R\), not \(R^2\). The separate definition of \(V_{0}\) does not require evaluating \(f(x_{0})\).

\textbf{Proof outline}

Write \(I_{q}(u,v)=q(v)-q(u)+\langle\nabla q(v),u-v\rangle\le0\) for convex \(q\in\{f,h\}\). Proximal optimality gives \(\nabla h(x_{k+1})=\nabla h(x_{k})-\alpha\nabla f(x_{k+1})\).

Use \(V_{0}=0\) and \(V_{k}=k(f(x_{k})-f(x_{\star}))/N+[D_{h}(x_{\star},x_{k})-D_{h}(x_{\star},x_{0})]/(\alpha N)\) for \(k\ge1\). The initial step is
\[V_{1}-V_{0}=\frac{1}{N}I_{f}(x_{\star},x_{1})+\frac{1}{\alpha N}I_{h}(x_{1},x_{0})\le0.\]
For \(1\le k<N\), direct expansion and the proximal relation give the complete identity
\[\begin{aligned}
&V_{k+1}-V_{k}=\frac{1}{N}I_{f}(x_{\star},x_{k+1})+\frac{k}{N}I_{f}(x_{k},x_{k+1})\\
&\quad +\frac{k+1}{\alpha N}I_{h}(x_{k+1},x_{k})+\frac{k}{\alpha N}I_{h}(x_{k},x_{k+1})\le0.
\end{aligned}\]
All multipliers are nonnegative, so \(V_{N}\le V_{0}=0\). There are no additional squares or residual families in this exact local certificate. Finally,
\[f(x_{N})-f(x_{\star})-\frac{R}{\alpha N}=V_{N}+\frac{D_{h}(x_{\star},x_{0})-R}{\alpha N}+\frac{I_{h}(x_{\star},x_{N})}{\alpha N}\le0.\]
Here the last two terms are nonpositive by the initial condition and convexity, respectively. The lower bound follows because \(x_{\star}\) minimizes \(f\). When \(N=1\), the initial step and boundary identity suffice; the proof also covers \(R=0\).

}
\end{wideexamplebox}
\endgroup

{%
\subsection{FEG}\label{appendix:feg_output}
}
\subsubsection{Input prompt}
\begin{commandbox}[coltext=experimentboxtext, listing options={style=experimentprompt}]
# Common assumptions

All cases use real finite-dimensional Euclidean spaces, positive `L` where present, nonnegative `R`, and integer horizons `N >= 1`. Proximal updates and stated evaluations are assumed well-defined. Supply these common assumptions and only the selected case subsection to its worker.


### `feg`


- Problem: monotone operators. `A` is monotone and `L`-Lipschitz.
- Initial condition: `A(x_star) = 0` and `||x_0 - x_star||^2 <= R^2`.
- Metric: `||A(x_N)||^2`.
- For `k = 0, ..., N-1`:

```text
x_{k+1/2} = x_k + (x_0-x_k)/(k+1) - k*A(x_k)/((k+1)*L)
x_{k+1} = x_k + (x_0-x_k)/(k+1) - A(x_{k+1/2})/L
```

- At `k = 0`, tag `x_0` also as `x_{1/2}`; they are the same point. Do not create an independent intermediate iterate.
- Conjectured rate: unknown.
\end{commandbox}
\subsubsection{{Generated output (repetition 1)}}
\begingroup
\tcbset{coltext=experimentboxtext,breakable,ignore nobreak=true,lines before break=4}
\begin{wideexamplebox}
{%
\textbf{Theorem}

Let \(A\) be monotone and \(L\)-Lipschitz in a real finite-dimensional Euclidean space, with \(L>0\), \(A(x_{\star})=0\), and \(\|x_{0}-x_{\star}\|\le R\) for \(R\ge0\). For integer \(N\ge1\), use
\[\begin{aligned}
&x_{k+1/2}=x_{k}+\frac{x_{0}-x_{k}}{k+1}-\frac{k}{(k+1)L}A(x_{k}),\qquad
\\
&\quad x_{k+1}=x_{k}+\frac{x_{0}-x_{k}}{k+1}-\frac1L A(x_{k+1/2}),
\end{aligned}\]
for \(0\le k<N\), with \(x_{1/2}=x_{0}\).
Define the normalized potential
\[V_{0}=0,\qquad V_{k}=\frac{2k^{2}}{N^{2}}\|A(x_{k})\|^{2}
+\frac{4kL}{N^{2}}\langle A(x_{k}),x_{k}-x_{0}\rangle,\quad 1\le k\le N.\]
Then \(V_{k+1}\le V_{k}\), and
\[\boxed{\|A(x_{N})\|^{2}\le\frac{4L^{2}R^{2}}{N^{2}}.}\]
The potential need not be nonnegative. Its terminal identity, rather than positivity, supplies the residual bound. The result covers \(N=1\) and \(R=0\).

\textbf{Proof outline}

Set \(g_{k}=A(x_{k})\), \(d=x_{0}-x_{\star}\), and define nonpositive residuals
\[\begin{aligned}
&\operatorname{Mon}(u,v)=-\langle A(u)-A(v),u-v\rangle,\qquad
\\
&\quad \operatorname{Lip}(u,v)=\|A(u)-A(v)\|^{2}-L^{2}\|u-v\|^{2}.
\end{aligned}\]
For \(V_{k}=2k^{2}\|g_{k}\|^{2}/N^{2}+4kL\langle g_{k},x_{k}-x_{0}\rangle/N^{2}\) and \(V_{0}=0\), the shared initial half-step gives
\[V_{1}-V_{0}=\frac{2}{N^{2}}\operatorname{Lip}(x_{1},x_{0})\le0.\]
For every \(1\le k<N\), substitution of both updates gives the exact identity
\[V_{k+1}-V_{k}=\frac{4Lk(k+1)}{N^{2}}\operatorname{Mon}(x_{k},x_{k+1})
+\frac{2(k+1)^{2}}{N^{2}}\operatorname{Lip}(x_{k+1},x_{k+1/2})\le0.\]
There are no other local residual or square terms. All multipliers are positive on the stated domain, so \(V_{N}\le V_{0}=0\).
The terminal completion of the square is
\[\begin{aligned}
&\|g_{N}\|^{2}-\frac{4L^{2}R^{2}}{N^{2}}
=V_{N}+\frac{4L^{2}}{N^{2}}(\|d\|^{2}-R^{2})
+\frac{4L}{N}\operatorname{Mon}(x_{N},x_{\star})
\\
&\quad -\left\|g_{N}-\frac{2L}{N}d\right\|^{2}.
\end{aligned}\]
Every term on the right is nonpositive: use the initial radius bound, monotonicity with \(A(x_{\star})=0\), and nonnegativity of a squared norm. This proves the claimed rate. For \(N=1\) the interior telescoping sum is empty; the base and terminal identities still apply. No division by \(R\) or \(\|g_{N}\|\) is needed.

}
\end{wideexamplebox}
\endgroup

{%
\subsection{Dual-FEG}\label{appendix:dual_feg_output}
}
\subsubsection{Input prompt}
\begin{commandbox}[coltext=experimentboxtext, listing options={style=experimentprompt}]
# Common assumptions

All cases use real finite-dimensional Euclidean spaces, positive `L` where present, nonnegative `R`, and integer horizons `N >= 1`. Proximal updates and stated evaluations are assumed well-defined. Supply these common assumptions and only the selected case subsection to its worker.


### `dual_feg`


- Problem: monotone operators. `A` is monotone and `L`-Lipschitz.
- Initial condition: `A(x_star) = 0`, `||x_0 - x_star||^2 <= R^2`, and `z_0 = 0`.
- Parameters: fixed horizon `N >= 1` and `alpha = 1/L`.
- Performance metric: unknown.
- For `k = 0, ..., N-1`:

```text
x_{k+1/2} = x_k - alpha*z_k - alpha*A(x_k)
x_{k+1} = x_{k+1/2} - ((N-k-1)/(N-k))*alpha*(A(x_{k+1/2})-A(x_k))
z_{k+1} = ((N-k-1)/(N-k))*z_k - A(x_{k+1/2})/(N-k)
```

- Conjectured rate: unknown.
\end{commandbox}
\subsubsection{{Generated output (repetition 1)}}
\begingroup
\tcbset{coltext=experimentboxtext,breakable,ignore nobreak=true,lines before break=4}
\begin{wideexamplebox}
{%
\textbf{Theorem}

For every integer \(N\ge1\), \(L>0\) and \(R\ge0\), the stated Dual FEG iteration for a monotone \(L\)-Lipschitz operator with \(A(x_{\star})=0\), \(\|x_{0}-x_{\star}\|^2\le R^2\) and \(z_{0}=0\) satisfies
\[\boxed{\|A(x_{N})\|^2\le\frac{4L^2R^2}{N^2}.}\]
Set \(G=A(x_{N})\), \(a=x_{0}-x_{N}\), and \(b_{k}=x_{k}-x_{N}\). A certificate is \(V_{0}=0\) and, for \(1\le k\le N-1\),
\[\begin{aligned}
&V_{k}=-\frac{4L}{N}\langle a,G\rangle-2\|z_{k}\|^2+\frac{4L}{N-k}\langle z_{k},b_{k}\rangle\\
&\quad -4\langle z_{k},G\rangle+\frac{4L}{N-k}\langle b_{k},G\rangle.
\end{aligned}\]
At the terminal index use the separate expression
\[V_{N}=2\|G\|^2-\frac{4L}{N}\langle a,G\rangle.\]
The sequence is nonincreasing. The interior expression also gives \(V_{0}=0\) when evaluated at \(k=0\), \(z_{0}=0\); it is never evaluated at \(k=N\).

\textbf{Proof outline}

Let \(y_{k}=x_{k+1/2}\) and define the nonpositive residuals
\[\begin{aligned}
&\operatorname{Lip}(u,v)=\|A(u)-A(v)\|^2-L^2\|u-v\|^2,\qquad
\\
&\quad \operatorname{Mon}(u,v)=-\langle A(u)-A(v),u-v\rangle.
\end{aligned}\]
With \(G=A(x_{N})\) and \(a=x_{0}-x_{N}\), write the interior potential as
\[\begin{aligned}
&V_{k}=-\frac{4L}{N}\langle a,G\rangle-2\|z_{k}\|^2-4\langle z_{k},G\rangle
\\
&\quad +\frac{4L}{N-k}\langle z_{k}+G,x_{k}-x_{N}\rangle.
\end{aligned}\]
Its initial value is zero. Substitution of the algorithm gives, for \(0\le k\le N-2\),
\[V_{k+1}-V_{k}=\frac{2}{(N-k)^2}\operatorname{Lip}(x_{k},y_{k})
+\frac{4L}{(N-k)(N-k-1)}\operatorname{Mon}(x_{N},y_{k})\le0.\]
The actual final update is \(x_{N}=x_{N-1}-(z_{N-1}+A(x_{N-1}))/L\). With the separate terminal potential \(V_{N}=2\|G\|^2-(4L/N)\langle a,G\rangle\), it gives
\[V_{N}-V_{N-1}=2\operatorname{Lip}(x_{N-1},x_{N})\le0.\]
For \(N=1\) there are no ordinary steps; the same last-step identity with \(z_{0}=0\) starts directly from \(V_{0}=0\). All multipliers are positive on their stated domains, so \(V_{N}\le0\) for every \(N\ge1\).

Finally, set \(d=x_{0}-x_{\star}\). The exact normalized identity is
\[\begin{aligned}
&\|G\|^2-\frac{4L^2R^2}{N^2}
=V_{N}+\frac{4L}{N}\operatorname{Mon}(x_{N},x_{\star})
\\
&\quad -\left\|\frac{2L}{N}d-G\right\|^2
+\frac{4L^2}{N^2}\bigl(\|d\|^2-R^2\bigr).
\end{aligned}\]
Every term on the right is nonpositive: use \(V_{N}\le0\), monotonicity and \(A(x_{\star})=0\), the square\textquotesingle s nonnegativity, and the given initial-distance bound. This proves the theorem, including \(R=0\). The notebook below derives the coefficients and checks every identity exactly in Gram, function-value and constant coordinates.

}
\end{wideexamplebox}
\endgroup

\IfFileExists{appendices/case_study.tex}{%

\section{Proofs for the case studies}
\label{app:case-study-proofs}

This appendix records the complete Lyapunov certificates behind
Theorems~\ref{thm:case-ofgm-primary}--\ref{thm:case-sfgm}
and~\ref{thm:case-fista}.
We used \texttt{GPT-6 Astra} with \texttt{Ultra} reasoning effort
throughout these case studies. The corresponding input prompt is recorded
at the start of each proof subsection.

{Throughout this appendix, cited code cells refer to exact algebraic
verifications in the companion notebook specified for the corresponding proof.
Cell numbers are the 1-based positions among code cells, excluding Markdown
and raw cells. After restarting the kernel and executing each code cell once
from top to bottom, these numbers agree with Jupyter's \texttt{In[\(n\)]} labels.}

\begin{table}[htbp]
\centering
{%
\caption{Theorem correspondences and companion notebooks.}
\label{tab:case-study-notebooks}
\begin{tabular*}{\linewidth}{@{\extracolsep{\fill}}lll@{}}
\toprule
Appendix~\ref{app:case-study-proofs} & Main text & Companion notebook \\
\midrule
Theorem~\ref{cs:sfgm:thm:main} & Theorem~\ref{thm:case-sfgm} & \texttt{fgm\_rational\_tight\_example\_lyap.ipynb} \\
Theorem~\ref{cs:ofgm-primary:thm:main} & Theorem~\ref{thm:case-ofgm-primary} & \texttt{fgm\_conjecture4\_example\_lyap.ipynb} \\
Theorem~\ref{cs:ofgm-secondary:thm:main} & Theorem~\ref{thm:case-ofgm-secondary} & \texttt{fgm\_conjecture5\_example\_lyap.ipynb} \\
Theorem~\ref{cs:ogm:thm:main} & Theorem~\ref{thm:case-ogm} & \texttt{ogm\_conjecture4\_example\_lyap.ipynb} \\
Theorem~\ref{cs:fista:thm:main} & Theorem~\ref{thm:case-fista} & \texttt{fista\_rational\_tight\_example\_lyap.ipynb} \\
\bottomrule
\end{tabular*}
}
\end{table}

We use the Hilbert space \(\mathcal H\) and the function class
\(\mathcal F_L\) defined in Appendix~\ref{app:notation}.
{Thus \(f:\mathcal H\to\mathbb R\), and the iterates
\(x_k,y_k,z_k\), the minimizer \(x_\star\), and the gradients
\(\nabla f(x_k)\) belong to \(\mathcal H\).
The parameters satisfy \(L,R,\theta_k\in\mathbb R\).
For each trajectory, the potentials \(V_k\) and square remainders
\(S_k\) defined below take values in \(\mathbb R\).}
For
\(p,q\in\mathcal H\), we use the smooth convex
{interpolation residual \(\mathcal I_f:\mathcal H\times\mathcal H\to\mathbb R\)}
\begin{equation}
 \mathcal I_f(p,q)
 =f(q)-f(p)+\csinner{\nabla f(q)}{p-q}
  +\frac{1}{2L}\norm{\nabla f(p)-\nabla f(q)}^2\leq0.
 \label{eq:case-interpolation-residual}
\end{equation}
For FGM and OGM, set \(\theta_{-1}=0\).  Their parameter
recurrence gives {the following identity for \(k\geq0\) in FGM
and \(0\leq k<N\) in OGM}:\footnote{{In standard OGM,
\(\theta_N=(1+\sqrt{1+8\theta_{N-1}^2})/2\) instead. The \(y_N\)
conjecture proved here uses only \(\theta_0,\ldots,\theta_{N-1}\).}}
\begin{equation}
 \theta_k^2-\theta_k=\theta_{k-1}^2.
 \label{eq:case-theta-identity}
\end{equation}

{In the proofs for FGM, define the scalars \(C_k\in\mathbb R\) by}
{%
\begin{equation}
 C_0=0,\qquad
 C_k=\frac{1}{\theta_{k-1}^2}\sum_{j=0}^{k-1}\theta_j^3,
 \qquad k\geq1.
 \label{eq:case-C-definition}
\end{equation}
}
{The definition of \(C_k\) gives}
\[
 \theta_k^2C_{k+1}
 =\theta_{k-1}^2C_k+\theta_k^3,\qquad k\geq0.
\]
Combining this with \eqref{eq:case-theta-identity} yields the recurrence
{used in the proofs below, with \(\delta_k\in\mathbb R\):}
\begin{equation}
 C_{k+1}=\theta_k+\left(1-\frac1{\theta_k}\right)C_k,
 \qquad {\delta_k}=C_{k+1}-C_k.
 \label{eq:case-C-recurrence}
\end{equation}

{The correspondence between these FGM coefficients and the original
conjecture notation is verified in Appendix~\ref{app:fgm-conjecture-relation}.}

The expanded certificate theorem in each subsection proves an upper bound.
Tightness follows by combining these bounds with the matching lower bounds
established by \citet[Section~4.2]{taylor2017smooth} for FGM,
\citet[Proposition~4.1]{KimFessler2017_convergence} for OGM, and
\citet[Table~1]{TaylorHendrickxGlineur2017_exacta} for FGM with rational
coefficients.

The proofs of Theorems~\ref{thm:case-ofgm-primary}--\ref{thm:case-ogm}
follow the same structure as that of Theorem~\ref{thm:case-sfgm}.
We therefore present Theorem~\ref{thm:case-sfgm} first: its rational
coefficient sequence makes the calculations comparatively simple and gives
the clearest introduction to the common argument.

\subsection{Proof of Theorem~\ref{thm:case-sfgm} (Table~1 conjecture in \texorpdfstring{\citet{TaylorHendrickxGlineur2017_exacta}}{Taylor et al. (2017)})}
\label{app:case-sfgm-proof}

\subsubsection{Input prompt}

\begin{commandbox}
/pep-implement

Function: f is convex and L-smooth
Parameters: L, R
Initial condition: ||x_0 - x_star|| <= R, where grad f(x_star) = 0
Performance metric: f(y_N) - f(x_star)
Algorithm: Fast gradient method with rational coefficients and fixed step size 1/L

Parameter sequence: theta_k = (k+2)/2 for k >= 0
Initialization: x_0 = y_0 = z_0
For k = 0,...,N-1:
y_{k+1} = x_k - (1/L) * grad f(x_k)
z_{k+1} = z_k - (theta_k/L) * grad f(x_k)
x_{k+1} = (1 - 1/theta_{k+1}) * y_{k+1} + (1/theta_{k+1}) * z_{k+1}

Conjectured rate: f(y_N) - f(x_star) <= 2 * L * R^2 / ((N+2) * (N+3))
\end{commandbox}
Then, we repeatedly entered:
\begin{commandbox}
proceed with the next step
\end{commandbox}

\begin{proof}[Proof outline for Theorem~\ref{thm:case-sfgm}]
Set \(\tau_N=2/[(N+2)(N+3)]{\in\mathbb R}\).
{For \(1\leq k<N\), define \(\boldsymbol w_k\in\mathcal H^2\) by}
\[
 \boldsymbol w_k=
 \begin{bmatrix}z_k-x_\star\\x_k-x_\star\end{bmatrix}.
\]
{Set \(V_0=0\in\mathbb R\). For \(N\geq2\), define \(V_1\) by}
\begin{equation}
\begin{aligned}
 V_1={}&a_1\bigl(f(x_1)-f_\star\bigr)
 -\tau_NL\norm{x_0-x_\star}^2
 +\frac{a_1}{2L}\norm{\nabla f(x_1)}^2
 \\
 &-b_1\csinner{\nabla f(x_1)}{x_1-x_\star}
 +\frac{L\tau_N(1-4\tau_N)}{(1-2\tau_N)^2}
 \norm{x_1-x_\star}^2
\end{aligned}
\label{cs:sfgm:eq:V-first}
\end{equation}
{For \(2\leq k<N\), with \(\vd_k\in\mathbb R^2\) specified below,
define the interior potential \(V_k\) by}
\begin{align}
\begin{aligned}
 V_k={}&a_k\bigl(f(x_k)-f(x_\star)\bigr)
 -\tau_NL\norm{x_0-x_\star}^2
 +\frac{a_k}{2L}\norm{\nabla f(x_k)}^2\\
 &-2\csinner{\nabla f(x_k)}{{\vd_k}^\top\boldsymbol w_k}
 +L\boldsymbol w_k^\top {\vQ_k}\boldsymbol w_k.
\end{aligned}
\label{cs:sfgm:eq:V-interior}
\end{align}
{At the reported point, set}
\begin{align}
    V_N=f(y_N)-f(x_\star)-\tau_NL\norm{x_0-x_\star}^2
    +\frac1{2L}\norm{\nabla f(y_N)-Lb_N(y_N-x_\star)}^2.
\label{cs:sfgm:eq:V-end}
\end{align}
We will show that each consecutive difference, or decrement, of $V_k$ is a positive-weighted sum of interpolation residuals minus at most two nonnegative squares.
{For \(N\ge2\), we show}
{%
\begin{equation}\label{cs:sfgm:eq:decrements}
 V_{k+1}-V_k=
 \begin{cases}
  \displaystyle b_0\mathcal I_f(x_\star,x_0)+a_0\mathcal I_f(x_0,x_1)+b_1\mathcal I_f(x_\star,x_1)-S_0, & k=0,\\
  \displaystyle a_k\mathcal I_f(x_k,x_{k+1})+b_{k+1}\mathcal I_f(x_\star,x_{k+1})-S_k, & 1\le k<N-1,\\
  \displaystyle a_{N-1}\mathcal I_f(x_{N-1},y_N)+b_N\mathcal I_f(x_\star,y_N)-S_{N-1}, & k=N-1,
 \end{cases}
\end{equation}
}
for some positive numbers $a_k, b_k$ and sum-of-squares $S_k{\in\mathbb R}$ with $S_k\ge0$.\footnote{{For \(N=1\), the first line holds with \(x_1\) replaced by \(y_1\) and that is the terminal identity.}}
{Together with the coefficient signs in Lemma~\ref{cs:sfgm:lem:signs}, these identities imply} \(V_N\leq V_0=0\), and dropping the nonnegative square term in $V_N$ yields \eqref{eq:case-sfgm-rate}.
Below, we state a separate theorem which fully characterizes the Lyapunov analysis and prove it in detail.
\end{proof}

\begin{theorem}[{Theorem~\ref{thm:case-sfgm} with full details of the Lyapunov function}]
\label{cs:sfgm:thm:main}
\normalfont
Let \(f\in\mathcal F_L\), let \(x_\star{\in\mathcal H}\) minimize \(f\), and write
\(f_\star=f(x_\star){\in\mathbb R}\). Suppose that \(\norm{x_0-x_\star}\leq R\).
{%
For a fixed integer horizon \(N\geq1\), set
\[
 \tau_N=\frac{2}{(N+2)(N+3)}\in\mathbb R.
\]
}
{For \(0\leq k<N\), define the real scalars \(D_k,a_k,b_k\in\mathbb R\) by}
\begin{align}
 D_k&=2N^2+10N+4-k^2-7k, \notag\\
 a_k&=\frac{(k+1)\{2N(N+1)(k+4)-k(k+2)(k+3)\}}
 {N(N+1)D_k},
 \qquad a_{-1}=0,\qquad a_N=1.\label{cs:sfgm:eq:a}\\
 b_k&=a_k-a_{k-1},\qquad 0\leq k\leq N.\label{cs:sfgm:eq:b}
\end{align}
{For \(1\leq k<N\), define the state \(\boldsymbol w_k\in\mathcal H^2\),
the matrix \(M_k\in\mathbb R^{2\times2}\), and the coefficient vectors
\(\vc_k,\vd_k\in\mathbb R^2\) by}
\begin{align*}
 \boldsymbol w_k&=
 \begin{bmatrix}z_k-x_\star\\x_k-x_\star\end{bmatrix},
 &{M_k}&=\begin{pmatrix}1&0\\2/(k+3)&(k+1)/(k+3)\end{pmatrix},\notag\\
 {\vc_k}&=
 \begin{bmatrix}(k+2)/2\\(2k+3)/(k+3)\end{bmatrix},
 &{\vd_k}&=\frac12
 \begin{bmatrix}2a_{k-1}/k\\a_k-(k+2)a_{k-1}/k\end{bmatrix}.
\end{align*}
{For \(2\leq k<N\), define \(\boldsymbol\xi_k,\boldsymbol\zeta_k\in\mathbb R^2\) by}
\begin{align*}
 \boldsymbol\xi_k&=
 \begin{bmatrix}
  1-\tau_Nk(k+3)/2\\
  1-\tau_Nk(k+7)/4
 \end{bmatrix},\notag\\
 \boldsymbol\zeta_k&=\tau_N
 \begin{bmatrix}
  2a_{k-1}/k+1-(k+1)(k+2)/[N(N+1)]\\
  (k+1)(k+2)/[N(N+1)]-(k+2)a_{k-1}/k
 \end{bmatrix}.
\end{align*}
The same expressions are used at adjacent indices when they occur below.
{For \(2\leq k<N\), define an auxiliary matrix
\(\overline{\vQ}_k\in\mathbb R^{2\times2}\) satisfying
\(\overline{\vQ}_k\boldsymbol\xi_k=\boldsymbol\zeta_k\) by}
{%
\begin{equation*}
 \overline{\vQ}_k=\frac{(\boldsymbol\zeta_k)_1}{(\boldsymbol\xi_k)_2}
 \begin{pmatrix}
  0&1\\[4pt]
  1&\dfrac{(\boldsymbol\zeta_k)_2}{(\boldsymbol\zeta_k)_1}
  -\dfrac{(\boldsymbol\xi_k)_1}{(\boldsymbol\xi_k)_2}
 \end{pmatrix}.
\end{equation*}
}
{With \(\boldsymbol e_1=[1,0]^\top\in\mathbb R^2\), define
\(\boldsymbol n_k\in\mathbb R^2\) and \(\vQ_k\in\mathbb R^{2\times2}\) by}
{%
\begin{align}
 \boldsymbol n_k&=
 \begin{bmatrix}1\\-(\boldsymbol\xi_k)_1/(\boldsymbol\xi_k)_2\end{bmatrix},\notag\\
 \vQ_k&=\overline{\vQ}_k+
 \frac{(\vd_{k-1})_1
 -\boldsymbol e_1^\top M_{k-1}^\top\overline{\vQ}_k\vc_{k-1}}
 {(\boldsymbol e_1^\top M_{k-1}^\top\boldsymbol n_k)
  (\boldsymbol n_k^\top\vc_{k-1})}
 \boldsymbol n_k\boldsymbol n_k^\top.
 \label{cs:sfgm:eq:Q}
\end{align}
}
{%
Define the Lyapunov values \(V_k\in\mathbb R\) as in
\eqref{cs:sfgm:eq:V-first}, \eqref{cs:sfgm:eq:V-interior}, and
\eqref{cs:sfgm:eq:V-end}, \ie, set \(V_0=0\).
For \(N\ge2\), define \(V_1\) by
\begin{equation*}
\begin{aligned}
 V_1={}&a_1\bigl(f(x_1)-f(x_\star)\bigr)
 -\tau_NL\norm{x_0-x_\star}^2
 +\frac{a_1}{2L}\norm{\nabla f(x_1)}^2
 \\
 &-b_1\csinner{\nabla f(x_1)}{x_1-x_\star}
 +\frac{L\tau_N(1-4\tau_N)}{(1-2\tau_N)^2}
 \norm{x_1-x_\star}^2
\end{aligned}
\end{equation*}
For \(2\le k<N\), define \(V_k\) by
\begin{equation*}
\begin{aligned}
 V_k={}&a_k\bigl(f(x_k)-f(x_\star)\bigr)
 -\tau_NL\norm{x_0-x_\star}^2
 +\frac{a_k}{2L}\norm{\nabla f(x_k)}^2\\
 &-2\csinner{\nabla f(x_k)}{{\vd_k}^\top\boldsymbol w_k}
 +L\boldsymbol w_k^\top {\vQ_k}\boldsymbol w_k,
\end{aligned}
\end{equation*}
At the reported point, set
\begin{equation*}
    V_N=f(y_N)-f(x_\star)-\tau_NL\norm{x_0-x_\star}^2
    +\frac1{2L}\norm{\nabla f(y_N)-Lb_N(y_N-x_\star)}^2.
\end{equation*}
}

{For \(2\leq k<N\), define the vectors \(s_k,t_k\in\mathcal H\) by}
{%
\begin{equation*}
 s_k=\frac{\nabla f(x_k)}L
 -\frac{2\tau_N}{(\boldsymbol\xi_k)_2}(x_k-x_\star),
 \qquad
 t_k=\boldsymbol n_k^\top\boldsymbol w_k.
\end{equation*}
}
{The boundary square terms are defined below, with \(\alpha_1\in\mathbb R\):}
\begin{align}
 S_0&=\frac{L\tau_N}{(1-2\tau_N)^2}
 \norm{\frac{\nabla f(x_0)}L-2\tau_N(x_0-x_\star)}^2,\notag\\
 S_1&=L\alpha_1
 \norm{\frac{\nabla f(x_1)}L
 -\frac{2\tau_N}{1-2\tau_N}(x_1-x_\star)}^2,
 \qquad
 \alpha_1=
 \frac{(N^2+5N+2)(5N^2+5N-6)}
 {N(N+1)(N^2+5N-2)^2},
 \label{cs:sfgm:eq:first-slacks}
\end{align}
where the second line is used when \(N\geq3\).
{For \(2\leq k\leq N-2\), define \(\alpha_k,\beta_k\in\mathbb R\) and \(S_k\) by}
\begin{align}
 \alpha_k&=a_k-{\vc_k}^\top {\vQ_{k+1}}{\vc_k}
 =\frac{a_k(D_k+8)}{2D_k},\notag\\
 \beta_k&=({\vQ_k}-{M_k}^\top {\vQ_{k+1}}{M_k})_{11},
 \qquad
 S_k=L\alpha_k\norm{s_k}^2+L\beta_k\norm{t_k}^2.
 \label{cs:sfgm:eq:interior-slack}
\end{align}
{For \(N\geq3\), define the terminal coefficients
\(\alpha_{N-1},\beta_{N-1}\in\mathbb R\) and square \(S_{N-1}\) by}
{%
\begin{align}
 \alpha_{N-1}&=1-b_N-\frac{b_N^2}{2},\qquad
 \beta_{N-1}=(\vQ_{N-1})_{11}
 -\frac{a_{N-2}^2}{(N-1)^2\alpha_{N-1}},\notag\\
 S_{N-1}&=L\alpha_{N-1}
 \norm{s_{N-1}-\frac{a_{N-2}}{(N-1)\alpha_{N-1}}t_{N-1}}^2
 +L\beta_{N-1}\norm{t_{N-1}}^2,
 \label{cs:sfgm:eq:last-slack}
\end{align}
}
while for \(N=2\), we instead set
\begin{equation}
 S_1=\frac{11L}{18}
 \norm{\frac{\nabla f(x_1)}L-\frac{x_1-x_\star}{4}}^2.
 \label{cs:sfgm:eq:N-two-slack}
\end{equation}
{%
The square sums satisfy \(S_k\ge0\). For \(N\ge2\), the following decrement identities hold:\footnote{For \(N=1\), use the first case with \(x_1\) replaced by \(y_1\); this is the single base--terminal step.}
\begin{equation*}
 V_{k+1}-V_k=
 \begin{cases}
  \displaystyle b_0\mathcal I_f(x_\star,x_0)+a_0\mathcal I_f(x_0,x_1)+b_1\mathcal I_f(x_\star,x_1)-S_0, & k=0,\\
  \displaystyle a_k\mathcal I_f(x_k,x_{k+1})+b_{k+1}\mathcal I_f(x_\star,x_{k+1})-S_k, & 1\le k<N-1,\\
  \displaystyle a_{N-1}\mathcal I_f(x_{N-1},y_N)+b_N\mathcal I_f(x_\star,y_N)-S_{N-1}, & k=N-1.
 \end{cases}
\end{equation*}
The coefficient signs in Lemma~\ref{cs:sfgm:lem:signs} and the nonpositive interpolation residuals imply \(V_N\le\cdots\le V_0=0\).
Dropping the nonnegative square in \(V_N\) and applying the radius assumption gives
}
\[
 f(y_N)-f_\star
 \leq\tau_NL\norm{x_0-x_\star}^2
 \leq\frac{2LR^2}{(N+2)(N+3)}.
\]
\end{theorem}

\begin{proof}
{%
The companion notebook is \texttt{fgm\_rational\_tight\_example\_lyap.ipynb}, in
\texttt{examples\_peppy/fgm\_rational\_tight/}.
{The notebook verifies the same certificate using \(\vQ_k\) and
the dimensionless constant \(\tau_N\), with different auxiliary notation.} The exact verification section
starts at Cell~31 and can be run independently of the numerical discovery
blocks. Its symbolic identities hold for arbitrary parameters on the
stated domains. Cell~36 provides an additional check of the complete
potential decrease identities at \(N=1,2,3,4,6,8\).
}

{%
For \(2\le k\le N-2\), subtracting two consecutive instances of \eqref{cs:sfgm:eq:V-interior} cancels the common initial-radius term and gives
\begin{equation}\label{cs:sfgm:eq:potential-difference}
\begin{aligned}
V_{k+1}-V_k
 &=a_{k+1}\bigl(f(x_{k+1})-f(x_\star)\bigr)-a_k\bigl(f(x_k)-f(x_\star)\bigr)\\
 &\quad+\frac{a_{k+1}}{2L}\norm{\nabla f(x_{k+1})}^2
 -2\csinner{\nabla f(x_{k+1})}{\vd_{k+1}^\top\boldsymbol w_{k+1}}
 +L\boldsymbol w_{k+1}^\top\vQ_{k+1}\boldsymbol w_{k+1}\\
 &\quad-\frac{a_k}{2L}\norm{\nabla f(x_k)}^2
 +2\csinner{\nabla f(x_k)}{\vd_k^\top\boldsymbol w_k}
 -L\boldsymbol w_k^\top\vQ_k\boldsymbol w_k.
\end{aligned}
\end{equation}
Our goal is to identify this expression with the interior case of \eqref{cs:sfgm:eq:decrements}. We first factor the quadratic remainder, then match the weighted interpolation residuals; the boundary cases are treated separately.
}

\medskip
\noindent\textbf{{Step 1: Expressing the remainder as a sum of squares.}}
For \(2\leq k\leq N-2\), Steps 1 and 2 prove the second line of
\eqref{cs:sfgm:eq:decrements}, namely,
\[
 V_{k+1}-V_k
 =a_k\mathcal I_f(x_k,x_{k+1})
 +b_{k+1}\mathcal I_f(x_\star,x_{k+1})-S_k,
\]
with \(S_k\) as defined in \eqref{cs:sfgm:eq:interior-slack}.
We first factor the quadratic remainder into the two squares defining
\(S_k\); Step 2 identifies this remainder in the Lyapunov difference.
The FGM-rational updates give ({notebook Cell~32})
\begin{equation}
 \boldsymbol w_{k+1}={M_k}\boldsymbol w_k
 -{\vc_k}\frac{\nabla f(x_k)}L.
 \label{cs:sfgm:eq:state-transition}
\end{equation}
{Define the symmetric matrix \(\vS_k\in\mathbb R^{3\times3}\) in the coordinates \((\boldsymbol w_k,\nabla f(x_k)/L)\) by}
\begin{equation}
 {\vS_k}=
 \begin{pmatrix}
 {\vQ_k}-{M_k}^\top {\vQ_{k+1}}{M_k}&
 -{\vd_k}+{M_k}^\top {\vQ_{k+1}}{\vc_k}\\
 (-{\vd_k}+{M_k}^\top {\vQ_{k+1}}{\vc_k})^\top&
 a_k-{\vc_k}^\top {\vQ_{k+1}}{\vc_k}
 \end{pmatrix}.
 \label{cs:sfgm:eq:slack-matrix}
\end{equation}
{Quadratic forms in these coordinates use inner products of the component vectors.}
The definition of \({\vQ_k}\) gives the following two identities, verified in
{notebook Cell~33}:
\[
 {\vQ_k}\boldsymbol\xi_k=\boldsymbol\zeta_k,
 \qquad
 \bigl(-{\vd_{k-1}}
 +{M_{k-1}}^\top {\vQ_k}{\vc_{k-1}}\bigr)_1=0.
\]
The same cell verifies all three transport identities by direct substitution:
\begin{align}
 {M_k}\boldsymbol\xi_k-2\tau_N{\vc_k}
 &=\boldsymbol\xi_{k+1},\notag\\
 \boldsymbol\zeta_k-{M_k}^\top\boldsymbol\zeta_{k+1}
 -2\tau_N{\vd_k}&=0,\notag\\
 -{\vd_k}^\top\boldsymbol\xi_k
 +{\vc_k}^\top\boldsymbol\zeta_{k+1}
 +2\tau_Na_k&=0.
 \label{cs:sfgm:eq:transport}
\end{align}
Using \({\vQ_j}\boldsymbol\xi_j=\boldsymbol\zeta_j\) at \(j=k,k+1\)
together with \eqref{cs:sfgm:eq:transport}, we find that \({\vS_k}\)
annihilates \([(\boldsymbol\xi_k)_1,(\boldsymbol\xi_k)_2,2\tau_N]^\top\).
{The cross-term identity at index \(k+1\) and the diagonal entries
identify the following coefficients of the quadratic form of \(\vS_k\) in \((\boldsymbol w_k,\nabla f(x_k)/L)\):
\[
\begin{aligned}
 \norm{(\boldsymbol w_k)_1}^2&:\quad (\vS_k)_{11}=\beta_k,\\
 \frac2L\csinner{\nabla f(x_k)}{(\boldsymbol w_k)_1}&:\quad
 (\vS_k)_{13}=0,\\
 \frac1{L^2}\norm{\nabla f(x_k)}^2&:\quad
 (\vS_k)_{33}=\alpha_k.
\end{aligned}
\]
Together with the kernel relation, these coefficients determine the matrix as}
{%
\[
 \vS_k=\alpha_k
 \begin{bmatrix}0\\-2\tau_N/(\boldsymbol\xi_k)_2\\1\end{bmatrix}
 \begin{bmatrix}0\\-2\tau_N/(\boldsymbol\xi_k)_2\\1\end{bmatrix}^{\!\top}
 +\beta_k
 \begin{bmatrix}1\\-(\boldsymbol\xi_k)_1/(\boldsymbol\xi_k)_2\\0\end{bmatrix}
 \begin{bmatrix}1\\-(\boldsymbol\xi_k)_1/(\boldsymbol\xi_k)_2\\0\end{bmatrix}^{\!\top}.
\]
}
{Evaluating the associated quadratic form and using the definitions of \(s_k,t_k\) gives}
{%
\[
 L\begin{bmatrix}\boldsymbol w_k\\\nabla f(x_k)/L\end{bmatrix}^{\!\top}\vS_k\begin{bmatrix}\boldsymbol w_k\\\nabla f(x_k)/L\end{bmatrix}
 =L\alpha_k\norm{s_k}^2+L\beta_k\norm{t_k}^2
 =S_k,
\]
}
The \((3,3)\) entry also yields the displayed formula for
\(\alpha_k\) in \eqref{cs:sfgm:eq:interior-slack}.
{Notebook Cell~33} verifies the kernel,
the zero \((1,3)\) entry, every entry of this matrix factorization, and
the displayed rational formula for \(\alpha_k\).

\medskip
\noindent\textbf{{Step 2: Matching the remainder to the Lyapunov difference.}}
{It remains to show that the quadratic form evaluated in Step 1 is the difference between the weighted interpolation residuals and \(V_{k+1}-V_k\).} {%
To match \eqref{cs:sfgm:eq:potential-difference}, expand the weighted residuals. Since \(b_{k+1}=a_{k+1}-a_k\) and \(\nabla f(x_\star)=0\), their function-value terms agree with the potential difference:
\begin{equation}\label{cs:sfgm:eq:residual-expansion}
\begin{aligned}
&a_k\mathcal I_f(x_k,x_{k+1})+b_{k+1}\mathcal I_f(x_\star,x_{k+1})\\
 &\quad=a_{k+1}\bigl(f(x_{k+1})-f(x_\star)\bigr)-a_k\bigl(f(x_k)-f(x_\star)\bigr)\\
 &\qquad+\frac{a_k}{2L}\norm{\nabla f(x_k)}^2
 +\frac{a_{k+1}}{2L}\norm{\nabla f(x_{k+1})}^2\\
 &\qquad+\csinner{\nabla f(x_{k+1})}
 {a_k\left(x_k-x_{k+1}-\frac{\nabla f(x_k)}L\right)+b_{k+1}(x_\star-x_{k+1})}.
\end{aligned}
\end{equation}
The remaining inner-product coefficient follows from \(y_{k+1}=x_k-\nabla f(x_k)/L\) and the mixing update:
\begin{equation}\label{cs:sfgm:eq:residual-vector}
\begin{aligned}
&a_k\left(x_k-x_{k+1}-\frac{\nabla f(x_k)}L\right)+b_{k+1}(x_\star-x_{k+1})\\
 &\quad=a_k(y_{k+1}-x_\star)-a_{k+1}(x_{k+1}-x_\star)\\
 &\quad=-\frac{2a_k}{k+1}(z_{k+1}-x_\star)
 +\left(\frac{(k+3)a_k}{k+1}-a_{k+1}\right)(x_{k+1}-x_\star)
 =-2\vd_{k+1}^\top\boldsymbol w_{k+1}.
\end{aligned}
\end{equation}
Subtracting \eqref{cs:sfgm:eq:potential-difference} now cancels all function-value terms and all terms involving \(\nabla f(x_{k+1})\). The remaining expression is
\[
 L\bigl(\boldsymbol w_k^\top\vQ_k\boldsymbol w_k
 -\boldsymbol w_{k+1}^\top\vQ_{k+1}\boldsymbol w_{k+1}\bigr)
 +\frac{a_k}{L}\norm{\nabla f(x_k)}^2
 -2\csinner{\nabla f(x_k)}{\vd_k^\top\boldsymbol w_k}.
\]
}
Substituting \eqref{cs:sfgm:eq:state-transition} into the remaining
quadratic terms gives
\[
\begin{aligned}
 &a_k\mathcal I_f(x_k,x_{k+1})
 +b_{k+1}\mathcal I_f(x_\star,x_{k+1})-(V_{k+1}-V_k)\\
 &\quad{=L\begin{bmatrix}\boldsymbol w_k\\\nabla f(x_k)/L\end{bmatrix}^{\!\top}\vS_k\begin{bmatrix}\boldsymbol w_k\\\nabla f(x_k)/L\end{bmatrix}=S_k,}
\end{aligned}
\]
where the last equality is the factorization proved in Step 1.
Rearranging proves the second line of \eqref{cs:sfgm:eq:decrements}
for \(2\leq k\leq N-2\). The case \(k=1\) uses the separate
potential \eqref{cs:sfgm:eq:V-first} and is treated below. The signs
needed to deduce \(V_{k+1}\leq V_k\) are established in {Lemma~\ref{cs:sfgm:lem:signs}.}
{Notebook Cell~37} verifies this cancellation in every quadratic,
function-value, and constant coefficient, including the specialization
to the first interior step.

\medskip
\noindent\textbf{Boundary indices.}
We now verify the remaining cases of \eqref{cs:sfgm:eq:decrements}:
the base step, the first interior step \(k=1\) for \(N\geq3\), and the terminal step.
Using the boundary potentials \eqref{cs:sfgm:eq:V-first} and
\eqref{cs:sfgm:eq:V-end}, we identify their quadratic remainders with
{the stated \(S_0,S_1,S_{N-1}\).}
At \(k=1\), for \(N\geq3\),
\(\boldsymbol w_1=[x_1-x_\star,x_1-x_\star]^\top\).
Substitution of \({\vQ_2}\) leaves the rank-one quadratic
\[
 L\alpha_1
 \norm{\frac{\nabla f(x_1)}L
 -\frac{2\tau_N}{1-2\tau_N}(x_1-x_\star)}^2,
\]
which is \(S_1\) in \eqref{cs:sfgm:eq:first-slacks}.
{Notebook Cell~35} verifies this rank-one factorization and the stated
formula for \(\alpha_1\).

At \(k=N-1\), for \(N\geq3\), use
\(y_N-x_\star=[0,1]\boldsymbol w_{N-1}
-\nabla f(x_{N-1})/L\) and the scalar \(b_N^2/2\) from the terminal
square. The resulting block has the same kernel as
\eqref{cs:sfgm:eq:slack-matrix}, with \((1,3)\) entry \({-\frac{a_{N-2}}{N-1}}\),
\((3,3)\) entry \(\alpha_{N-1}\), and \((1,1)\) entry
\(({\vQ_{N-1}})_{11}\). Completing the square gives
\eqref{cs:sfgm:eq:last-slack}.
{Notebook Cell~35} verifies the terminal block, its kernel, and the
completed-square factorization; {Cell~39} verifies the full terminal
decrement identity, including the function values.

For the base step, substituting
\(a_0=2\tau_N/(1-2\tau_N)\) and
\(x_1=x_0-\nabla f(x_0)/L\) gives \(S_0\) in
\eqref{cs:sfgm:eq:first-slacks} ({notebook Cells~35 and~38}).\footnote{{When \(N=1\),
\(a_0=b_0=b_1=1/2\), and the same calculation uses the terminal
definition of \(V_1\) ({Cell~40}). When \(N=2\),
\(\tau_N=1/10\), \(a_0=1/4\), \(a_1=2/3\), and \(b_2=1/3\);
the remaining quadratic is \eqref{cs:sfgm:eq:N-two-slack} ({Cell~41}).}}

\medskip
\noindent\textbf{Positivity of the coefficients.}
{%
Lemma~\ref{cs:sfgm:lem:signs} proves the coefficient signs and the
bounds needed for the defining denominators. Hence every \(S_k\geq0\),
and the nonpositive interpolation residuals in
\eqref{cs:sfgm:eq:decrements} give \(V_{k+1}\leq V_k\).
}

{%
\medskip
\noindent\textbf{Conclusion.}
Thus \(V_N\leq\cdots\leq V_0=0\). Dropping the nonnegative square in
\eqref{cs:sfgm:eq:V-end} and using \(\norm{x_0-x_\star}\leq R\)
proves the claimed bound.
}

\end{proof}

\begin{lemma}%
\label{cs:sfgm:lem:signs}
\normalfont
{%
The coefficients are positive on their respective index ranges:\footnote{{The unused value \(a_{-1}\) is zero.}}
\[
 a_k>0,\qquad b_k>0,\qquad \alpha_k>0,\qquad \beta_k>0.
\]
The remaining boundary-square coefficients in \eqref{cs:sfgm:eq:first-slacks} and \eqref{cs:sfgm:eq:N-two-slack} are also positive.
}
\end{lemma}
\begin{proof}
{%
For \(0\leq k<N\),
\(D_k\geq D_{N-1}=N^2+5N+10>0\).
{For \(N\geq3\), use the polynomials \(p_4,p_9:\mathbb R\to\mathbb R\) defined by}
\begin{align*}
 p_4(N)&=N^4+10N^3+37N^2+60N-12,\\
 p_9(N)&=N^9+26N^8+280N^7+1696N^6+6325N^5+13538N^4\\
 &\qquad+11402N^3-4844N^2-2904N+5456.
\end{align*}
{For the following expansions, define \(\ell:=N-3\in\mathbb R_{\geq0}\):}
\begin{align*}
 p_4(3+\ell)&=\ell^4+22\ell^3+181\ell^2+660\ell+852,\\
 p_9(3+\ell)&=\ell^9+53\ell^8+1228\ell^7+16396\ell^6+139291\ell^5
 +780011\ell^4\\
 &\qquad+2867792\ell^3+6635500\ell^2+8715168\ell+4933568.
\end{align*}
Both polynomials are positive for \(\ell\geq0\).

\begin{itemize}
\item \textbf{Positivity of \(a_k\).}
The coefficient formula \eqref{cs:sfgm:eq:a} gives the initial value and the
successive differences (for \(1\leq k<N\)):
\[
 D_0a_0=8,\qquad
 D_ka_k-D_{k-1}a_{k-1}
 =\frac{4(k+2)(N-k)(N+k+1)}{N(N+1)}>0.
\]
Since \(D_k>0\) and \(D_{k-1}-D_k=2(k+3)>0\), these identities imply
that \(a_k>0\) for \(0\leq k<N\) and that \(a_k\) is strictly increasing
through \(k=N-1\). The remaining value is \(a_N=1\).
{%
For \(2\le k<N\), \((k+1)(k+2)\le N(N+1)\), so
\[
 (\boldsymbol\zeta_k)_1
 =\tau_N\left(\frac{2a_{k-1}}{k}+1-\frac{(k+1)(k+2)}{N(N+1)}\right)
 \ge\frac{2\tau_Na_{k-1}}{k}>0.
\]
}
{The remaining denominators in \eqref{cs:sfgm:eq:Q} are nonzero because of the
following factorizations, also verified in {notebook Cell~34}:}
\begin{align*}
 (\boldsymbol\xi_k)_2
 &=\frac{D_k+8}{2(N+2)(N+3)}>0,\notag\\
 (\boldsymbol e_1^\top {M_{k-1}}^\top\boldsymbol n_k)
 (\boldsymbol n_k^\top\vc_{k-1})
 &=\frac{k^2(k-1)D_{k-2}D_{k-1}}
 {2(k+2)^2(D_k+8)^2}>0.
\end{align*}

\item \textbf{Positivity of \(b_k\).}
For \(1\leq k<N\), strict increase gives \(b_k=a_k-a_{k-1}>0\).
At the endpoints,
\[
 b_0=a_0>0,\qquad b_N=1-a_{N-1}=\frac8{N^2+5N+10}>0.
\]
Notebook Cell~34 verifies the identities for \(D_k,a_k,b_k\)
used in these two items.

\item \textbf{Positivity of \(\alpha_k\).}
For \(N\geq3\), the numerator factors in \(\alpha_1\) in
\eqref{cs:sfgm:eq:first-slacks} are positive. For \(2\leq k\leq N-2\),
\eqref{cs:sfgm:eq:interior-slack} gives
\[
 \alpha_k=\frac{a_k(D_k+8)}{2D_k}>0.
\]
At the terminal index,
\[
 \alpha_{N-1}=\frac{p_4(N)}{(N^2+5N+10)^2}>0.
\]
Notebook Cells~33 and~35 verify the interior and boundary formulas.

\item \textbf{Positivity of \(\beta_k\).}
{%
For \(2\leq k\leq N-2\), {set \(m=k-2\in\mathbb R\) and \(\ell=N-k-2\in\mathbb R\).}
Substitution into \eqref{cs:sfgm:eq:Q} gives
\begin{equation*}
 \beta_k=\frac{2(D_k+8)P_{m,\ell}}
 {Nk^2(N+1)(k-1)D_{k-2}^2D_{k-1}^2},
\end{equation*}
{where \(P_{m,\ell}\in\mathbb R\) is given by}
\begin{align*}
P_{m,\ell}={}&15m^7
+m^6(140\ell+753)
+m^5(446\ell^2+5214\ell+14835)\\
&+m^4(624\ell^3+12008\ell^2+74836\ell+150455)\\
&+m^3(532\ell^4+13176\ell^3+126946\ell^2+545318\ell+861038)\\
&+m^2(240\ell^5+7988\ell^4+101880\ell^3+658392\ell^2
+2157668\ell+2814280)\\
&+m\bigl(40\ell^6+2072\ell^5+37696\ell^4+336856\ell^3\\
&\qquad+1663624\ell^2+4416768\ell+4903840\bigr)\\
&+112\ell^6+3920\ell^5+54576\ell^4+395056\ell^3\\
&+1620512\ell^2+3649344\ell+3529344.
\end{align*}
Every coefficient is positive and \(m,\ell\ge0\), so \(\beta_k>0\) on the interior range.
{Notebook Cell~33} verifies the rational expression for \(\beta_k\)
against \eqref{cs:sfgm:eq:Q}, checks the displayed polynomial \(P_{m,\ell}\)
coefficient by coefficient, and checks its positive coefficients and constant term.
}

For \(N\geq3\), the terminal coefficient satisfies
\[
 \beta_{N-1}=
 \frac{(N^2+5N+18)^2p_9(N)}
 {(N-2)(N-1)^2(N^2+7N+14)^2(N^2+9N+16)^2p_4(N)}>0.
\]
Notebook Cell~35 verifies both terminal rational expressions and
both shifted-polynomial expansions above, including their coefficient signs.
\end{itemize}
Since \(0<\tau_N\leq1/6\), the coefficient of \(S_0\) in
\eqref{cs:sfgm:eq:first-slacks} is positive.\footnote{{When \(N=2\), the
coefficient \(11L/18\) in \eqref{cs:sfgm:eq:N-two-slack} is positive as well.}}
}
\end{proof}

\subsection{Proof of Theorem~\ref{thm:case-ofgm-primary} (Conjecture~4 in \texorpdfstring{\citet{taylor2017smooth}}{Taylor et al. (2017)})}
\label{app:case-ofgm-primary-proof}

For this case, we used the following initial prompt.

\begin{commandbox}
/pep-implement

Function: f is convex and L-smooth
Parameters: L, R
Initial condition: ||x_0 - x_star|| <= R, where grad f(x_star) = 0
Performance metric: f(y_N) - f(x_star)
Algorithm: Fast gradient method with fixed step size 1/L

Parameter sequence: theta_0 = 1 and theta_{k+1} = (1 + sqrt(1 + 4 * theta_k^2)) / 2 for k >= 0
Initialization: x_0 = y_0 = z_0
For k = 0,...,N-1:
y_{k+1} = x_k - (1/L) * grad f(x_k)
z_{k+1} = z_k - (theta_k/L) * grad f(x_k)
x_{k+1} = (1 - 1/theta_{k+1}) * y_{k+1} + (1/theta_{k+1}) * z_{k+1}

Conjectured rate: f(y_N) - f(x_star) <= L * R^2 / (2 * (2 * C_N + 1)), where C_N = 1 + sum_{j=0}^{N-2} h_{N-1,j} and x_i = x_0 - (1/L) * sum_{j=0}^{i-1} h_{i,j} * grad f(x_j)
\end{commandbox}
Then, we repeatedly entered:
\begin{commandbox}
proceed with the next step
\end{commandbox}

The initial proof of Conjecture~4 required a lengthy and intricate argument to establish positivity of the coefficients.
We therefore asked the agent to revise it using the two-square decomposition
in the FGM-rational proof (Appendix~\ref{app:case-sfgm-proof}), with the
additional prompt below. This led to formulas for the square coefficients
and a simpler positivity argument, as presented in the proof below.

\colorlet{csadditionalpromptcolor}{.}
\begin{commandbox}[coltext=csadditionalpromptcolor,
 listing options={basicstyle=\ttfamily\small\color{csadditionalpromptcolor},
 columns=fullflexible,breaklines=true,breakatwhitespace=true,breakindent=0pt,breakautoindent=false,keepspaces=true,showstringspaces=false,
 escapeinside={(*@}{@*)}}]
Could you explore whether allowing two squared-norm terms per iteration, as in (*@\texttt{fgm\_rational\_tight\_example\_lyap.ipynb}@*), can give explicit formulas for the square coefficients and simplify their positivity proof for FGM Conjecture 4?
\end{commandbox}

\begin{proof}[Proof outline for Theorem~\ref{thm:case-ofgm-primary}]
Set $D_k=2C_N+1-C_k\in\mathbb R$ and
$\tau_N=1/(2D_0)\in\mathbb R$. Define the Lyapunov function by
\(V_{0}=0\) and
\begin{equation}\label{cs:ofgm-primary:potential}
\begin{aligned}
 V_k&=a_k\bigl(f(x_k)-f(x_\star)\bigr)-\tau_NL\norm{x_0-x_\star}^2
      -\frac{a_k}{2L}\norm{\nabla f(x_k)}^2+L\boldsymbol w_k^\top \vQ_k\boldsymbol w_k
      &&(1\le k<N).
\end{aligned}
\end{equation}

At the last iterate, set
\begin{equation}\label{eq:case-ofgm-primary-lyapunov-end}
 V_N=f(y_N)-f(x_\star)-\tau_NL\norm{x_0-x_\star}^2.
\end{equation}
For $N\ge2$, the exact decrement identities are\footnote{For $N=1$, only the first line is used, with $x_1$ replaced by $y_1$,
$a_0=b_0=b_1=1/2$, and $S_0$ as specified in the footnote to
\eqref{cs:ofgm-primary:basesq}, and $V_1 = V_N$ is given by \eqref{eq:case-ofgm-primary-lyapunov-end}.
For $N=2$, only the first and last lines are used.}
\begin{equation}\label{cs:ofgm-primary:decrements}
 V_{k+1}-V_k=
 \begin{cases}
  \displaystyle b_0\mathcal I_f(x_\star,x_0)+a_0\mathcal I_f(x_0,x_1)+b_1\mathcal I_f(x_\star,x_1)-S_0, & k=0,\\
  \displaystyle a_k\mathcal I_f(x_k,x_{k+1})+b_{k+1}\mathcal I_f(x_\star,x_{k+1})-S_k, & 1\le k<N-1,\\
  \displaystyle a_{N-1}\mathcal I_f(x_{N-1},y_N)+b_N\mathcal I_f(x_\star,y_N)-S_{N-1}, & k=N-1.
 \end{cases}
\end{equation}

Each $S_k$ is the sum of two squares with positive coefficients.
Theorem~\ref{cs:ofgm-primary:thm:main} below specifies the Lyapunov function coefficients and $S_k$.
Its proof verifies the identities \eqref{cs:ofgm-primary:decrements}, and Lemma~\ref{cs:ofgm-primary:lem:signs} establishes the coefficient signs and hence $S_k\ge0$.
Together with the fact that the interpolation inequalities $\mathcal I_f$ are nonpositive,
\eqref{cs:ofgm-primary:decrements} therefore gives $V_N\le\cdots\le V_0=0$ and consequently shows \eqref{eq:case-ofgm-primary-rate}.
\end{proof}

\begin{theorem}[{Theorem~\ref{thm:case-ofgm-primary} with full details of the Lyapunov function}]
\label{cs:ofgm-primary:thm:main}
\normalfont
Let \(f\in\mathcal F_L\), let \(x_\star\in\mathcal H\) minimize \(f\), and suppose \(\norm{x_0-x_\star}\le R\),
where $R\ge0$. Fix an integer $N\ge1$. Run~\ref{eq:case-theta-recursive} from $x_0=y_0=z_0$.
Define
\begin{equation*}
 D_k=2C_N+1-C_k\in\mathbb R\quad(0\le k\le N),\qquad
 \tau_N=\frac1{2(2C_N+1)}=\frac1{2D_0}\in\mathbb R.
\end{equation*}
Set $a_{-1}=0$ and $a_N=1$. For $0\leq k<N$, define \(a_k\in\mathbb R\), and for \(0\leq k\leq N\), define \(b_k\in\mathbb R\), by
\begin{equation}\label{cs:ofgm-primary:weights}
\begin{aligned}
 a_k&=\frac{\sum_{i=0}^{k}\theta_i(1-\theta_{i-1}^2/\theta_{N-1}^2)}{D_{k+1}}=\frac{\theta_k^2(\theta_{N-1}^2-\theta_k^2+C_{k+1})}
 {\theta_{N-1}^2D_{k+1}},\\
 b_k&=a_k-a_{k-1}.
\end{aligned}
\end{equation}

For $0\le k<N$, define the auxiliary scalar $\sigma_k\in\mathbb R$ by
\begin{equation}\label{cs:ofgm-primary:sigma}
 \sigma_k=\frac{a_k}{\theta_k^2}-\frac{a_{k-1}}{\theta_{k-1}^2}.
\end{equation}

The second term in $\sigma_0$ is defined to be zero. In particular,
$a_0=b_0=1/(D_0-1)$ and $b_N=1/(C_N+1)$.
For \(1\le k<N\), define the state \(\boldsymbol w_k\in\mathcal H^2\)
and the matrix \(\vQ_k\in\mathbb R^{2\times2}\) by
\begin{equation}\label{cs:ofgm-primary:Q}
 \boldsymbol w_k=\begin{bmatrix}z_{k+1}-x_\star\\y_{k+1}-z_{k+1}\end{bmatrix},\qquad
 \vQ_k=\frac12\begin{pmatrix}
 \dfrac{1-a_k}{D_0-\theta_k^2}&0\\[6pt]
 0&\dfrac{\theta_k^2/\theta_{N-1}^2-a_k}{\theta_k^2-C_{k+1}}
 \end{pmatrix}.
\end{equation}
Define the Lyapunov values \(V_k\in\mathbb R\) as in \eqref{cs:ofgm-primary:potential} and \eqref{eq:case-ofgm-primary-lyapunov-end}, \ie, set \(V_{0}=0\) and
\begin{equation*}
\begin{aligned}
 V_k&=a_k\bigl(f(x_k)-f(x_\star)\bigr)-\tau_NL\norm{x_0-x_\star}^2
      -\frac{a_k}{2L}\norm{\nabla f(x_k)}^2+L\boldsymbol w_k^\top \vQ_k\boldsymbol w_k
      &&(1\le k<N).
\end{aligned}
\end{equation*}
At the reported point, set
\begin{equation*}
 V_N=f(y_N)-f(x_\star)-\tau_NL\norm{x_0-x_\star}^2.
\end{equation*}

For \(1\le k<N-1\), define the vectors \(s_k,t_k\in\mathcal H\) by
\begin{equation}\label{cs:ofgm-primary:directions}
\begin{aligned}
 s_k&=z_{k+1}-x_\star-(D_0-\theta_k^2)\frac{\nabla f(x_{k+1})}{L},\\
 t_k&=\frac{\nabla f(x_{k+1})}{L}-\frac{y_{k+1}-z_{k+1}}{\theta_k^2-C_{k+1}}.
\end{aligned}
\end{equation}
For \(N\ge2\), define \(s_{N-1},t_{N-1}\in\mathcal H\) by the same
formulas with \(k=N-1\) and \(\nabla f(x_{k+1})\) replaced by
\(\nabla f(y_N)\).

For \(0\le k<N-1\), define \(\alpha_k\in\mathbb R\) by
\begin{equation}\label{cs:ofgm-primary:alpha}
 \alpha_k=
 \frac{(D_0b_{k+1}-\theta_{k+1})(D_0-\theta_{k+1}^2)
       +(D_0a_{k+1}-\theta_{k+1}^2)\theta_{k+1}}
 {2D_0(D_0-\theta_{k+1}^2)(D_0-\theta_k^2)}.
\end{equation}
For \(1\le k<N-1\), define \(\beta_k\in\mathbb R\) by
\begin{equation}\label{cs:ofgm-primary:beta}
 \beta_k=\frac{\theta_k^2(\theta_k^2-C_{k+1})}{2}
 \left[\sigma_{k+1}+\frac{\theta_{k+1}-1}{\theta_{k+1}^2-C_{k+2}}
 \left(\frac1{\theta_{N-1}^2}-\frac{a_{k+1}}{\theta_{k+1}^2}\right)\right].
\end{equation}
At the terminal index define \(\alpha_{N-1},\beta_{N-1}\in\mathbb R\) by
\begin{equation}\label{cs:ofgm-primary:terminalcoeff}
 \alpha_{N-1}=\frac{b_N}{2(D_0-\theta_{N-1}^2)},\qquad
 \beta_{N-1}=\frac{b_N(\theta_{N-1}^2-C_N)}2.
\end{equation}
For \(1\le k<N\), set
\begin{equation}\label{cs:ofgm-primary:interiorsq}
 S_k=L\alpha_k\norm{s_k}^2+L\beta_k\norm{t_k}^2.
\end{equation}

The initial term is, for $N\ge2$,\footnote{When $N=1$, only \eqref{cs:ofgm-primary:basesq} is used: replace $x_1$ by $y_1$ and set
$D_0=3$ and $\alpha_0=1/8$. No expression containing
$(\theta_0^2-C_1)^{-1}$ is evaluated.}
\begin{equation}\label{cs:ofgm-primary:basesq}
\begin{aligned}
 S_0={}&\frac{LD_0}{2(D_0-1)^2}
   \left\|\frac{\nabla f(x_0)}L-\frac{x_0-x_\star}{D_0}\right\|^2
   +L\alpha_0\left\|z_1-x_\star-(D_0-1)\frac{\nabla f(x_1)}L\right\|^2.
\end{aligned}
\end{equation}

The square sums satisfy \(S_k\ge0\). For \(N\ge2\), the following decrement identities hold:\footnote{For \(N=1\), use the first case with \(x_1\) replaced by \(y_1\), \(a_0=b_0=b_1=1/2\), and the convention for \(S_0\) above. For \(N=2\), the middle range is empty.}
\begin{equation*}
 V_{k+1}-V_k=
 \begin{cases}
  \displaystyle b_0\mathcal I_f(x_\star,x_0)+a_0\mathcal I_f(x_0,x_1)+b_1\mathcal I_f(x_\star,x_1)-S_0, & k=0,\\
  \displaystyle a_k\mathcal I_f(x_k,x_{k+1})+b_{k+1}\mathcal I_f(x_\star,x_{k+1})-S_k, & 1\le k<N-1,\\
  \displaystyle a_{N-1}\mathcal I_f(x_{N-1},y_N)+b_N\mathcal I_f(x_\star,y_N)-S_{N-1}, & k=N-1.
 \end{cases}
\end{equation*}
The coefficient signs in Lemma~\ref{cs:ofgm-primary:lem:signs} and the nonpositive interpolation residuals imply \(V_N\le\cdots\le V_0=0\).
The terminal definition and the radius assumption therefore give

\begin{equation}\label{cs:ofgm-primary:bound}
 f(y_N)-f(x_\star)\le \frac{LR^2}{2(2C_N+1)}
 =\frac{LR^2\theta_{N-1}^2}
 {2\bigl(2\sum_{k=0}^{N-1}\theta_k^3+\theta_{N-1}^2\bigr)}.
\end{equation}

\end{theorem}

\begin{proof}
The companion notebook is
\path{examples_peppy/fgm_conjecture4/fgm_conjecture4_example_lyap.ipynb}.
The notebook also documents
the smaller prerequisite set for running just the proof checks.

For \(1\le k<N-1\), subtracting two consecutive instances of \eqref{cs:ofgm-primary:potential} cancels the common initial-radius term and gives
\begin{equation}\label{cs:ofgm-primary:potential-difference}
\begin{aligned}
V_{k+1}-V_k
 &=a_{k+1}\bigl(f(x_{k+1})-f(x_\star)\bigr)-a_k\bigl(f(x_k)-f(x_\star)\bigr)\\
 &\quad-\frac{a_{k+1}}{2L}\norm{\nabla f(x_{k+1})}^2
 +\frac{a_k}{2L}\norm{\nabla f(x_k)}^2\\
 &\quad+L\boldsymbol w_{k+1}^\top\vQ_{k+1}\boldsymbol w_{k+1}
 -L\boldsymbol w_k^\top\vQ_k\boldsymbol w_k.
\end{aligned}
\end{equation}
Our goal is to identify this expression with the interior case of \eqref{cs:ofgm-primary:decrements}. We first factor the quadratic remainder, then match the weighted interpolation residuals; the boundary cases are treated separately.

\textbf{Step 1: factorization of the interior slack.}
For $1\le k<N-1$, Steps 1 and 2 prove the second line of
\eqref{cs:ofgm-primary:decrements}, namely,
\[
 V_{k+1}-V_k
 =a_k\mathcal I_f(x_k,x_{k+1})
 +b_{k+1}\mathcal I_f(x_\star,x_{k+1})-S_k,
\]
with $S_k$ as defined in \eqref{cs:ofgm-primary:interiorsq}.
We first factor a quadratic remainder into the two squares defining $S_k$;
Step 2 identifies this remainder in the Lyapunov difference.
For this index, the algorithm gives
\begin{equation}\label{cs:ofgm-primary:state}
\begin{aligned}
 \boldsymbol w_{k+1}
 &=\begin{pmatrix}1&0\\0&\theta_k^2/\theta_{k+1}^2\end{pmatrix}\boldsymbol w_k
   -\begin{bmatrix}\theta_{k+1}\\1-\theta_{k+1}\end{bmatrix}\frac{\nabla f(x_{k+1})}{L},\\
 x_{k+1}-x_\star
 &=z_{k+1}-x_\star+\frac{\theta_k^2}{\theta_{k+1}^2}(y_{k+1}-z_{k+1}).
\end{aligned}
\end{equation}

Define the symmetric matrix \(\vS_k\in\mathbb R^{3\times3}\) by
\begin{equation}\label{cs:ofgm-primary:slack-matrix}
\begin{aligned}
 \vS_k={}&\begin{pmatrix}
 \vQ_k&-\frac12\begin{bmatrix}b_{k+1}\\\theta_k^2\sigma_{k+1}\end{bmatrix}\\[3pt]
 -\frac12\begin{bmatrix}b_{k+1}\\\theta_k^2\sigma_{k+1}\end{bmatrix}^{\!\top}&a_{k+1}
 \end{pmatrix}\\
 &-\begin{pmatrix}1&0&-\theta_{k+1}\\0&\theta_k^2/\theta_{k+1}^2&\theta_{k+1}-1\end{pmatrix}^{\!\top}
 \vQ_{k+1}\begin{pmatrix}1&0&-\theta_{k+1}\\0&\theta_k^2/\theta_{k+1}^2&\theta_{k+1}-1\end{pmatrix}.
\end{aligned}
\end{equation}
By \eqref{cs:ofgm-primary:state}, the last matrix is the transition matrix.
Substituting \eqref{cs:ofgm-primary:Q}, \eqref{cs:ofgm-primary:alpha}, and
\eqref{cs:ofgm-primary:beta} gives
\[
 \vS_k=
 \begin{bmatrix}
  1&0\\
  0&-1/(\theta_k^2-C_{k+1})\\
  -(D_0-\theta_k^2)&1
 \end{bmatrix}
 \begin{pmatrix}\alpha_k&0\\0&\beta_k\end{pmatrix}
 \begin{bmatrix}
  1&0\\
  0&-1/(\theta_k^2-C_{k+1})\\
  -(D_0-\theta_k^2)&1
 \end{bmatrix}^{\!\top}.
\]
Evaluating the associated quadratic form and using \eqref{cs:ofgm-primary:directions} gives
\begin{equation}\label{cs:ofgm-primary:quadratic}
 L\begin{bmatrix}\boldsymbol w_k\\\nabla f(x_{k+1})/L\end{bmatrix}^{\!\top}\vS_k\begin{bmatrix}\boldsymbol w_k\\\nabla f(x_{k+1})/L\end{bmatrix}
 =L\alpha_k\norm{s_k}^2+L\beta_k\norm{t_k}^2=S_k.
\end{equation}

This proves the claim. Code Cell~12 checks this factorization coefficient by
coefficient.
The two pure position terms give the coefficient identities
\[
\begin{aligned}
 \norm{z_{k+1}-x_\star}^2&:\quad
 (\vQ_k)_{11}-(\vQ_{k+1})_{11}=\alpha_k,\\
 \norm{y_{k+1}-z_{k+1}}^2&:\quad
 (\vQ_k)_{22}-\frac{\theta_k^4}{\theta_{k+1}^4}(\vQ_{k+1})_{22}
       =\frac{\beta_k}{(\theta_k^2-C_{k+1})^2}.
\end{aligned}
\]

The interior coefficient calculation after Step 2 checks the two gradient cross terms and the squared gradient norm. The position cross term
\(\csinner{z_{k+1}-x_\star}{y_{k+1}-z_{k+1}}\) is absent from both sides of \eqref{cs:ofgm-primary:quadratic}.

\textbf{Step 2: verification of the interior decrement identity.}
It remains to show that the left-hand side of \eqref{cs:ofgm-primary:quadratic} is the difference between the
weighted interpolation residuals and $V_{k+1}-V_k$.
To match \eqref{cs:ofgm-primary:potential-difference}, expand the weighted residuals. Since \(b_{k+1}=a_{k+1}-a_k\) and \(\nabla f(x_\star)=0\), their function-value terms agree with the potential difference:
\begin{equation}\label{cs:ofgm-primary:residual-expansion}
\begin{aligned}
&a_k\mathcal I_f(x_k,x_{k+1})+b_{k+1}\mathcal I_f(x_\star,x_{k+1})\\
 &\quad=a_{k+1}\bigl(f(x_{k+1})-f(x_\star)\bigr)-a_k\bigl(f(x_k)-f(x_\star)\bigr)\\
 &\qquad+\frac{a_k}{2L}\norm{\nabla f(x_k)}^2
 +\frac{a_{k+1}}{2L}\norm{\nabla f(x_{k+1})}^2\\
 &\qquad+\csinner{\nabla f(x_{k+1})}
 {a_k\left(x_k-x_{k+1}-\frac{\nabla f(x_k)}L\right)+b_{k+1}(x_\star-x_{k+1})}.
\end{aligned}
\end{equation}
The remaining inner-product coefficient follows from \(y_{k+1}=x_k-\nabla f(x_k)/L\) and the mixing update:
\begin{equation}\label{cs:ofgm-primary:residual-vector}
 a_k(y_{k+1}-x_{k+1})+b_{k+1}(x_\star-x_{k+1})
 =-b_{k+1}(z_{k+1}-x_\star)
 -\theta_k^2\sigma_{k+1}(y_{k+1}-z_{k+1}).
\end{equation}
Here \(x_{k+1}-x_\star=z_{k+1}-x_\star+\theta_k^2(y_{k+1}-z_{k+1})/\theta_{k+1}^2\), and the definition of \(\sigma_{k+1}\) gives the stated equality.

Subtract the difference of the potentials in
\eqref{cs:ofgm-primary:potential}. The function-value terms and the terms
involving $\norm{\nabla f(x_k)}^2$ cancel, while the common initial-radius
term cancels between the two potentials. The two contributions involving
$\norm{\nabla f(x_{k+1})}^2$ add to \(\frac{a_{k+1}}L\norm{\nabla f(x_{k+1})}^2\).
The remaining position and cross terms are exactly those in
\eqref{cs:ofgm-primary:slack-matrix}, so
\begin{equation}\label{cs:ofgm-primary:interpolation-bridge}
\begin{aligned}
 &a_k\mathcal I_f(x_k,x_{k+1})
 +b_{k+1}\mathcal I_f(x_\star,x_{k+1})-(V_{k+1}-V_k)
\\
 &\quad=L\begin{bmatrix}\boldsymbol w_k\\\nabla f(x_{k+1})/L\end{bmatrix}^{\!\top}\vS_k\begin{bmatrix}\boldsymbol w_k\\\nabla f(x_{k+1})/L\end{bmatrix}=S_k,
\end{aligned}
\end{equation}

where the last equality is the factorization proved in Step 1.
Rearranging proves the second line of \eqref{cs:ofgm-primary:decrements}
for $1\le k<N-1$. The base and terminal cases are treated below, followed
by the coefficient signs needed to deduce $V_{k+1}\le V_k$.
Code Cell~12 checks this interpolation cancellation and the complete decrement,
including every Gram, function-value, and constant coefficient.

\textbf{Terminal index.}
For $N\ge2$, this calculation proves the last line of
\eqref{cs:ofgm-primary:decrements} using the endpoint
\eqref{eq:case-ofgm-primary-lyapunov-end} (Code Cell~14).
The final gradient step gives
\(x_{N-1}-x_\star=y_N-x_\star+\nabla f(x_{N-1})/L\).
Because $a_{N-1}=1-b_N$, the remaining normalized quadratic is
\[
 \boldsymbol w_{N-1}^\top \vQ_{N-1}\boldsymbol w_{N-1}
 +\frac{1}{2L^2}\norm{\nabla f(y_N)}^2
 -\frac{b_N}{L}\csinner{\nabla f(y_N)}{y_N-x_\star}.
\]

At this index, the position coefficients are
\[
\begin{aligned}
 \norm{z_N-x_\star}^2&:\quad
 (\vQ_{N-1})_{11}=\frac{b_N}{2(D_0-\theta_{N-1}^2)},\\
 \norm{y_N-z_N}^2&:\quad
 (\vQ_{N-1})_{22}=\frac{b_N}{2(\theta_{N-1}^2-C_N)}.
\end{aligned}
\]
The gradient coefficient in the two squares satisfies
\[
 \frac1{L^2}\norm{\nabla f(y_N)}^2:\quad
 (D_0-\theta_{N-1}^2)^2\alpha_{N-1}+\beta_{N-1}
 =\frac{b_N(C_N+1)}2=\frac12.
\]
Completing the two squares therefore gives \(\alpha_{N-1}\norm{s_{N-1}}^2+\beta_{N-1}\norm{t_{N-1}}^2=S_{N-1}/L\).
This proves the terminal line of
\eqref{cs:ofgm-primary:decrements}, without any additional terminal square in $V_N$.

\textbf{Base index.}
This calculation settles the first line of \eqref{cs:ofgm-primary:decrements}
for $N\ge2$ (Code Cell~13).\footnote{Code Cell~15 verifies the single base--terminal case for $N=1$.}
For $N\ge2$, $x_1=y_1=z_1=x_0-\nabla f(x_0)/L$ and
$\theta_1^2-\theta_1=1$. Substitution of these relations and
$a_0=b_0=1/(D_0-1)$ into \eqref{cs:ofgm-primary:potential} yields
\[
 b_0\mathcal I_f(x_\star,x_0)+a_0\mathcal I_f(x_0,x_1)+b_1\mathcal I_f(x_\star,x_1)-V_1=S_0.
\]
Both sides are quadratic in $x_0-x_\star$, $\nabla f(x_0)/L$, and $\nabla f(x_1)/L$
after their function terms cancel. For $N=1$, the same expansion uses
$a_0=b_0=b_1=1/2$ and gives directly
\begin{align*}
 &\tfrac12\mathcal I_f(x_\star,x_0)+\tfrac12\mathcal I_f(x_0,y_1)+\tfrac12\mathcal I_f(x_\star,y_1)
 -\bigl(f(y_1)-f(x_\star)-L\norm{x_0-x_\star}^2/6\bigr)\\
 &\quad=\frac{3L}{8}\left\|\frac{\nabla f(x_0)}L-\frac{x_0-x_\star}{3}\right\|^2
 +\frac L8\left\|y_1-x_\star-2\frac{\nabla f(y_1)}L\right\|^2.
\end{align*}
This verifies the decrement identity directly for $N=1$.\footnote{When $N=2$, the base and terminal calculations already cover every transition.}

\textbf{Positivity of the coefficients.}
Lemma~\ref{cs:ofgm-primary:lem:signs} proves every required denominator and coefficient sign by
scalar identities, checked in Code Cell~10. The denominator bounds make the
formulas well-defined; the positive square coefficients give $S_k\ge0$,
and the positive interpolation weights make
\eqref{cs:ofgm-primary:decrements} nonpositive.

\textbf{Conclusion.}
Every interpolation residual in \eqref{cs:ofgm-primary:decrements} is nonpositive and every
$S_k$ is nonnegative. Thus $V_N\le\cdots\le V_0=0$.
The terminal definition in \eqref{eq:case-ofgm-primary-lyapunov-end} and
$\norm{x_0-x_\star}\le R$ give \eqref{cs:ofgm-primary:bound}.
\end{proof}

\textbf{Interior coefficient calculation.}
This verifies the vector cancellation in Step 2 and all coefficients of
\eqref{cs:ofgm-primary:quadratic} (Code Cell~12).
Use $x_k-\nabla f(x_k)/L=y_{k+1}$ and the mixing step to obtain
\begin{align*}
 &a_k\left(x_k-x_{k+1}-\frac{\nabla f(x_k)}L\right)+b_{k+1}(x_\star-x_{k+1})\\
 &\quad=-b_{k+1}(z_{k+1}-x_\star)
 +\left(a_k-\frac{\theta_k^2}{\theta_{k+1}^2}a_{k+1}\right)(y_{k+1}-z_{k+1})\\
 &\quad=-b_{k+1}(z_{k+1}-x_\star)-\theta_k^2\sigma_{k+1}(y_{k+1}-z_{k+1}).
\end{align*}
Together with the two position coefficients displayed after \eqref{cs:ofgm-primary:quadratic},
that quadratic identity is determined by the following three coefficient comparisons:
\begin{align*}
 \frac1L\csinner{\nabla f(x_{k+1})}{z_{k+1}-x_\star}&:\quad
 2\theta_{k+1}(\vQ_{k+1})_{11}-b_{k+1}=-2(D_0-\theta_k^2)\alpha_k,\\
 \frac1L\csinner{\nabla f(x_{k+1})}{y_{k+1}-z_{k+1}}&:\quad
 -\frac{2\theta_k^2(\theta_{k+1}-1)}{\theta_{k+1}^2}(\vQ_{k+1})_{22}-\theta_k^2\sigma_{k+1}
 =-\frac{2\beta_k}{\theta_k^2-C_{k+1}},\\
 \frac1{L^2}\norm{\nabla f(x_{k+1})}^2&:\quad
 \begin{aligned}[t]
 &a_{k+1}-\theta_{k+1}^2(\vQ_{k+1})_{11}-(\theta_{k+1}-1)^2(\vQ_{k+1})_{22}\\
 &\qquad=(D_0-\theta_k^2)^2\alpha_k+\beta_k.
 \end{aligned}
\end{align*}

They follow by substituting the displayed rational coefficients and using
\(\theta_{k+1}^2-\theta_{k+1}=\theta_k^2\) and \eqref{eq:case-C-recurrence}. Together with the zero
\(\csinner{z_{k+1}-x_\star}{y_{k+1}-z_{k+1}}\) coefficient,
these five coefficients cover all quadratic terms in
\(z_{k+1}-x_\star\), \(y_{k+1}-z_{k+1}\), and \(\nabla f(x_{k+1})/L\), so the verification is independent of $N$ and dimension.

\begin{lemma}%
\label{cs:ofgm-primary:lem:signs}
\normalfont
The interpolation and square coefficients are positive on their respective index ranges:\footnote{The unused values $a_{-1}$ and, when $N=1$, $\beta_0$ are zero.}
\[
 a_k>0,\qquad b_k>0,\qquad \alpha_k>0,\qquad \beta_k>0.
\]

\end{lemma}
\begin{proof}
Define \(\rho_k\in\mathbb R\) for this scalar argument by
\begin{equation*}
 \rho_k=\frac{\sum_{i=0}^{k-1}\theta_i^2}{\theta_k^2},\qquad\rho_0=0.
\end{equation*}
The theta relation gives $\theta_k=\theta_k^2-\theta_{k-1}^2$ and, by factoring,
$\theta_k^4-\theta_{k-1}^4
=\theta_k(\theta_k^2+\theta_{k-1}^2)=2\theta_k^3-\theta_k^2$.
Summing these identities gives
\begin{equation}\label{cs:ofgm-primary:moment}
 C_{k+1}=\frac{\theta_k^2+\rho_k+1}{2},\qquad
 \rho_k+1=1+\left(1-\frac1{\theta_k}\right)(\rho_{k-1}+1).
\end{equation}
The recurrence is understood for $k\geq1$. Starting from $\rho_0+1=1$ and using
$\theta_k>\theta_{k-1}+1/2$, induction yields
\[
 \rho_k+1\le\frac{2\theta_k+1}{3},\qquad
 \rho_{k-1}+1<\frac{2\theta_k}{3},\qquad
 \rho_k-\rho_{k-1}>\frac13\quad(k\ge1).
\]
Indeed the second inequality follows from the preceding-index upper bound,
and inserting it into \eqref{cs:ofgm-primary:moment} proves the next upper bound and the
increment estimate. Hence $C_k$ increases,
$D_k\ge C_N+1>0$, and
$D_0-\theta_k^2\ge\rho_{N-1}+2>0$ for $k<N$.
Moreover $\theta_k^2-C_{k+1}>0$ for $k\ge1$, by \eqref{cs:ofgm-primary:moment}
and $\rho_k\le2(\theta_k-1)/3$.

\begin{itemize}
\item \textbf{Positivity of $a_k$.}
The same telescoping identities give
\[
 \sum_{i=0}^k\theta_i\theta_{i-1}^2
   =\theta_k^2(\theta_k^2-C_{k+1}),\qquad
 \sum_{i=0}^k\theta_i=\theta_k^2.
\]
They prove the two formulas for $a_k$ in \eqref{cs:ofgm-primary:weights}.
For $0\le k<N$, every summand in the first formula is positive because
$\theta_i>0$ and $\theta_{i-1}<\theta_{N-1}$ for $0\le i\le k$.
Together with $D_{k+1}>0$, this gives $a_k>0$.
At the boundary indices, $a_N=1$ and $a_{-1}=0$ by definition.
Code Cell~10 checks the telescoping increments and coefficient identities.

For the remaining coefficient signs, we first show that the auxiliary difference $\sigma_k$ is positive.
At index zero, $\sigma_0=a_0>0$. For $1\le k<N$, using
$\theta_{N-1}\ge\theta_k$, substitution of
\eqref{cs:ofgm-primary:weights} into \eqref{cs:ofgm-primary:sigma} gives the positive expansion
\begin{equation}\label{cs:ofgm-primary:signexp}
\begin{aligned}
 &\theta_k\theta_{N-1}^2D_kD_{k+1}\sigma_k\\
 &\quad=\frac{\theta_k}{18}\bigl[(3\theta_k-4)(\theta_k-1)
 +2(3\theta_k+1)(\theta_{N-1}-\theta_k)
 +6(\theta_{N-1}-\theta_k)^2\bigr]\\
 &\qquad+\left(\frac{2\theta_k}{3}-(\rho_{k-1}+1)\right)
 \left[\theta_{N-1}^2-\theta_k^2+\frac{\theta_k}{2}
 +\frac{\theta_{N-1}}{3}+\frac23\right]\\
 &\qquad+C_k\left(\frac{2\theta_{N-1}+1}{3}-(\rho_{N-1}+1)\right).
\end{aligned}
\end{equation}
By the bounds on $\rho_{k-1}$ and $\rho_{N-1}$ above, every summand is
nonnegative, and the first is strictly positive because
$\theta_k\ge\theta_1>4/3$. The factor
$\theta_k\theta_{N-1}^2D_kD_{k+1}$ is positive, so $\sigma_k>0$.
Code Cell~10 checks \eqref{cs:ofgm-primary:signexp} on the stated scalar relations.

\item \textbf{Positivity of $b_k$.}
For $1\le k<N$, the theta relation and \eqref{cs:ofgm-primary:sigma} give
\[
 b_k=\frac{a_k}{\theta_k}+\theta_{k-1}^2\sigma_k>0.
\]
At the boundary indices, $b_0=a_0>0$ and $b_N=1/(C_N+1)>0$.
Since $a_N=1$ and $b_k=a_k-a_{k-1}>0$, we also have $1-a_k>0$ for $0\le k<N$.

\item \textbf{Positivity of $\alpha_k$.}
The inequalities $\sigma_k>0$ imply that $a_k/\theta_k^2$ increases from
$a_0=1/(D_0-1)>1/D_0$. Together with the formula for $b_k$ above, this gives
\[
 D_0b_k-\theta_k>0,\qquad D_0a_k-\theta_k^2>0
 \quad(0\le k<N),
\]
where the signs at $k=0$ follow from $b_0=a_0$ and $\theta_0=1$.
All numerator factors in \eqref{cs:ofgm-primary:alpha} are therefore positive,
as are its denominator factors, so $\alpha_k>0$ for $0\le k<N-1$.
The terminal formula \eqref{cs:ofgm-primary:terminalcoeff} gives
$\alpha_{N-1}>0$ for $N\ge2$.\footnote{For $N=1$, $\alpha_0=1/8$, and the other coefficient in $S_0$ is $3/8$.}

\item \textbf{Positivity of $\beta_k$.}
For $0\le k<N$, the exact tail identity is
\begin{equation}\label{cs:ofgm-primary:tail}
 \frac1{\theta_{N-1}^2}-\frac{a_k}{\theta_k^2}
 =\frac{\rho_{N-1}+1-\rho_k}{\theta_{N-1}^2D_{k+1}}>0.
\end{equation}
Its numerator is at least one because $\rho_k\le\rho_{N-1}$.
Code Cell~10 checks \eqref{cs:ofgm-primary:tail} on the same relations.
Thus \eqref{cs:ofgm-primary:beta} is a sum of positive terms multiplied
by a positive factor, proving $\beta_k>0$ for $1\le k<N-1$.
For $N\ge2$, the terminal formula \eqref{cs:ofgm-primary:terminalcoeff}
is positive because $b_N>0$ and $\theta_{N-1}^2-C_N>0$.\footnote{For $N=1$, the unused value $\beta_0$ is zero.}
\end{itemize}

\end{proof}

\subsection{Proof of Theorem~\ref{thm:case-ofgm-secondary} (Conjecture~5 in \texorpdfstring{\citet{taylor2017smooth}}{Taylor et al. (2017)})}
\label{app:case-ofgm-secondary-proof}

\subsubsection{Input prompt}

\begin{commandbox}
/pep-implement

Function: f is convex and L-smooth
Parameters: L, R
Initial condition: ||x_0 - x_star|| <= R, where grad f(x_star) = 0
Performance metric: f(x_N) - f(x_star)
Algorithm: Fast gradient method with fixed step size 1/L

Parameter sequence: theta_0 = 1 and theta_{k+1} = (1 + sqrt(1 + 4 * theta_k^2)) / 2 for k >= 0
Initialization: x_0 = y_0 = z_0
For k = 0,...,N-1:
y_{k+1} = x_k - (1/L) * grad f(x_k)
z_{k+1} = z_k - (theta_k/L) * grad f(x_k)
x_{k+1} = (1 - 1/theta_{k+1}) * y_{k+1} + (1/theta_{k+1}) * z_{k+1}

Conjectured rate: f(x_N) - f(x_star) <= L * R^2 / (2 * (2 * D_N + 1)), where D_N = sum_{j=0}^{N-1} h_{N,j} and x_i = x_0 - (1/L) * sum_{j=0}^{i-1} h_{i,j} * grad f(x_j)
\end{commandbox}
Then, we repeatedly entered:
\begin{commandbox}
proceed with the next step
\end{commandbox}

{The initial proof of Conjecture~5 was algebraically involved.
We therefore asked the agent to revise it using the two-square decomposition
in the FGM-rational proof (Appendix~\ref{app:case-sfgm-proof}), with the
additional prompt below. This led to the simpler two-square proof
presented below.

\colorlet{csadditionalpromptcolor}{.}
\begin{commandbox}[coltext=csadditionalpromptcolor,
 listing options={basicstyle=\ttfamily\small\color{csadditionalpromptcolor},
 columns=fullflexible,breaklines=true,breakatwhitespace=true,breakindent=0pt,breakautoindent=false,keepspaces=true,showstringspaces=false,
 escapeinside={(*@}{@*)}}]
Could you explore whether allowing two squared-norm terms per iteration in the SOS decomposition, as in (*@\texttt{fgm\_rational\_tight\_example\_lyap.ipynb}@*), can simplify the proof of FGM Conjecture 5?
\end{commandbox}
}

\begin{proof}[Proof outline for Theorem~\ref{thm:case-ofgm-secondary}]
Write $f_\star=f(x_\star){\in\mathbb R}$, and set $D_j=2C_{N+1}-1-C_j{\in\mathbb R}$ and
$\tau_N=1/(2D_0){\in\mathbb R}$. Define the initial and interior potentials by
{\(V_{0}=0\) and}
{%
\begin{equation}\label{cs:ofgm-secondary:potential}
\begin{aligned}
 V_k={}&a_{k-1}(f(x_k)-f_\star)-\tau_NL\norm{x_0-x_\star}^2
       -\frac{a_{k-1}}{2L}\norm{\nabla f(x_k)}^2\\
      &+b_k\csinner{\nabla f(x_k)}{y_{k+1}-x_\star}
       +L\boldsymbol w_k^\top {\vQ_k}\boldsymbol w_k,
       \qquad 1\le k<N.
\end{aligned}
\end{equation}
}
At the output point, use the terminal potential
\begin{equation}\label{cs:ofgm-secondary:potential-end}
\begin{aligned}
 V_N={}&a_{N-1}(f(x_N)-f_\star)-\tau_NL\norm{x_0-x_\star}^2
 +b_N\left(\csinner{\nabla f(x_N)}{x_N-x_\star}
 -\frac1{2L}\norm{\nabla f(x_N)}^2\right).
\end{aligned}
\end{equation}
For $N\ge2$, the exact decrement identities are\footnote{{For $N=2$, the middle range is empty. For $N=1$, the first line alone holds
with $a_0=1/2$, $V_1$ given by \eqref{cs:ofgm-secondary:potential-end},
and $S_0$ given by \eqref{cs:ofgm-secondary:S0-one}; no interior matrix is used.}}
{%
\begin{equation}\label{cs:ofgm-secondary:decrement}
 V_{k+1}-V_k=
 \begin{cases}
  \displaystyle a_0\bigl[\mathcal I_f(x_\star,x_0)+\mathcal I_f(x_0,x_1)\bigr]-S_0, & k=0,\\
  \displaystyle b_k\mathcal I_f(x_\star,x_k)+a_k\mathcal I_f(x_k,x_{k+1})-S_k, & 1\le k\le N-2,\\
  \displaystyle b_{N-1}\mathcal I_f(x_\star,x_{N-1})+a_{N-1}\mathcal I_f(x_{N-1},x_N)-S_{N-1}, & k=N-1.
 \end{cases}
\end{equation}
}
{The full theorem and proof below verify the identities, and Lemma~\ref{cs:ofgm-secondary:lem:signs} proves positivity of every used interpolation and square coefficient.} Thus $S_k\ge0$ and $\mathcal I_f\le0$
imply $V_N\le\cdots\le V_0=0$.
The endpoint identity is
\begin{equation}\label{cs:ofgm-secondary:normalization}
 f(x_N)-f_\star-\tau_NLR^2
 =V_N+b_N\mathcal I_f(x_\star,x_N)
      +\tau_NL\bigl(\norm{x_0-x_\star}^2-R^2\bigr).
\end{equation}
Since $b_N,\tau_N>0$, the stationary-point residual and the radius assumption
make the last two terms nonpositive. This gives \eqref{eq:case-ofgm-secondary-rate}.
\end{proof}

\begin{theorem}[{Theorem~\ref{thm:case-ofgm-secondary} with full details of the Lyapunov function}]
\label{cs:ofgm-secondary:thm:main}
\normalfont
Let $f\in\mathcal F_L$, let $x_\star{\in\mathcal H}$ minimize $f$, write $f_\star=f(x_\star){\in\mathbb R}$, and suppose $\norm{x_0-x_\star}\le R$,
where $R\ge0$. For a fixed integer $N\ge1$, run~\ref{eq:case-theta-recursive} from $x_0=y_0=z_0$.
{For $0\le k\le N$, define}
{%
\begin{equation*}
 D_k=2C_{N+1}-1-C_k\in\mathbb R,\qquad
 \tau_N=\frac1{2(2C_{N+1}-1)}=\frac1{2D_0}\in\mathbb R.
\end{equation*}
Here $D_k$ denotes a gap variable, not a fixed-step row sum.
For $0\le k<N$, define $h_{N,k},m_k\in\mathbb R$ by
\begin{equation}\label{cs:ofgm-secondary:h}
 h_{N,k}=\theta_k\left(1-\frac{\theta_{k-1}^2}{\theta_N^2}\right),
 \qquad m_k=\frac{\theta_k-1}{\theta_{k+1}}.
\end{equation}
}
Set $a_{-1}=0$ and $a_N=1$. {Define the chain weights \(a_k\in\mathbb R\) for \(0\le k<N\)
and the star weights \(b_k\in\mathbb R\) for \(0\le k\le N\) by}
\begin{equation}\label{cs:ofgm-secondary:ab}
 a_k=\frac{\sum_{j=0}^k h_{N,j}}{D_{k+1}},
 \qquad b_k=a_k-a_{k-1}.
\end{equation}
In particular, $a_0=b_0=1/D_1$.

{For \(1\le k<N\), define the state \(\boldsymbol w_k\in\mathcal H^2\)
and the symmetric matrix \(\vQ_k\in\mathbb R^{2\times2}\) by}
{%
\begin{equation}\label{cs:ofgm-secondary:Q}
 \boldsymbol w_k=\begin{bmatrix}y_k-x_\star\\y_{k+1}-x_\star\end{bmatrix},\qquad
 \vQ_k=\frac12\begin{pmatrix}
 \dfrac{1-h_{N,k}+D_{k+1}(a_{k-1}+a_k/D_k)}{D_k}
 &-a_{k-1}-\dfrac{a_k}{D_k}\\[6pt]
 -a_{k-1}-\dfrac{a_k}{D_k}&a_k
 \end{pmatrix}.
\end{equation}
}
{%
Define the Lyapunov values \(V_k\in\mathbb R\) as in \eqref{cs:ofgm-secondary:potential} and \eqref{cs:ofgm-secondary:potential-end}, \ie, set \(V_{0}=0\) and
\begin{equation*}
\begin{aligned}
 V_k={}&a_{k-1}(f(x_k)-f(x_\star))-\tau_NL\norm{x_0-x_\star}^2
       -\frac{a_{k-1}}{2L}\norm{\nabla f(x_k)}^2\\
      &+b_k\csinner{\nabla f(x_k)}{y_{k+1}-x_\star}
       +L\boldsymbol w_k^\top {\vQ_k}\boldsymbol w_k,
       \qquad 1\le k<N.
\end{aligned}
\end{equation*}
At the output point, set
\begin{equation*}
\begin{aligned}
 V_N={}&a_{N-1}(f(x_N)-f(x_\star))-\tau_NL\norm{x_0-x_\star}^2\\
 &+b_N\left(\csinner{\nabla f(x_N)}{x_N-x_\star}
 -\frac1{2L}\norm{\nabla f(x_N)}^2\right).
\end{aligned}
\end{equation*}
}

Here $\boldsymbol w^\top {\vQ}\boldsymbol w$ denotes the corresponding scalar sum of inner products.
The gradient at $x_N$ is a proof witness and does not change the algorithm.\footnote{{For $N=1$, only $V_0$ and the terminal definition are used.}}

{For $1\le k\le N-2$, define the square coefficients
\(\alpha_k,\beta_k\in\mathbb R\) by}
{%
\begin{equation}\label{cs:ofgm-secondary:margin}
\begin{aligned}
 \alpha_k={}&\frac{a_{k+1}}2,\\
 \beta_k={}&\frac12\bigl[a_kD_{k+1}^2+(a_k-1)D_{k+1}
 -(h_{N,k}-1)(D_k+D_{k+1})-a_{k+1}m_k^2D_k^2\bigr].
\end{aligned}
\end{equation}
}
{With the vectors \(s_k,t_k\in\mathcal H\) defined by}
\begin{equation}\label{cs:ofgm-secondary:directions}
 s_k=\frac{\nabla f(x_{k+1})}L-\frac{y_{k+1}-x_\star}{D_{k+1}},\qquad
 t_k=\frac{y_k-x_\star}{D_k}-\frac{y_{k+1}-x_\star}{D_{k+1}},
\end{equation}
set
{%
\[
 S_k=L\alpha_k\norm{s_k}^2+L\beta_k\norm{t_k}^2.
\]
}
For $N\geq2$, {define the terminal coefficients \(\alpha_{N-1},\beta_{N-1}\in\mathbb R\), directions \(s_{N-1},t_{N-1}\in\mathcal H\), and square sum \(S_{N-1}\) by}
\begin{equation}\label{cs:ofgm-secondary:terminalS}
\begin{aligned}
 \alpha_{N-1}&=\frac12,\qquad
 \beta_{N-1}=\frac{a_{N-1}-(m_{N-1}+b_N)^2}{2},\\
 s_{N-1}&=\frac{\nabla f(x_N)}L-(x_N-x_\star)
              +a_{N-1}(y_N-x_\star),\\
 t_{N-1}&=y_N-x_\star-\frac{D_N}{D_{N-1}}(y_{N-1}-x_\star),\\
 S_{N-1}&=L\alpha_{N-1}\norm{s_{N-1}}^2
              +L\beta_{N-1}\norm{t_{N-1}}^2.
\end{aligned}
\end{equation}
{For $N\ge2$, the initial square sum is}
{%
\begin{equation}\label{cs:ofgm-secondary:S0}
S_0=\frac{L}{2D_0}\left\|x_0-x_\star-\frac{D_0}{D_1}(x_1-x_\star)\right\|^2
        +\frac{La_1}{2}\left\|\frac{\nabla f(x_1)}L-\frac{x_1-x_\star}{D_1}\right\|^2.
\end{equation}
}
For $N=1$, define the initial square sum separately by
\begin{equation}\label{cs:ofgm-secondary:S0-one}
\begin{aligned}
 S_0={}&\frac{3L}{8}\left\|\frac{\nabla f(x_0)}L-\frac{x_0-x_\star}{3}\right\|^2
 +\frac L2\left\|\frac{\nabla f(x_1)}L-\frac{x_1-x_\star}{2}\right\|^2.
\end{aligned}
\end{equation}
{%
The square sums satisfy \(S_k\ge0\). For \(N\ge2\), the following decrement identities hold:\footnote{For \(N=1\), the first case holds with \(a_0=1/2\), the terminal definition of \(V_1\), and \eqref{cs:ofgm-secondary:S0-one}. For \(N=2\), the middle range is empty.}
\begin{equation*}
 V_{k+1}-V_k=
 \begin{cases}
  \displaystyle a_0\bigl[\mathcal I_f(x_\star,x_0)+\mathcal I_f(x_0,x_1)\bigr]-S_0, & k=0,\\
  \displaystyle b_k\mathcal I_f(x_\star,x_k)+a_k\mathcal I_f(x_k,x_{k+1})-S_k, & 1\le k\le N-2,\\
  \displaystyle b_{N-1}\mathcal I_f(x_\star,x_{N-1})+a_{N-1}\mathcal I_f(x_{N-1},x_N)-S_{N-1}, & k=N-1.
 \end{cases}
\end{equation*}
The coefficient signs in Lemma~\ref{cs:ofgm-secondary:lem:signs} and \(\mathcal I_f\le0\) imply \(V_N\le\cdots\le V_0=0\).
Since \(a_{N-1}+b_N=1\), the terminal definition gives
\[
 f(x_N)-f(x_\star)=V_N+b_N\mathcal I_f(x_\star,x_N)+\tau_NL\norm{x_0-x_\star}^2
 \le\tau_NLR^2.
\]
}
Consequently,
\begin{equation}\label{cs:ofgm-secondary:bound}
 f(x_N)-f_\star\le \frac{LR^2}{2(2C_{N+1}-1)}
 =\frac{LR^2\theta_N^2}{2\bigl(2\sum_{j=0}^N\theta_j^3-\theta_N^2\bigr)}.
\end{equation}
\end{theorem}

\begin{proof}
{%
The companion notebook is
\path{examples_peppy/fgm_conjecture5/fgm_conjecture5_example_lyap.ipynb}.
The notebook also documents
the prerequisites for running its exact proof checks separately.
}

{%
For \(1\le k\le N-2\), subtracting two consecutive instances of \eqref{cs:ofgm-secondary:potential} cancels the common initial-radius term and gives
\begin{equation}\label{cs:ofgm-secondary:potential-difference}
\begin{aligned}
V_{k+1}-V_k
 &=a_k\bigl(f(x_{k+1})-f(x_\star)\bigr)-a_{k-1}\bigl(f(x_k)-f(x_\star)\bigr)\\
 &\quad-\frac{a_k}{2L}\norm{\nabla f(x_{k+1})}^2
 +\frac{a_{k-1}}{2L}\norm{\nabla f(x_k)}^2\\
 &\quad+b_{k+1}\csinner{\nabla f(x_{k+1})}{y_{k+2}-x_\star}
 -b_k\csinner{\nabla f(x_k)}{y_{k+1}-x_\star}\\
 &\quad+L\boldsymbol w_{k+1}^\top\vQ_{k+1}\boldsymbol w_{k+1}
 -L\boldsymbol w_k^\top\vQ_k\boldsymbol w_k.
\end{aligned}
\end{equation}
Our goal is to identify this expression with the interior case of \eqref{cs:ofgm-secondary:decrement}. We first factor the quadratic remainder, then match the weighted interpolation residuals; the boundary cases are treated separately.
}

\textbf{Step 1: factorization of the interior slack.}
For $1\le k\le N-2$, Steps 1 and 2 prove the second line of
\eqref{cs:ofgm-secondary:decrement}, namely,
\[
 V_{k+1}-V_k=b_k\mathcal I_f(x_\star,x_k)
              +a_k\mathcal I_f(x_k,x_{k+1})-S_k,
\]
with $S_k$ defined by \eqref{cs:ofgm-secondary:margin}--\eqref{cs:ofgm-secondary:directions}.
We first factor a quadratic remainder into its two squares; Step 2 identifies
that remainder by subtracting the potential difference from the interpolation residuals.
{%
Fix $1\le k\le N-2$.
}
{%
Define the symmetric matrix \(\vS_k\in\mathbb R^{3\times3}\) by
}
{%
\begin{equation}\label{cs:ofgm-secondary:slack-matrix}
\begin{aligned}
 \vS_k={}&\begin{pmatrix}
 \vQ_k&\frac12\begin{bmatrix}m_ka_{k+1}\\-a_km_k-b_{k+1}(1+m_k)\end{bmatrix}\\[3pt]
 \frac12\begin{bmatrix}m_ka_{k+1}\\-a_km_k-b_{k+1}(1+m_k)\end{bmatrix}^{\!\top}&a_{k+1}
 \end{pmatrix}\\
 &-\begin{pmatrix}0&1&0\\-m_k&1+m_k&-1\end{pmatrix}^{\!\top}
 \vQ_{k+1}\begin{pmatrix}0&1&0\\-m_k&1+m_k&-1\end{pmatrix}.
\end{aligned}
\end{equation}
}
{%
Eliminating $z_k=\theta_kx_k-(\theta_k-1)y_k$ from the updates gives
\begin{equation}\label{cs:ofgm-secondary:momentum}
 x_{k+1}=(1+m_k)y_{k+1}-m_ky_k,\qquad 0\le k<N.
\end{equation}
Together with \(y_{k+2}=x_{k+1}-\nabla f(x_{k+1})/L\), this gives
\[
 \boldsymbol w_{k+1}=\begin{pmatrix}0&1&0\\-m_k&1+m_k&-1\end{pmatrix}\begin{bmatrix}\boldsymbol w_k\\\nabla f(x_{k+1})/L\end{bmatrix}.
\]
Unrolling this momentum form to $x_N$ gives the final-row relations
\begin{equation}\label{cs:ofgm-secondary:final-row}
 h_{N,k}=1+m_kh_{N,k+1}\quad(0\le k<N-1),\qquad
 h_{N,N-1}=1+m_{N-1}.
\end{equation}
The formula in \eqref{cs:ofgm-secondary:h} satisfies these relations by
\eqref{eq:case-theta-identity}, so it gives the final row of the fixed-step
representation in the common FGM lemma.
Substituting \eqref{eq:case-C-recurrence} and
\eqref{eq:case-theta-identity} gives $\delta_{k+1}=1+m_k\delta_k$ for $0\le k<N$.
Thus the definitions of $D_k$ and $a_k$ give
\begin{equation}\label{cs:ofgm-secondary:local-relations}
\begin{gathered}
 D_k-D_{k+1}=\delta_k,\qquad D_{k+2}=D_{k+1}-1-m_k\delta_k,\\
 a_kD_{k+1}-a_{k-1}D_k=h_{N,k}.
\end{gathered}
\end{equation}
The second identity is used for $0\le k<N-1$; the others hold for
$0\le k<N$. Code Cell~29 checks these scalar identities.
These relations suffice for the following coefficient comparison;
the denominator signs are proved independently in
Lemma~\ref{cs:ofgm-secondary:lem:signs}.
}
The {position matrices} satisfy
{%
\begin{equation}\label{cs:ofgm-secondary:position}
\begin{aligned}
 &\boldsymbol w_k^\top \vQ_k\boldsymbol w_k
 -\begin{bmatrix}y_{k+1}-x_\star\\x_{k+1}-x_\star\end{bmatrix}^{\!\top}
 \vQ_{k+1}\begin{bmatrix}y_{k+1}-x_\star\\x_{k+1}-x_\star\end{bmatrix}\\
 &\qquad=\frac{a_{k+1}}{2D_{k+1}^2}\norm{y_{k+1}-x_\star}^2
       +\beta_k\left\|\frac{y_k-x_\star}{D_k}-\frac{y_{k+1}-x_\star}{D_{k+1}}\right\|^2.
\end{aligned}
\end{equation}
}
{%
Using \eqref{cs:ofgm-secondary:momentum} to expand
\(x_{k+1}-x_\star=(1+m_k)(y_{k+1}-x_\star)-m_k(y_k-x_\star)\),
the three position coefficients satisfy
\[
\begin{aligned}
 \norm{y_k-x_\star}^2&:\quad
 (\vQ_k)_{11}-m_k^2(\vQ_{k+1})_{22}=\frac{\beta_k}{D_k^2},\\
 2\csinner{y_k-x_\star}{y_{k+1}-x_\star}&:\quad
 (\vQ_k)_{12}+m_k(\vQ_{k+1})_{12}+m_k(1+m_k)(\vQ_{k+1})_{22}
 =-\frac{\beta_k}{D_kD_{k+1}},\\
 \norm{y_{k+1}-x_\star}^2&:\quad
 \begin{aligned}[t]
 &(\vQ_k)_{22}-(\vQ_{k+1})_{11}-2(1+m_k)(\vQ_{k+1})_{12}\\
 &\qquad-(1+m_k)^2(\vQ_{k+1})_{22}
 =\frac{a_{k+1}+2\beta_k}{2D_{k+1}^2}.
 \end{aligned}
\end{aligned}
\]
}
{Substitution of \eqref{cs:ofgm-secondary:final-row} and
\eqref{cs:ofgm-secondary:local-relations} verifies all three identities;
Code Cell~30 checks the coefficients of \eqref{cs:ofgm-secondary:position}.}

{%
Using \(a_{k+1}=a_k+b_{k+1}\) and
\(2(\vQ_{k+1})_{22}=a_{k+1}\), the gradient coefficients are
\[
\begin{aligned}
 \frac1{L^2}\norm{\nabla f(x_{k+1})}^2&:\quad
 a_{k+1}-(\vQ_{k+1})_{22}=\frac{a_{k+1}}2,\\
 \frac1L\csinner{\nabla f(x_{k+1})}{y_k-x_\star}&:\quad
 m_k\bigl(a_{k+1}-2(\vQ_{k+1})_{22}\bigr)=0,\\
 \frac1L\csinner{\nabla f(x_{k+1})}{y_{k+1}-x_\star}&:\quad
 a_k+2(\vQ_{k+1})_{12}=-\frac{a_{k+1}}{D_{k+1}}.
\end{aligned}
\]
}
{%
Together with \eqref{cs:ofgm-secondary:position}, these coefficient identities give
\[
 \vS_k=\alpha_k\begin{bmatrix}0\\-1/D_{k+1}\\1\end{bmatrix}
 \begin{bmatrix}0\\-1/D_{k+1}\\1\end{bmatrix}^{\!\top}
 +\beta_k\begin{bmatrix}1/D_k\\-1/D_{k+1}\\0\end{bmatrix}
 \begin{bmatrix}1/D_k\\-1/D_{k+1}\\0\end{bmatrix}^{\!\top}.
\]
Evaluating the associated quadratic form and using \eqref{cs:ofgm-secondary:directions} gives
\begin{equation}\label{cs:ofgm-secondary:factorization}
 L\begin{bmatrix}\boldsymbol w_k\\\nabla f(x_{k+1})/L\end{bmatrix}^{\!\top}\vS_k\begin{bmatrix}\boldsymbol w_k\\\nabla f(x_{k+1})/L\end{bmatrix}
 =L\alpha_k\norm{s_k}^2+L\beta_k\norm{t_k}^2=S_k,
\end{equation}
}
{%
as claimed. {Code Cell~30} checks all six quadratic coefficients in $y_k-x_\star$, $y_{k+1}-x_\star$, and $\nabla f(x_{k+1})$.
}

\textbf{Step 2: verification of the interior decrement identity.}
It remains to show that {the left-hand side of \eqref{cs:ofgm-secondary:factorization}} is the difference between
$b_k\mathcal I_f(x_\star,x_k)+a_k\mathcal I_f(x_k,x_{k+1})$ and $V_{k+1}-V_k$.
{%
To match \eqref{cs:ofgm-secondary:potential-difference}, first compare the function values.
The identity \(b_k=a_k-a_{k-1}\) gives the coefficients \(a_k\) and \(-a_{k-1}\) of \(f(x_{k+1})-f(x_\star)\) and \(f(x_k)-f(x_\star)\).
For the gradient terms, the update and \eqref{cs:ofgm-secondary:momentum} give
\[
 x_k-x_\star=y_{k+1}-x_\star+\frac{\nabla f(x_k)}L,\qquad
 x_k-x_{k+1}-\frac{\nabla f(x_k)}L=m_k(y_k-y_{k+1}).
\]
Expanding the residuals using these relations yields
\begin{equation}\label{cs:ofgm-secondary:residual-expansion}
\begin{aligned}
&b_k\mathcal I_f(x_\star,x_k)+a_k\mathcal I_f(x_k,x_{k+1}) =a_k\bigl(f(x_{k+1})-f(x_\star)\bigr)-a_{k-1}\bigl(f(x_k)-f(x_\star)\bigr)\\
 &\qquad+\frac{a_{k-1}}{2L}\norm{\nabla f(x_k)}^2
 +\frac{a_k}{2L}\norm{\nabla f(x_{k+1})}^2\\
 &\qquad-b_k\csinner{\nabla f(x_k)}{y_{k+1}-x_\star}
 +a_km_k\csinner{\nabla f(x_{k+1})}{y_k-y_{k+1}}.
\end{aligned}
\end{equation}
Subtracting \eqref{cs:ofgm-secondary:potential-difference} cancels the function values and the terms involving \(\nabla f(x_k)\).
The two next-gradient norm terms add to \((a_k/L)\norm{\nabla f(x_{k+1})}^2\); the remaining position and inner-product terms give the left-hand side of \eqref{cs:ofgm-secondary:factorization}. Thus
}
{%
\begin{equation}\label{cs:ofgm-secondary:interpolation-bridge}
\begin{aligned}
 &b_k\mathcal I_f(x_\star,x_k)+a_k\mathcal I_f(x_k,x_{k+1})-(V_{k+1}-V_k)
\\
 &\quad=L\begin{bmatrix}\boldsymbol w_k\\\nabla f(x_{k+1})/L\end{bmatrix}^{\!\top}\vS_k\begin{bmatrix}\boldsymbol w_k\\\nabla f(x_{k+1})/L\end{bmatrix}=S_k,
\end{aligned}
\end{equation}
}
where the last equality is the factorization proved in Step 1. Rearranging
proves the second line of \eqref{cs:ofgm-secondary:decrement} for
$1\le k\le N-2$. {Code Cell~36} checks this bridge and the complete decrement
in every quadratic, function-value, and constant coefficient.
The base and terminal cases are treated next; the coefficient signs needed
for monotonicity are established after them.

\textbf{Terminal index.} Assume $N\geq2$.
This verifies the last line of \eqref{cs:ofgm-secondary:decrement} using
\eqref{cs:ofgm-secondary:potential-end}; {Code Cell~37} checks both the position
factorization and the full decrement.
{%
Let \(m=m_{N-1}\in\mathbb R\).
}
The {matrix} gives
{%
\begin{equation}\label{cs:ofgm-secondary:terminal-position}
\begin{aligned}
 \boldsymbol w_{N-1}^\top \vQ_{N-1}\boldsymbol w_{N-1}
 ={}&\frac12\norm{m(y_{N-1}-x_\star)-(m+b_N)(y_N-x_\star)}^2\\
 &+\beta_{N-1}\norm{t_{N-1}}^2.
\end{aligned}
\end{equation}
}
{The common FGM lemma gives
$\sum_{k=0}^{N-1}h_{N,k}=C_{N+1}-1$, so \eqref{cs:ofgm-secondary:ab}
yields $b_N=\delta_N/D_N$. Together with
$h_{N,N-1}=1+m_{N-1}$ and $\delta_N=1+m_{N-1}\delta_{N-1}$, this gives the identities}
\[
 {\vQ_{N-1}}\begin{bmatrix}D_{N-1}\\D_N\end{bmatrix}
 =\frac12\begin{bmatrix}1-h_{N,N-1}\\h_{N,N-1}-a_{N-1}\end{bmatrix},
 \qquad mD_{N-1}-(m+b_N)D_N=-1.
\]
Both sides of \eqref{cs:ofgm-secondary:terminal-position} have this action on the indicated
{%
vector and the same $\norm{y_N-x_\star}^2$ coefficient, $a_{N-1}/2$.
}
Since $D_{N-1}>0$, these facts determine the three matrix entries uniquely.
{%
Using $x_N-x_\star=(1+m)(y_N-x_\star)-m(y_{N-1}-x_\star)$ and $b_N=1-a_{N-1}$, expanding the two
}
final interpolation residuals converts the first square in
\eqref{cs:ofgm-secondary:terminal-position} into $\norm{s_{N-1}}^2$ and gives
\eqref{cs:ofgm-secondary:terminalS}. This proves the final decrement.

\textbf{Base index.}
This verifies the first line of \eqref{cs:ofgm-secondary:decrement} for
$N\ge2$ ({Code Cell~38}).\footnote{{Together with the terminal calculation, this also settles every transition when $N=2$.}}
{%
For $N\ge2$, we have $x_1=y_1=z_1$, $D_0=D_1+1$, $D_2=D_1-1$,
$a_0=1/D_1$, and $h_{N,1}=a_1D_2-1$. Direct expansion gives
\[
\begin{aligned}
 &\mathcal I_f(x_\star,x_0)+\mathcal I_f(x_0,x_1)\\
 &\qquad=f(x_1)-f_\star+L\norm{x_1-x_\star}^2
 -L\csinner{x_0-x_\star}{x_1-x_\star}+\frac1{2L}\norm{\nabla f(x_1)}^2.
\end{aligned}
\]
}
Substituting \eqref{cs:ofgm-secondary:Q} in $V_1$ and subtracting gives precisely $S_0$
in \eqref{cs:ofgm-secondary:S0}, proving the base identity.

\textbf{Positivity of the coefficients.}
{%
Lemma~\ref{cs:ofgm-secondary:lem:signs} proves all required coefficient
signs and denominator bounds directly from the scalar definitions. The positive square coefficients give
$S_k\ge0$, and the positive interpolation weights make
\eqref{cs:ofgm-secondary:decrement} nonpositive.
}

\textbf{Conclusion.}
For $N\geq2$, the decrements in \eqref{cs:ofgm-secondary:decrement} are nonpositive, so $V_N\le V_0=0$.
Using \eqref{cs:ofgm-secondary:normalization}, checked in {Code Cell~39},
proves \eqref{cs:ofgm-secondary:bound}.

\textbf{Small horizons.}
For $N=1$, $D_0=3$, $a_0=b_0=b_1=1/2$, and $x_1=x_0-\nabla f(x_0)/L$.\footnote{{Code Cells~40--41 verify the single base--terminal decrement and endpoint normalization using \eqref{cs:ofgm-secondary:potential-end} and \eqref{cs:ofgm-secondary:S0-one} directly.}}
The terminal definition directly satisfies
\begin{align*}
 V_1={}&\tfrac12\bigl[\mathcal I_f(x_\star,x_0)+\mathcal I_f(x_0,x_1)\bigr]\\
 &-\frac{3L}{8}\left\|\frac{\nabla f(x_0)}L-\frac{x_0-x_\star}{3}\right\|^2
  -\frac L2\left\|\frac{\nabla f(x_1)}L-\frac{x_1-x_\star}{2}\right\|^2.
\end{align*}
Together with \eqref{cs:ofgm-secondary:normalization}, this proves $LR^2/6$ without
an interior matrix. No step divides by $R$.

\end{proof}

\begin{lemma}%
\label{cs:ofgm-secondary:lem:signs}
\normalfont
{%
The interpolation and square coefficients are positive on their respective
index ranges:\footnote{{The unused value $a_{-1}$ is zero;
no $\alpha_k,\beta_k$ is needed at index zero.}}
\[
 a_k>0,\qquad b_k>0,\qquad \alpha_k>0,\qquad \beta_k>0.
\]
The initial square coefficients in \eqref{cs:ofgm-secondary:S0} and
\eqref{cs:ofgm-secondary:S0-one} are also positive.
}
\end{lemma}
\begin{proof}
{%
For $0\le k\le N$, define the scalar $\rho_k\in\mathbb R$ by
\[
 \rho_k=\frac1{\theta_k^2}\sum_{i=0}^{k-1}\theta_i^2,\qquad \rho_0=0.
\]
Its recurrence follows directly from this definition. Telescoping
$\theta_k=\theta_k^2-\theta_{k-1}^2$ and
$2\theta_k^3=\theta_k^4-\theta_{k-1}^4+\theta_k^2$ gives
\begin{equation}\label{cs:ofgm-secondary:moments}
\begin{aligned}
 \rho_k&=\left(1-\frac1{\theta_k}\right)(\rho_{k-1}+1)\quad(k\ge1),\\
 C_{k+1}&=\frac{\theta_k^2+\rho_k+1}{2},\\
 \sum_{i=0}^k h_{N,i}
 &=\frac{\theta_k^2(2\theta_N^2-\theta_k^2+\rho_k+1)}{2\theta_N^2}
 \quad(k<N).
\end{aligned}
\end{equation}
The last identity uses the final-row formula \eqref{cs:ofgm-secondary:h}.
The theta equation gives $1/2<\theta_k-\theta_{k-1}<1$ for $k\ge1$.
Starting from equality at $k=0$, induction in the $\rho$ recurrence yields
\begin{equation}\label{cs:ofgm-secondary:rho-bounds}
 \frac{\theta_k-1}{2}\le\rho_k\le\frac{2(\theta_k-1)}3,
 \qquad 0\le k\le N.
\end{equation}
Indeed, the lower and upper induction steps reduce, respectively,
to the positive quantities
\[
 \frac{(\theta_k-1)(\theta_{k-1}+1-\theta_k)}{2\theta_k}>0,
 \qquad
 \frac{(\theta_k-1)(2\theta_k-2\theta_{k-1}-1)}{3\theta_k}>0.
\]
Moreover, $\rho_{k-1}+1\le(2\theta_{k-1}+1)/3\le\theta_{k-1}<\theta_k$, so
\[
 \rho_k-\rho_{k-1}=1-\frac{\rho_{k-1}+1}{\theta_k}>0\qquad(k\ge1).
\]
In particular $\rho_k>0$ for $k\ge1$.
Subtracting adjacent moment identities gives
\begin{equation}\label{cs:ofgm-secondary:gap-signs}
\begin{gathered}
 \delta_0=1,\qquad
 \delta_k=\frac12\left(\theta_k+1-\frac{\rho_k}{\theta_k-1}\right)>0
 \quad(1\le k\le N),\\
 D_k=D_{k+1}+\delta_k\quad(0\le k<N).
\end{gathered}
\end{equation}
Here positivity follows from \eqref{cs:ofgm-secondary:rho-bounds}.
The moment formula gives $C_{N+1}>1$, so $D_N=C_{N+1}-1+\delta_N>0$ and hence
$D_k\ge D_N>0$ for $0\le k\le N$.
Finally, \eqref{cs:ofgm-secondary:h} gives $h_{N,k}>0$ for $0\le k<N$,
because $\theta_{k-1}^2<\theta_N^2$.
Code Cell~29 checks the moment identities; the preceding induction proves
the bounds used below.

\begin{itemize}
\item \textbf{Positivity of $a_k$.}
For $0\le k<N$, all $h_{N,j}$ in the numerator of \eqref{cs:ofgm-secondary:ab}
are positive, and $D_{k+1}>0$.
Thus $a_k>0$; also $a_N=1$.

\item \textbf{Positivity of $b_k$.}
Direct subtraction in \eqref{cs:ofgm-secondary:ab} gives
\[
 b_k=\frac{h_{N,k}+a_{k-1}\delta_k}{D_{k+1}}>0\quad(1\le k<N),
 \qquad b_0=a_0>0.
\]
The row-sum identity in Lemma~\ref{lem:case-fgm-conjecture-relation} gives
$a_{N-1}=(C_{N+1}-1)/D_N$, so $b_N=1-a_{N-1}=\delta_N/D_N>0$.
Code Cell~29 checks these coefficient identities.

\item \textbf{Positivity of $\alpha_k$.}
For $1\le k\le N-2$, \eqref{cs:ofgm-secondary:margin} gives
$\alpha_k=a_{k+1}/2>0$. For $N\ge2$, the terminal value is
$\alpha_{N-1}=1/2$.
The initial square coefficients $L/(2D_0)$ and $La_1/2$ in
\eqref{cs:ofgm-secondary:S0} are also positive.\footnote{{For $N=1$,
the coefficients in \eqref{cs:ofgm-secondary:S0-one} are $3L/8$ and $L/2$,
both positive.}}

\item \textbf{Positivity of $\beta_k$.}

\textbf{Terminal coefficient.}
For $N\ge2$, the preceding formulas give
\[
 b_N=\frac{\delta_N}{D_N},\qquad a_{N-1}=\frac{C_{N+1}-1}{D_N}.
\]
The moment identities and \eqref{cs:ofgm-secondary:gap-signs} give
$D_N=\rho_N+\theta_N\delta_N$.
Using $m_{N-1}=(\theta_{N-1}-1)/\theta_N$ and
$\theta_{N-1}^2=\theta_N^2-\theta_N$, expansion yields
\begin{equation}\label{cs:ofgm-secondary:terminal-margin}
 a_{N-1}-(m_{N-1}+b_N)^2
 =\frac{\rho_N\bigl((2\theta_{N-1}+\theta_N)D_N-\rho_N\bigr)}{\theta_N^2D_N^2}>0.
\end{equation}
Indeed, $\rho_N,D_N,\delta_N>0$, and
\[
 (2\theta_{N-1}+\theta_N)D_N-\rho_N
 =(2\theta_{N-1}+\theta_N-1)D_N+\theta_N\delta_N>0.
\]
Thus \eqref{cs:ofgm-secondary:terminalS} gives $\beta_{N-1}>0$.
Code Cell~32 verifies \eqref{cs:ofgm-secondary:terminal-margin} and its positive-factor identities.\footnote{{The factorization also holds for $N=1$, where the scalar difference is $1/4$,
although its square coefficient is not used in the one-step formula.}}

\textbf{Interior coefficient.}
The original proof of interior $\beta_k>0$ involved lengthy polynomial-table
calculations. We therefore gave the agent the following additional prompt
to obtain the simpler analytic regrouping proof below.

\colorlet{csadditionalpromptcolor}{.}
\begin{commandbox}[coltext=csadditionalpromptcolor,
 listing options={basicstyle=\ttfamily\small\color{csadditionalpromptcolor},
 columns=fullflexible,breaklines=true,breakatwhitespace=true,breakindent=0pt,breakautoindent=false,keepspaces=true,showstringspaces=false,
 escapeinside={(*@}{@*)}}]
In (*@\texttt{fgm\_conjecture5\_example\_lyap.ipynb}@*), find a simpler analytic proof that (*@\(\beta_k>0\)@*) for all (*@\(N\ge3\)@*) and (*@\(1\le k\le N-2\)@*), starting from its current direct definition and regrouping terms using (*@\(D_k,D_{k+1}\)@*), and (*@\(\delta_k:=D_k-D_{k+1}\)@*). Preserve the existing coefficient definitions and normalization, avoid the large polynomial tables, justify any sequence bounds and denominator signs used, and verify the resulting identities exactly with SymPy. Return a concise proof and executable verification code, clearly identifying any remaining gap.
\end{commandbox}

Fix $N\ge3$ and $1\le k\le N-2$.
We first strengthen the sequence bounds used above:
\begin{equation}\label{cs:ofgm-secondary:beta-sequence-bounds}
\begin{gathered}
 \frac12<\theta_j-\theta_{j-1}<\frac23,\qquad
 \frac35(\theta_j-1)<\rho_j<\frac23(\theta_j-1),\\
 \rho_j-\rho_{j-1}>\frac13,\qquad 1\le j\le N.
\end{gathered}
\end{equation}
For $1\le j\le N$, the recurrence gives
$\theta_j^2-\theta_j-\theta_{j-1}^2=0$ with $\theta_j>0$.
Since $\theta_{j-1}\ge1$,
\[
\begin{aligned}
 \left(\theta_{j-1}+\frac12\right)^2
 -\left(\theta_{j-1}+\frac12\right)-\theta_{j-1}^2
 &=-\frac14<0,\\
 \left(\theta_{j-1}+\frac23\right)^2
 -\left(\theta_{j-1}+\frac23\right)-\theta_{j-1}^2
 &=\frac{\theta_{j-1}}3-\frac29>0.
\end{aligned}
\]
The positive root $\theta_j$ therefore lies between
$\theta_{j-1}+1/2$ and $\theta_{j-1}+2/3$, proving the first bounds.
Starting from $\theta_0=1$ and $\rho_0=0$, the recurrence in
\eqref{cs:ofgm-secondary:moments} transports the non-strict bounds at
index $j-1$ to
\[
\begin{aligned}
 \rho_j-\frac35(\theta_j-1)
 &\ge\frac{\theta_j-1}{5\theta_j}\bigl(2-3(\theta_j-\theta_{j-1})\bigr)>0,\\
 \frac23(\theta_j-1)-\rho_j
 &\ge\frac{\theta_j-1}{3\theta_j}\bigl(2(\theta_j-\theta_{j-1})-1\bigr)>0.
\end{aligned}
\]
Both margins are positive since $\theta_j>1$ and
$1/2<\theta_j-\theta_{j-1}<2/3$. This proves the strict bounds by induction, and
\[
 \rho_j-\rho_{j-1}
 =1-\frac{\rho_{j-1}+1}{\theta_j}
 \ge\frac13+\frac{2(\theta_j-\theta_{j-1})-1}{3\theta_j}>\frac13.
\]

For the fixed $k$, we have $\theta_k>1$,
$\theta_{k+1}^2-\theta_{k+1}=\theta_k^2$, and
$m_k=(\theta_k-1)/\theta_{k+1}\in(0,1)$.
Using $\delta_k=D_k-D_{k+1}=C_{k+1}-C_k$, the moment and gap
identities give
\begin{equation}\label{cs:ofgm-secondary:beta-gap-data}
\begin{aligned}
 \delta_k&=\frac12\left(\theta_k+1-\frac{\rho_k}{\theta_k-1}\right)>0,\qquad
 \delta_{k+1}=1+m_k\delta_k>1,\\
 C_{k+1}&=\theta_k^2-(\theta_k-1)\delta_k,\qquad D_k=D_{k+1}+\delta_k,\\
 D_{k+1}&=\theta_N^2+\rho_N-C_{k+1},\qquad D_{k+2}=D_{k+1}-\delta_{k+1}.
\end{aligned}
\end{equation}
Since the $C_j$ are strictly increasing and $k+2\le N$,
\[
 D_{k+2}=2C_{N+1}-1-C_{k+2}
 \ge C_{k+2}-1=C_{k+1}+\delta_{k+1}-1>C_{k+1}>0.
\]
In particular, $D_k,D_{k+1},D_{k+2},\theta_N^2,\theta_{k+1},\theta_k-1,D_0$
{are positive.}
The moment identities and \eqref{cs:ofgm-secondary:h}--\eqref{cs:ofgm-secondary:ab} yield
\begin{equation}\label{cs:ofgm-secondary:beta-local-identities}
\begin{aligned}
 a_kD_{k+1}&=\theta_k^2\left(1-\frac{(\theta_k-1)\delta_k}{\theta_N^2}\right),\\
 a_kD_{k+1}+(1-h_{N,k})\delta_k&=C_{k+1},\\
 m_k^2a_{k+1}D_{k+2}
 &=(\theta_k-1)^2\left(1-\frac{(\theta_{k+1}-1)\delta_{k+1}}{\theta_N^2}\right).
\end{aligned}
\end{equation}
Code Cell~31 verifies these three identities.
The direct definition \eqref{cs:ofgm-secondary:margin} becomes
\[
 2\beta_k=a_kD_{k+1}^2+(a_k-1)D_{k+1}
 -(h_{N,k}-1)(D_k+D_{k+1})-a_{k+1}m_k^2D_k^2.
\]
Define the real scalars $\chi_k,\varphi_k,\psi_k\in\mathbb R$ by
\[
\begin{aligned}
 \chi_k&=(m_k+2)(\rho_k+1)-\delta_k-(\theta_k-1),\\
 \varphi_k&=\frac{\theta_k-1}{\theta_N^2}\bigl[\chi_kD_{k+2}
 -(\theta_k-1)(1+\rho_{k+1}-\rho_N)(2\delta_k+\delta_{k+1})\bigr],\\
 \psi_k&=D_{k+2}\bigl(C_{k+1}-m_k^2a_{k+1}\delta_k^2\bigr).
\end{aligned}
\]
Substituting \eqref{cs:ofgm-secondary:beta-gap-data}--\eqref{cs:ofgm-secondary:beta-local-identities}
and using the recurrences for $\theta_k$ and $\rho_k$ gives
\begin{equation}\label{cs:ofgm-secondary:beta-regrouping}
 2\beta_kD_{k+2}=D_{k+1}\varphi_k+\psi_k.
\end{equation}
Code Cell~31 checks this regrouping from the direct definition.

Since $(\theta_{k+1}-1)\delta_{k+1}=\theta_{k+1}^2-C_{k+2}\in(0,\theta_{k+1}^2)$
and $\theta_N^2\ge\theta_{k+1}^2$, we have
$0<m_k^2a_{k+1}D_{k+2}<(\theta_k-1)^2$.
Also $C_{k+1}-(\theta_k-1)\delta_k=\rho_k+1>0$ and
$C_{k+1}+(\theta_k-1)\delta_k=\theta_k^2$. Hence
\[
 \psi_k>C_{k+1}^2-(\theta_k-1)^2\delta_k^2=\theta_k^2(\rho_k+1)>0.
\]
For $\varphi_k$, the sequence bounds imply
\[
 \delta_k<\frac{\theta_k}{2}+\frac15,\qquad
 C_{k+1}>\frac{\theta_k^2}{2}+\frac{3\theta_k}{10}+\frac15.
\]
Using $\rho_k>3(\theta_k-1)/5$ and the bound on $\delta_k$ gives
\[
\begin{aligned}
 \chi_k&>m_k\left(\frac{3\theta_k}{5}+\frac25\right)-\frac{3\theta_k}{10}+\frac85\\
  &=\frac{3\theta_k}{10}+1
    +\frac{3(\theta_k-1)}{5\theta_{k+1}}
     \left(\theta_k+\frac23-\theta_{k+1}\right)
   >\frac{3\theta_k}{10}+1>0,
\end{aligned}
\]
because $\theta_{k+1}-\theta_k<2/3$.
The rearrangement in this bound is checked in Code Cell~31.
Since $N\ge k+2$, we have $\rho_N-\rho_{k+1}>1/3$, so
$1+\rho_{k+1}-\rho_N<2/3$; this quantity need not be positive.
Using $\theta_k-1,2\delta_k+\delta_{k+1}>0$ and
$2\delta_k+\delta_{k+1}<3\delta_k+1<3\theta_k/2+8/5$, we obtain
\[
\begin{aligned}
 (\theta_k-1)(1+\rho_{k+1}-\rho_N)(2\delta_k+\delta_{k+1})
 &<\frac23(\theta_k-1)(2\delta_k+\delta_{k+1})\\
 &<(\theta_k-1)\left(\theta_k+\frac{16}{15}\right).
\end{aligned}
\]
Together with $D_{k+2}>C_{k+1}>0$ and the lower bounds above, this proves
\begin{equation}\label{cs:ofgm-secondary:beta-positive-lower-bound}
\begin{aligned}
 &\chi_kD_{k+2}-(\theta_k-1)(1+\rho_{k+1}-\rho_N)(2\delta_k+\delta_{k+1})\\
 &\quad>\left(\frac{3\theta_k}{10}+1\right)
          \left(\frac{\theta_k^2}{2}+\frac{3\theta_k}{10}+\frac15\right)
          -(\theta_k-1)\left(\theta_k+\frac{16}{15}\right)\\
 &\quad=\frac{45\theta_k(\theta_k-41/30)^2+(79/20)\theta_k+380}{300}>0.
\end{aligned}
\end{equation}
Code Cell~31 verifies the final equality.
Thus $\varphi_k,\psi_k>0$, and \eqref{cs:ofgm-secondary:beta-regrouping}, with
$D_{k+1},D_{k+2}>0$, proves $\beta_k>0$ throughout the interior range.
The induction and inequalities above establish the signs for all such $N,k$.

\end{itemize}
}
\end{proof}

\subsection{Proof of Theorem~\ref{thm:case-ogm} (Conjecture~4 in \texorpdfstring{\citet{taylor2017smooth}}{Taylor et al. (2017)})}
\label{app:case-ogm-proof}

\subsubsection{Input prompt}

\begin{commandbox}
/pep-implement

Function: f is convex and L-smooth
Parameters: L, R
Initial condition: ||x_0 - x_star|| <= R, where grad f(x_star) = 0
Performance metric: f(y_N) - f(x_star)
Algorithm: Optimized gradient method with fixed step size 1/L

Parameter sequence: theta_0 = 1 and theta_{k+1} = (1 + sqrt(1 + 4 * theta_k^2)) / 2 for k = 0,...,N-2
Initialization: x_0 = y_0 = z_0
For k = 0,...,N-2:
y_{k+1} = x_k - (1/L) * grad f(x_k)
z_{k+1} = z_k - (2 * theta_k/L) * grad f(x_k)
x_{k+1} = (1 - 1/theta_{k+1}) * y_{k+1} + (1/theta_{k+1}) * z_{k+1}

Final step:
y_N = x_{N-1} - (1/L) * grad f(x_{N-1})
z_N = z_{N-1} - (2 * theta_{N-1}/L) * grad f(x_{N-1})

Reported point: y_N

Conjectured rate: f(y_N) - f(x_star) <= L * R^2 / (2 * (2 * theta_{N-1}^2 + 1))
\end{commandbox}
Then, we repeatedly entered:
\begin{commandbox}
proceed with the next step
\end{commandbox}

\begin{proof}[Proof outline for Theorem~\ref{thm:case-ogm}]
Write \(f_\star=f(x_\star){\in\mathbb R}\) {and} set
\(\tau_N=1/[2(2\theta_{N-1}^2+1)]{\in\mathbb R}\). Define the potentials by
{\(V_{-1}=0\) and}
{%
\begin{equation}\label{cs:ogm:eq:V-interior}
\begin{aligned}
 V_k&=a_{k+1}\bigl(f(x_k)-f_\star\bigr)
 -\tau_NL\norm{x_0-x_\star}^2
 -\frac{a_{k+1}}{2L}\norm{\nabla f(x_k)}^2\\
 &\quad+L\boldsymbol w_k^\top {\vQ_k}\boldsymbol w_k,
 \quad 0\leq k<N.
\end{aligned}
\end{equation}
}
At the reported point, use
\begin{equation}\label{cs:ogm:eq:V-end}
 V_N=f(y_N)-f_\star-\tau_NL\norm{x_0-x_\star}^2.
\end{equation}
The exact decrement identities are\footnote{{For \(N=1\), the middle range is empty; the base and terminal lines give
the two transitions \(V_{-1}\to V_0\to V_1\), with the same boundary formulas.}}
{%
\begin{equation}\label{cs:ogm:eq:decrements}
 V_k-V_{k-1}=
 \begin{cases}
  \displaystyle a_1\mathcal I_f(x_\star,x_0)-S_0, & k=0,\\
  \displaystyle b_{k+1}\mathcal I_f(x_\star,x_k)+a_k\mathcal I_f(x_{k-1},x_k)-S_k, & 1\le k<N,\\
  \displaystyle b_{N+1}\mathcal I_f(x_\star,y_N)+a_N\mathcal I_f(x_{N-1},y_N)-S_N, & k=N.
 \end{cases}
\end{equation}
}
{The full proof establishes these identities, and Lemma~\ref{cs:ogm:lem:signs} proves positivity of every used interpolation and square coefficient.} Since \(\mathcal I_f\leq0\) and
\(S_k\geq0\), they imply \(V_N\leq\cdots\leq V_{-1}=0\).
Finally,
\begin{equation}\label{cs:ogm:eq:normalization}
 f(y_N)-f_\star-\tau_NLR^2
 =V_N+\tau_NL\bigl(\norm{x_0-x_\star}^2-R^2\bigr)\leq0,
\end{equation}
where the last term is nonpositive by the radius assumption.
This proves \eqref{eq:case-ogm-rate}.
\end{proof}

\begin{theorem}[{Theorem~\ref{thm:case-ogm} with full details of the Lyapunov function}]
\label{cs:ogm:thm:main}
\normalfont
Let \(f\in\mathcal F_L\), let \(x_\star{\in\mathcal H}\) minimize \(f\), and write
\(f_\star=f(x_\star){\in\mathbb R}\). Suppose that \(\norm{x_0-x_\star}\leq R\), with \(R\geq0\).
For a fixed integer horizon \(N\geq1\), run the OGM iteration in
\eqref{eq:case-ogm-parameters} and set
\[
 \tau_N=\frac1{2(2\theta_{N-1}^2+1)}{\in\mathbb R}.
\]
{%
{Define the real scalar coefficients \(a_k,b_k\in\mathbb R\) by}
\begin{align}
 a_0&=0,\qquad
 a_k=\frac{\theta_{k-1}^2(2\theta_{N-1}^2-\theta_{k-1}^2)}
 {\theta_{N-1}^2(2\theta_{N-1}^2-\theta_{k-1}^2+1)}
 \quad(1\leq k\leq N),\qquad a_{N+1}=1,\notag\\
 b_k&=a_k-a_{k-1},\qquad 1\leq k\leq N+1.
 \label{cs:ogm:eq:weights}
\end{align}
}

In particular,
\[
 b_{N+1}=\frac1{\theta_{N-1}^2+1}.
\]

For \(0\leq k<N\), define the state
\begin{equation}
 \boldsymbol w_k=
 \begin{bmatrix}
  z_{k+1}-x_\star\\
  y_{k+1}-x_\star
 \end{bmatrix}\in\mathcal H^2
 \label{cs:ogm:eq:w-state}
\end{equation}
{and the symmetric matrix \(\vQ_k\in\mathbb R^{2\times2}\) by}
\begin{equation}
 {\vQ_k}=
 \frac{1}
 {4\theta_{N-1}^2(2\theta_{N-1}^2+1-\theta_k^2)^3}
 \begin{pmatrix}
 (2\theta_{N-1}^2-\theta_k^2)
 (2\theta_{N-1}^2+1-\theta_k^2)^2
 &
 -\theta_k^2(2\theta_{N-1}^2+1-\theta_k^2)
 \\[2mm]
 -\theta_k^2(2\theta_{N-1}^2+1-\theta_k^2)
 &
 \theta_k^2(6\theta_{N-1}^2-4\theta_k^2+3)
 \end{pmatrix}.
 \label{cs:ogm:eq:Q}
\end{equation}
{%
Define the Lyapunov values \(V_k\in\mathbb R\) as in \eqref{cs:ogm:eq:V-interior} and \eqref{cs:ogm:eq:V-end}, \ie, set \(V_{-1}=0\) and
\begin{equation*}
\begin{aligned}
 V_k&=a_{k+1}\bigl(f(x_k)-f(x_\star)\bigr)
 -\tau_NL\norm{x_0-x_\star}^2
 -\frac{a_{k+1}}{2L}\norm{\nabla f(x_k)}^2\\
 &\quad+L\boldsymbol w_k^\top {\vQ_k}\boldsymbol w_k,
 \quad 0\leq k<N.
\end{aligned}
\end{equation*}
At the reported point, set
\begin{equation*}
 V_N=f(y_N)-f(x_\star)-\tau_NL\norm{x_0-x_\star}^2.
\end{equation*}
}

The quadratic form denotes the corresponding sum of inner products of
the vector components of \(\boldsymbol w_k\).

{For \(1\leq k<N\), define the scalars \(P_k,\alpha_k,\beta_k,\rho_k\in\mathbb R\) by}
\begin{align}
 P_k={}&(8\theta_k-6)(\theta_{N-1}^2-\theta_k^2)
 +4(\theta_k-1)^3+10(\theta_k-1)^2
 +12(\theta_k-1)+3,\notag\\
 \alpha_k={}&
 \frac{\theta_k^2P_k}
 {4\theta_{N-1}^2(2\theta_{N-1}^2+1-\theta_k^2)^3},\notag\\
 \beta_k={}&
 \frac{3\theta_k(\theta_k-1)}
 {4\theta_{N-1}^2
  (2\theta_{N-1}^2+1-\theta_{k-1}^2)P_k},\qquad
 \rho_k=
 \frac{2\theta_{N-1}^2+1-2\theta_{k-1}^2}
 {2\theta_{N-1}^2+1-\theta_{k-1}^2}.
 \label{cs:ogm:eq:two-square-weights}
\end{align}
{Define \(M_k\in\mathbb R^{2\times2}\), \(\vc_k,\vd_k\in\mathbb R^2\),
and the square directions \(s_k,t_k\in\mathcal H\), together with \(S_k\), for \(1\le k<N\) by}
{%
\begin{align}
 M_k&=
 \begin{pmatrix}
  1&0\\
  1/\theta_k&(\theta_k-1)/\theta_k
 \end{pmatrix},\notag\\
 \vc_k&=\begin{bmatrix}2\theta_k\\1\end{bmatrix},\qquad
 \vd_k=\frac{a_{k+1}}2
 \begin{bmatrix}1/\theta_k\\(\theta_k-1)/\theta_k\end{bmatrix}
 -\frac{a_k}2\begin{bmatrix}0\\1\end{bmatrix},
 \label{cs:ogm:eq:transition-data}
\end{align}
\begin{align}
 s_k&=\frac{\nabla f(x_k)}L
 +\frac{(-\vd_k+M_k^\top\vQ_k\vc_k)^\top
 \boldsymbol w_{k-1}}{\alpha_k},\quad
 t_k=z_k-x_\star-\rho_k(y_k-x_\star),\notag\\
 S_k&=L\alpha_k\norm{s_k}^2+L\beta_k\norm{t_k}^2.
 \label{cs:ogm:eq:interior-slack}
\end{align}
}
The boundary square terms are
\begin{align}
 S_0={}&
 \frac{L}
 {32\theta_{N-1}^8(2\theta_{N-1}^2+1)}
 \norm{x_0-x_\star
 -\frac{2\theta_{N-1}^2+1}{L}\nabla f(x_0)}^2,
 \notag\\
 S_N={}&
 \frac1{2L}\norm{
 \nabla f(y_N)-\frac{L}{\theta_{N-1}^2+1}(y_N-x_\star)}^2
 +\frac{L}{4(\theta_{N-1}^2+1)}
 \norm{z_N-x_\star
 -\frac{y_N-x_\star}{\theta_{N-1}^2+1}}^2.
 \label{cs:ogm:eq:boundary-slacks}
\end{align}

{%
The square sums satisfy \(S_k\ge0\), and the following decrement identities hold:\footnote{When \(N=1\), the middle range is empty and the two boundary cases give \(V_{-1}\to V_0\to V_1\).}
\begin{equation*}
 V_k-V_{k-1}=
 \begin{cases}
  \displaystyle a_1\mathcal I_f(x_\star,x_0)-S_0, & k=0,\\
  \displaystyle b_{k+1}\mathcal I_f(x_\star,x_k)+a_k\mathcal I_f(x_{k-1},x_k)-S_k, & 1\le k<N,\\
  \displaystyle b_{N+1}\mathcal I_f(x_\star,y_N)+a_N\mathcal I_f(x_{N-1},y_N)-S_N, & k=N.
 \end{cases}
\end{equation*}
The coefficient signs in Lemma~\ref{cs:ogm:lem:signs} and the nonpositive interpolation residuals imply \(V_N\le\cdots\le V_{-1}=0\).
The terminal definition and the radius assumption therefore give
}
\[
 f(y_N)-f_\star
 \leq\tau_NL\norm{x_0-x_\star}^2
 \leq\frac{LR^2}{2(2\theta_{N-1}^2+1)}.
\]
\end{theorem}

\begin{proof}
{%
The companion notebook is
\path{examples_peppy/ogm_conjecture4/ogm_conjecture4_example_lyap.ipynb}.
The notebook's auxiliary vector
\(\theta_kx_k-(\theta_k-1)y_k-x_\star\) equals \(z_k-x_\star\)
(Code Cell~25), so its momentum form uses the same proof state.
}

{%
For \(1\le k<N\), subtracting two consecutive instances of \eqref{cs:ogm:eq:V-interior} cancels the common initial-radius term and gives
\begin{equation}\label{cs:ogm:eq:potential-difference}
\begin{aligned}
V_k-V_{k-1}
 &=a_{k+1}\bigl(f(x_k)-f(x_\star)\bigr)-a_k\bigl(f(x_{k-1})-f(x_\star)\bigr)\\
 &\quad-\frac{a_{k+1}}{2L}\norm{\nabla f(x_k)}^2
 +\frac{a_k}{2L}\norm{\nabla f(x_{k-1})}^2\\
 &\quad+L\boldsymbol w_k^\top\vQ_k\boldsymbol w_k
 -L\boldsymbol w_{k-1}^\top\vQ_{k-1}\boldsymbol w_{k-1}.
\end{aligned}
\end{equation}
Our goal is to identify this expression with the interior case of \eqref{cs:ogm:eq:decrements}. We first factor the quadratic remainder, then match the weighted interpolation residuals; the boundary cases are treated separately.
}

\medskip
\noindent\textbf{Step 1: factorization of the interior slack.}
For \(1\leq k<N\), Steps 1 and 2 prove the second line of
\eqref{cs:ogm:eq:decrements}, namely,
\[
 V_k-V_{k-1}
 =b_{k+1}\mathcal I_f(x_\star,x_k)
 +a_k\mathcal I_f(x_{k-1},x_k)-S_k.
\]
Step 1 factors a quadratic remainder into the squares in
\eqref{cs:ogm:eq:interior-slack}; Step 2 identifies that remainder by
subtracting the potential difference from the weighted interpolation residuals.
For \(1\leq k<N\), the OGM updates and the mixing identity for \(x_k\)
give
\begin{equation}
 \boldsymbol w_k
 ={M_k}\boldsymbol w_{k-1}
 -{\vc_k}\frac{\nabla f(x_k)}L,
 \qquad
 x_k-x_\star
 =\frac{(\boldsymbol w_{k-1})_1
 +(\theta_k-1)(\boldsymbol w_{k-1})_2}{\theta_k}.
 \label{cs:ogm:eq:state-transition}
\end{equation}
{%
Define the symmetric matrix \(\vS_k\in\mathbb R^{3\times3}\) by
\begin{equation}\label{cs:ogm:eq:slack-matrix}
 \vS_k=
 \begin{pmatrix}
 \vQ_{k-1}-M_k^\top\vQ_kM_k&-\vd_k+M_k^\top\vQ_k\vc_k\\
 (-\vd_k+M_k^\top\vQ_k\vc_k)^\top&
 a_{k+1}-\vc_k^\top\vQ_k\vc_k
 \end{pmatrix}.
\end{equation}
The claim in this step is the quadratic-form identity \eqref{cs:ogm:eq:factorization}.
}
Direct substitution of \eqref{cs:ogm:eq:Q} and
\eqref{cs:ogm:eq:transition-data}, using
\(\theta_{k-1}^2=\theta_k^2-\theta_k\), gives
{%
\begin{equation}\label{cs:ogm:eq:matrix-factorization}
\begin{aligned}
 \frac1{L^2}\norm{\nabla f(x_k)}^2&:\quad
 a_{k+1}-\vc_k^\top \vQ_k\vc_k=\alpha_k,\\
 \boldsymbol w_{k-1}^\top(\cdot)\boldsymbol w_{k-1}&:\quad
 \begin{aligned}[t]
 \vQ_{k-1}-M_k^\top \vQ_kM_k&=\frac{(-\vd_k+M_k^\top\vQ_k\vc_k)
 (-\vd_k+M_k^\top\vQ_k\vc_k)^\top}{\alpha_k}\\
 &\quad+\beta_k\begin{bmatrix}1\\-\rho_k\end{bmatrix}
 \begin{bmatrix}1\\-\rho_k\end{bmatrix}^{\!\top}.
 \end{aligned}
\end{aligned}
\end{equation}
}
These are rational identities in
\(\theta_k\) and \(\theta_{N-1}\).
For the first identity, the numerator after collecting terms is
\(\theta_k^2P_k\); expanding the entries of the second identity verifies
its four scalar entries. {Code Cell~42} verifies both identities against the
literal formulas and every entry of the resulting \(3\)-by-\(3\) factorization.
{Evaluating that factorization and using the definitions of \(s_k,t_k\) in \eqref{cs:ogm:eq:interior-slack} gives}
{%
\begin{equation}\label{cs:ogm:eq:factorization}
 L\begin{bmatrix}\boldsymbol w_{k-1}\\\nabla f(x_k)/L\end{bmatrix}^{\!\top}\vS_k\begin{bmatrix}\boldsymbol w_{k-1}\\\nabla f(x_k)/L\end{bmatrix}
 =L\alpha_k\norm{s_k}^2+L\beta_k\norm{t_k}^2=S_k,
\end{equation}
}
as claimed.

\medskip
\noindent\textbf{Step 2: verification of the interior decrement identity.}
{For \(1\leq k<N\), it remains to identify the left-hand side of
\eqref{cs:ogm:eq:factorization} with the weighted interpolation residuals
minus \(V_k-V_{k-1}\).}
{%
To match \eqref{cs:ogm:eq:potential-difference}, use \(b_{k+1}=a_{k+1}-a_k\) and \(\nabla f(x_\star)=0\) to expand the residuals:
\begin{equation}\label{cs:ogm:eq:residual-expansion}
\begin{aligned}
&b_{k+1}\mathcal I_f(x_\star,x_k)+a_k\mathcal I_f(x_{k-1},x_k)\\
 &\quad=a_{k+1}\bigl(f(x_k)-f(x_\star)\bigr)-a_k\bigl(f(x_{k-1})-f(x_\star)\bigr)\\
 &\qquad+\frac{a_k}{2L}\norm{\nabla f(x_{k-1})}^2
 +\frac{a_{k+1}}{2L}\norm{\nabla f(x_k)}^2\\
 &\qquad+\csinner{\nabla f(x_k)}
 {a_k\left(x_{k-1}-x_k-\frac{\nabla f(x_{k-1})}L\right)+b_{k+1}(x_\star-x_k)}.
\end{aligned}
\end{equation}
The function-value terms already agree. The remaining vector coefficient is determined by the gradient step and the mixing identity in \eqref{cs:ogm:eq:state-transition}:
\begin{equation}\label{cs:ogm:eq:residual-vector}
\begin{aligned}
&a_k\left(x_{k-1}-x_k-\frac{\nabla f(x_{k-1})}L\right)+b_{k+1}(x_\star-x_k)\\
 &\quad=a_k(y_k-x_\star)-a_{k+1}(x_k-x_\star)\\
 &\quad=a_k(\boldsymbol w_{k-1})_2
 -\frac{a_{k+1}}{\theta_k}\bigl[(\boldsymbol w_{k-1})_1+(\theta_k-1)(\boldsymbol w_{k-1})_2\bigr]
 =-2\vd_k^\top\boldsymbol w_{k-1}.
\end{aligned}
\end{equation}
Subtracting \eqref{cs:ogm:eq:potential-difference} cancels the function values and the old-gradient norm. The current-gradient terms give \((a_{k+1}/L)\norm{\nabla f(x_k)}^2-2\csinner{\nabla f(x_k)}{\vd_k^\top\boldsymbol w_{k-1}}\).
}
Substituting \eqref{cs:ogm:eq:state-transition} in the remaining terms gives
{%
\begin{equation}\label{cs:ogm:eq:interpolation-bridge}
\begin{aligned}
 &b_{k+1}\mathcal I_f(x_\star,x_k)
 +a_k\mathcal I_f(x_{k-1},x_k)-(V_k-V_{k-1})\\
 &\quad=L\begin{bmatrix}\boldsymbol w_{k-1}\\\nabla f(x_k)/L\end{bmatrix}^{\!\top}\vS_k\begin{bmatrix}\boldsymbol w_{k-1}\\\nabla f(x_k)/L\end{bmatrix}=S_k,
\end{aligned}
\end{equation}
}
where the last equality is the factorization proved in Step 1.
Rearranging proves the second line of \eqref{cs:ogm:eq:decrements} for
every \(1\leq k<N\), including \(k=1\) whenever this range is nonempty.
{Code Cell~44} checks the bridge and full decrement in every Gram,
function-value, and constant coordinate. The remaining base and terminal
cases follow next; {positivity is established in Lemma~\ref{cs:ogm:lem:signs}.}

\medskip
\noindent\textbf{Terminal index.}
This verifies the last line of \eqref{cs:ogm:eq:decrements}, using
\eqref{cs:ogm:eq:V-end} and the terminal square sum in
\eqref{cs:ogm:eq:boundary-slacks}, for every \(N\geq1\).
The coefficient definitions give
\[
 a_N=\frac{\theta_{N-1}^2}{\theta_{N-1}^2+1},
 \qquad
 b_{N+1}=\frac1{\theta_{N-1}^2+1}.
\]
Using \(y_N=x_{N-1}-\nabla f(x_{N-1})/L\) cancels the old gradient
terms in the terminal interpolation residuals. Completing the new gradient
square in \(S_N\) leaves the exact position identity
\begin{equation}\label{cs:ogm:eq:terminal-factorization}
 {\vQ_{N-1}}
 -\begin{pmatrix}
  0&0\\
  0&1/[2(\theta_{N-1}^2+1)^2]
 \end{pmatrix}
 =
 \frac1{4(\theta_{N-1}^2+1)}
 \begin{bmatrix}1\\-1/(\theta_{N-1}^2+1)\end{bmatrix}
 {\begin{bmatrix}1\\-1/(\theta_{N-1}^2+1)\end{bmatrix}^{\!\top}}.
\end{equation}
{Code Cell~42} verifies this position factorization, and {Code Cell~46}
checks the complete terminal decrement and endpoint normalization.

\medskip
\noindent\textbf{Base index.}
{This verifies the first line of \eqref{cs:ogm:eq:decrements}, with
\(V_{-1}=0\) and the initial square in \eqref{cs:ogm:eq:boundary-slacks}.}\footnote{{The same base calculation applies when \(N=1\).}}
Since
\[
 \boldsymbol w_0=
 \begin{pmatrix}1&-2\\1&-1\end{pmatrix}
 \begin{bmatrix}x_0-x_\star\\ \nabla f(x_0)/L\end{bmatrix},
\]
direct substitution gives
\begin{equation}\label{cs:ogm:eq:base-factorization}
\begin{aligned}
&\begin{pmatrix}
  \tau_N&-a_1/2\\
  -a_1/2&a_1
 \end{pmatrix}
 -\begin{pmatrix}1&1\\-2&-1\end{pmatrix}
 {\vQ_0}
 \begin{pmatrix}1&-2\\1&-1\end{pmatrix}\\
&\qquad=
 \frac1{32\theta_{N-1}^8(2\theta_{N-1}^2+1)}
 \begin{bmatrix}1\\-(2\theta_{N-1}^2+1)\end{bmatrix}
 {\begin{bmatrix}1\\-(2\theta_{N-1}^2+1)\end{bmatrix}^{\!\top}}.
\end{aligned}
\end{equation}
Expanding
\(a_1\mathcal I_f(x_\star,x_0)-V_0\)
produces the quadratic form on the left multiplied by \(L\), and hence
equals \(S_0\). This proves the base line of
\eqref{cs:ogm:eq:decrements}. {Code Cell~42} checks the matrix factorization,
and {Code Cell~45} checks the full base decrement.

\paragraph{Positivity of the coefficients.}
{%
Lemma~\ref{cs:ogm:lem:signs} proves the required coefficient signs
and denominator bounds. Its positive square coefficients give \(S_k\geq0\)
at the interior and boundary indices.
}

\paragraph{Conclusion.}
Each interpolation residual in \eqref{cs:ogm:eq:decrements} is
nonpositive. {The coefficient signs from Lemma~\ref{cs:ogm:lem:signs} and the square decompositions above therefore}
give
\[
 V_N\leq V_{N-1}\leq\cdots\leq V_0\leq V_{-1}=0.
\]
Using \eqref{cs:ogm:eq:normalization}, checked in {Code Cell~46}, yields\footnote{{When \(N=1\), there is no interior transition; the base and terminal
identities above remain valid and give the same conclusion.
{Code Cells~45--46} also check these two transitions on the shared
one-step trajectory, where \(\theta_0=1\) and \(\tau_1=1/6\).}}
\[
 f(y_N)-f_\star
 \leq\tau_NL\norm{x_0-x_\star}^2
 \leq\frac{LR^2}{2(2\theta_{N-1}^2+1)}.
\]
\end{proof}

\begin{lemma}%
\label{cs:ogm:lem:signs}
\normalfont
{%
The coefficients are positive on their respective index ranges:\footnote{{The unused value \(a_0\) is zero.}}
\[
 a_k>0,\qquad b_k>0,\qquad \alpha_k>0,\qquad \beta_k>0.
\]
The three boundary-square coefficients in \eqref{cs:ogm:eq:boundary-slacks} are also positive.
}
\end{lemma}
\begin{proof}
{%
For \(0\leq k<N\), \(\theta_{N-1}^2\geq\theta_k^2\geq1\), so
\[
 2\theta_{N-1}^2+1-\theta_k^2\geq\theta_{N-1}^2+1>0.
\]
For \(1\leq k<N\), \(\theta_k>1\). Every term in the formula for \(P_k\)
in \eqref{cs:ogm:eq:two-square-weights} is nonnegative, and its constant
term is positive. Hence \(P_k>0\).
Code Cell~42 verifies the derivative, denominator,
and polynomial identities used below.

\begin{itemize}
\item \textbf{Positivity of \(a_k\).}
For \(1\leq k\leq N\),
\(2\theta_{N-1}^2-\theta_{k-1}^2\geq\theta_{N-1}^2>0\), so
\eqref{cs:ogm:eq:weights} gives \(a_k>0\). Also, \(a_{N+1}=1\).

\item \textbf{Positivity of \(b_k\).}
{Define \(p:[0,\theta_{N-1}^2]\to\mathbb R\) by}
\[
 p(t)=
 \frac{t(2\theta_{N-1}^2-t)}
 {\theta_{N-1}^2(2\theta_{N-1}^2+1-t)}.
\]
Then
\[
 p'(t)=
 \frac{1-(2\theta_{N-1}^2+1)/
 (2\theta_{N-1}^2+1-t)^2}
 {\theta_{N-1}^2}>0.
\]
Indeed,
\(2\theta_{N-1}^2+1-t\geq\theta_{N-1}^2+1\), while
\((\theta_{N-1}^2+1)^2-(2\theta_{N-1}^2+1)
=\theta_{N-1}^4>0\).
Since \(p(0)=a_0=0\), \(a_k=p(\theta_{k-1}^2)\) for \(1\leq k\leq N\),
and \(\theta_k\) strictly increases, \(b_k=a_k-a_{k-1}>0\) through \(k=N\).
At the terminal index, \(b_{N+1}=1/(\theta_{N-1}^2+1)>0\).

\item \textbf{Positivity of \(\alpha_k\).}
For \(1\leq k<N\), the numerator \(\theta_k^2P_k\) and denominator in
\eqref{cs:ogm:eq:two-square-weights} are positive, so \(\alpha_k>0\).

\item \textbf{Positivity of \(\beta_k\).}
For \(1\leq k<N\), \(3\theta_k(\theta_k-1)>0\), and every denominator
factor in \eqref{cs:ogm:eq:two-square-weights} is positive. Thus \(\beta_k>0\).
\end{itemize}
Finally, \(L>0\) and \(\theta_{N-1}^2\geq1\) make all three coefficients
in \eqref{cs:ogm:eq:boundary-slacks} positive.\footnote{{This also includes \(N=1\).}}
}
\end{proof}

\begingroup
\Needspace{10\baselineskip}
\subsection{Proof of Theorem~\ref{thm:case-fista} (nonsmooth FPGM1 conjecture in Table~1 of \texorpdfstring{\citet{TaylorHendrickxGlineur2017_exacta}}{Taylor et al. (2017)})}
\label{app:case-fista-proof}

\subsubsection{Input prompt}

\begin{commandbox}
/pep-implement

Function: F = f + g, where f is convex and L-smooth and g is closed, proper, and convex
Parameters: L, R
Initial condition: ||x_0 - x_star|| <= R, where 0 in grad f(x_star) + partial g(x_star)
Performance metric: F(y_N) - F(x_star)
Algorithm: FISTA/FPGM1 with fixed step size 1/L and inertial parameter k/(k+3)

Initialization: y_0 = x_0
For k >= 0:
y_{k+1} = prox_{g/L}(x_k - (1/L) * grad f(x_k))
x_{k+1} = y_{k+1} + (k/(k+3)) * (y_{k+1} - y_k)

Conjectured rate: F(y_N) - F(x_star) <= 2 * L * R^2 / (N^2 + 5*N + 2)
\end{commandbox}
Then, we repeatedly entered:
\begin{commandbox}
proceed with the next step
\end{commandbox}

\subsubsection{Proof outline}

{Define \(V_k\in\mathbb R\) by}
{\(V_{0}=0\) and}
{%
\begin{equation}\label{cs:fista:eq:V-endpoints}
 V_N=F(y_N)-F(x_\star)-\tau_NL\norm{x_0-x_\star}^2,
\end{equation}
}
and, for \(1\leq k<N\),
\begin{equation}\label{cs:fista:eq:V-interior}
\begin{aligned}
 V_k&=a_k\bigl[f(x_k)-f(x_\star)\bigr]
 +\widetilde a_k\bigl[g(y_k)-g(x_\star)\bigr]
 -\tau_NL\norm{x_0-x_\star}^2\\
 &\quad-2\csinner{\nabla f(x_k)}{{\vd_k}^\top\boldsymbol w_k}
 +L\boldsymbol w_k^\top {\vQ_k}\boldsymbol w_k.
\end{aligned}
\end{equation}
Thus the interior checkpoint uses the smooth oracle at \(x_k\) and the
proximal point \(y_k\). No value of \(g(y_0)\) is needed.
{To include the terminal step, define \(\widehat x_{k+1}\in\mathcal H\) by}
\[
 \widehat x_{k+1}=\begin{cases}x_{k+1},&0\leq k<N-1,\\
 y_N,&k=N-1.\end{cases}
\]
{For notational convenience, define \(\mathcal T_k\in\mathbb R\) using the nonpositive interpolation residuals defined below:}
\begin{equation}\label{cs:fista:eq:weighted-residual}
\begin{aligned}
 \mathcal T_k
 &=a_k\mathcal I_f(x_k,\widehat x_{k+1})
 +b_{k+1}\mathcal I_f(x_\star,\widehat x_{k+1})\\
 &\quad+\widetilde a_k\mathcal I_g(y_k,y_{k+1})
 +\widetilde b_{k+1}\mathcal I_g(x_\star,y_{k+1}),
 \qquad 0\leq k<N,
\end{aligned}
\end{equation}
where the term with \(\widetilde a_0=0\) is omitted, and {define \(S_k\in\mathbb R\) by}
{%
\begin{equation}\label{cs:fista:eq:slack}
S_k=\frac{La_k}{2}\norm{s_k}^2 +\frac{b_{k+1}}{2L}\norm{\nabla f(\widehat x_{k+1})-\nabla f(x_\star)}^2 +L\alpha_k\norm{Y_{k+1}-Y_k}^2.
\end{equation}
}
For \(N\geq2\), the proof establishes the exact decrements
{%
\begin{equation}\label{eq:case-fista-increment}
 V_{k+1}-V_k=
 \begin{cases}
  \displaystyle \mathcal T_0-S_0+b_0\mathcal I_f(x_\star,x_0)-\frac{b_0}{2L}\norm{\nabla f(x_0)-\nabla f(x_\star)}^2, & k=0,\\
  \displaystyle \mathcal T_k-S_k, & 1\le k<N-1,\\
  \displaystyle \mathcal T_{N-1}-S_{N-1}, & k=N-1.
 \end{cases}
\end{equation}
}
For \(N=1\), the initial and terminal conventions apply to the same step:
\begin{equation}\label{cs:fista:eq:N1-decrement}
 V_1-V_0=\mathcal T_0-S_0+b_0\mathcal I_f(x_\star,x_0)
 -\frac{b_0}{2L}\norm{\nabla f(x_0)-\nabla f(x_\star)}^2,
 \qquad \widehat x_1=y_1.
\end{equation}
{Lemma~\ref{cs:fista:lem:signs} gives nonnegative interpolation weights and $S_k\geq0$, so these}
identities imply \(V_N\leq\cdots\leq V_0=0\). The endpoint identity
\begin{equation}\label{cs:fista:eq:normalization}
 F(y_N)-F(x_\star)-\tau_NLR^2
 =V_N+\tau_NL\bigl(\norm{x_0-x_\star}^2-R^2\bigr)\leq0
\end{equation}
then gives the claimed bound, including \(R=0\).

\subsubsection{Full proof}

\begin{theorem}[{Theorem~\ref{thm:case-fista} with full details of the Lyapunov function}]
\label{cs:fista:thm:main}
\normalfont
Let \(L>0\), let \(f\in\mathcal F_L\), and let
\(g:\mathcal H\to\mathbb R\cup\{+\infty\}\) be proper, closed, and
convex, and {set \(F=f+g:\mathcal H\to\mathbb R\cup\{+\infty\}\).} Fix an integer \(N\geq1\) and \(R\geq0\),
and suppose that
\(0\in\nabla f(x_\star)+\partial g(x_\star)\) and
\(\norm{x_0-x_\star}\leq R\). Let the iterates be generated by
\ref{eq:case-fista}. For this fixed horizon, set
\[
 \tau_N=\frac{2}{N^2+5N+2}{\in\mathbb R},
\]
and use the coefficients \(a_k,b_k{\in\mathbb R}\) from
\eqref{cs:sfgm:eq:a}--\eqref{cs:sfgm:eq:b}.
Thus, as in Theorem~\ref{cs:sfgm:thm:main}, \(a_k\) weights the
smooth function value and \(b_k=a_k-a_{k-1}\).
{For the nonsmooth function value and \(1\leq k\leq N\), define the scalars \(D_k,\widetilde a_k,\widetilde b_k\in\mathbb R\) by}
\begin{align*}
 D_0&=\frac{N^2+5N+2}{4},\qquad
 D_k=\frac{2N^2+10N+2-k^2-5k}{8},\\
 \widetilde a_0&=0,\qquad
 \widetilde a_k=\frac{D_k+1}{D_k}a_{k-1},\qquad
 \widetilde b_k=\widetilde a_k-\widetilde a_{k-1}.
\end{align*}
In particular, \(D_N=1/(4\tau_N)\) and
\(a_N=\widetilde a_N=1\). {For \(0\leq k<N\), define \(\delta_k\in\mathbb R\) by}
\[
 \delta_k=\frac{(k+2)(N-k)(N+k+1)}{2N(N+1)}.
\]
For \(0\leq k<N-1\), define \(\alpha_k{\in\mathbb R}\) by the first line below;
the second line defines the terminal weight \(\alpha_{N-1}{\in\mathbb R}\):
\begin{equation}
\begin{aligned}
 \alpha_k
 &=\frac{a_k(D_{k+1}+1)^2}{2}-D_N
   -\frac{kD_k\delta_{k+1}}{k+3}
   -\frac{a_{k+1}k^2D_k^2}{2(k+3)^2},\\
 \alpha_{N-1}
 &=\frac{D_N(D_N-1)}2
   =\frac{(N-1)(N+6)(N^2+5N+2)}{128}.
\end{aligned}
\label{eq:case-fista-alpha}
\end{equation}
These definitions give \(\alpha_0=0\).\footnote{{This includes the terminal definition when \(N=1\).}}

{For \(0\leq k\leq N\), define the normalized position \(Y_k\in\mathcal H\) by
\[
 Y_k=\frac{y_k-x_\star}{D_k}.
\]
}

For \(1\leq k<N\), {define the state, coefficient vector, and matrix \(\vQ_k\in\mathbb R^{2\times2}\) by}
{%
\begin{align*}
 \boldsymbol w_k&=
 \begin{bmatrix}Y_k\\[2mm]x_k-x_\star\end{bmatrix}\in\mathcal H^2,
 \qquad \vd_k=\frac12
 \begin{bmatrix}-\widetilde a_kD_k\\a_k\end{bmatrix}\in\mathbb R^2,\\
 \vQ_k&=
 \begin{bmatrix}
 D_N-\widetilde a_kD_k+
       \dfrac{(\widetilde a_kD_k)^2-\delta_k^2}{2a_k}
     &-\dfrac{\widetilde a_kD_k}{2}\\[2mm]
 -\dfrac{\widetilde a_kD_k}{2}&\dfrac{a_k}{2}
 \end{bmatrix}.
\end{align*}
}
{%
Define the Lyapunov values \(V_k\in\mathbb R\) as in \eqref{cs:fista:eq:V-endpoints} and \eqref{cs:fista:eq:V-interior}, \ie, set \(V_{0}=0\) and
\begin{equation*}
 V_N=F(y_N)-F(x_\star)-\tau_NL\norm{x_0-x_\star}^2,
\end{equation*}
and, for \(1\le k<N\),
\begin{equation*}
\begin{aligned}
 V_k&=a_k\bigl[f(x_k)-f(x_\star)\bigr]
 +\widetilde a_k\bigl[g(y_k)-g(x_\star)\bigr]
 -\tau_NL\norm{x_0-x_\star}^2\\
 &\quad-2\csinner{\nabla f(x_k)}{{\vd_k}^\top\boldsymbol w_k}
 +L\boldsymbol w_k^\top {\vQ_k}\boldsymbol w_k.
\end{aligned}
\end{equation*}
}

Use the smooth interpolation residual \(\mathcal I_f\) from
\eqref{eq:case-interpolation-residual}. Proximal optimality gives
\(L(x_k-y_{k+1})-\nabla f(x_k)\in\partial g(y_{k+1})\), so {define the real-valued residual \(\mathcal I_g(\cdot,y_{k+1}):\operatorname{dom}g\to\mathbb R\) by}
\[
 \mathcal I_g(p,y_{k+1})
 =g(y_{k+1})-g(p)
  +\csinner{L(x_k-y_{k+1})-\nabla f(x_k)}{p-y_{k+1}}\leq0,
 \qquad p\in\operatorname{dom}g.
\]
The subgradient is attached to the indexed occurrence of \(y_{k+1}\),
even when iterates coincide. For \(0\leq k<N-1\), define \(s_k{\in\mathcal H}\)
by the first line below, and set \(s_{N-1}{\in\mathcal H}\) by the second:
{%
\begin{align*}
 s_k&=\frac{\nabla f(x_k)-\nabla f(x_{k+1})}{L}
       +Y_{k+1}-(x_k-y_{k+1}),\\
 s_{N-1}&=\frac{\nabla f(x_{N-1})-\nabla f(y_N)}{L}
       +Y_N-(x_{N-1}-y_N).
\end{align*}
}
{%
For the terminal step, use
\[
 \widehat x_{k+1}=\begin{cases}x_{k+1},&0\le k<N-1,\\y_N,&k=N-1.\end{cases}
\]
For notational convenience, define \(\mathcal T_k,S_k\in\mathbb R\), for \(0\le k<N\), by
\begin{equation*}
\begin{aligned}
 \mathcal T_k
 &=a_k\mathcal I_f(x_k,\widehat x_{k+1})
 +b_{k+1}\mathcal I_f(x_\star,\widehat x_{k+1})\\
 &\quad+\widetilde a_k\mathcal I_g(y_k,y_{k+1})
 +\widetilde b_{k+1}\mathcal I_g(x_\star,y_{k+1}),
 \qquad 0\leq k<N,
\end{aligned}
\end{equation*}
where the term with \(\widetilde a_0=0\) is omitted, and
\begin{equation*}
S_k=\frac{La_k}{2}\norm{s_k}^2 +\frac{b_{k+1}}{2L}\norm{\nabla f(\widehat x_{k+1})-\nabla f(x_\star)}^2 +L\alpha_k\norm{Y_{k+1}-Y_k}^2.
\end{equation*}
Then \(S_k\ge0\), and, for \(N\ge2\), the following decrement identities hold:\footnote{For \(N=1\), the first case holds with \(\widehat x_1=y_1\), so the initial and terminal conventions apply to the same step.}
\begin{equation*}
 V_{k+1}-V_k=
 \begin{cases}
  \displaystyle \mathcal T_0-S_0+b_0\mathcal I_f(x_\star,x_0)-\frac{b_0}{2L}\norm{\nabla f(x_0)-\nabla f(x_\star)}^2, & k=0,\\
  \displaystyle \mathcal T_k-S_k, & 1\le k<N-1,\\
  \displaystyle \mathcal T_{N-1}-S_{N-1}, & k=N-1.
 \end{cases}
\end{equation*}
The coefficient signs in Lemma~\ref{cs:fista:lem:signs} and the nonpositive interpolation residuals imply \(V_N\le\cdots\le V_0=0\).
The terminal definition and the radius assumption therefore give
}
\[
 F(y_N)-F(x_\star)\leq\tau_NL\norm{x_0-x_\star}^2
 \leq\frac{2LR^2}{N^2+5N+2}.
\]
\end{theorem}

\begin{proof}
{%
The companion notebook is
\path{examples_peppy/fista_rational_tight/fista_rational_tight_example_lyap.ipynb}.
The notebook retains its historical potential at the proximal checkpoints
separately. Code Cell~18 verifies the interpolation redistribution relating
it to the potential used here; the checks below use
\eqref{cs:fista:eq:V-endpoints}--\eqref{cs:fista:eq:V-interior} literally.
}

\paragraph{Coefficient identities.}
The weight identities below hold for \(0\leq k<N\), and the scale
recurrence holds for \(1\leq k<N\). {Code Cell~18} verifies each identity
by exact substitution in the literal coefficient formulas:
\begin{align}
 (D_{k+1}+1)a_k&=\widetilde a_{k+1}D_{k+1},&
 \widetilde a_{k+1}D_{k+1}-\widetilde a_kD_k&=\delta_k,
       \label{eq:case-fista-weight-identities}\\
 \frac{2k+1}{k+2}D_k-\frac{k-1}{k+2}D_{k-1}
       &=D_{k+1}+1.
       \label{eq:case-fista-scale-recurrence}
\end{align}
At startup, \(D_0=D_1+1=2D_N\).
For \(1\leq k<N-1\), we also have
\begin{equation}
 \widetilde a_kD_k+\delta_k(D_{k+1}+1)-2D_N
       =\frac{kD_k\delta_{k+1}}{k+3}.
 \label{eq:case-fista-adjacent-identity}
\end{equation}
The FISTA recurrence and \eqref{eq:case-fista-scale-recurrence}
therefore give, for \(1\leq k<N\),
{%
\begin{equation}
 Y_{k+1}-(x_k-y_{k+1})
 =(D_{k+1}+1)\left(Y_{k+1}-Y_k\right)
 -\frac{(k-1)D_{k-1}}{k+2}\left(Y_k-Y_{k-1}\right).
 \label{eq:case-fista-position-recurrence}
\end{equation}
}

{%
For \(1\le k<N-1\), subtracting two consecutive instances of \eqref{cs:fista:eq:V-interior} cancels the common initial-radius term and gives
\begin{equation}\label{cs:fista:eq:potential-difference}
\begin{aligned}
V_{k+1}-V_k
 &=a_{k+1}\bigl(f(x_{k+1})-f(x_\star)\bigr)-a_k\bigl(f(x_k)-f(x_\star)\bigr)\\
 &\quad+\widetilde a_{k+1}\bigl(g(y_{k+1})-g(x_\star)\bigr)
 -\widetilde a_k\bigl(g(y_k)-g(x_\star)\bigr)\\
 &\quad-2\csinner{\nabla f(x_{k+1})}{\vd_{k+1}^\top\boldsymbol w_{k+1}}
 +2\csinner{\nabla f(x_k)}{\vd_k^\top\boldsymbol w_k}\\
 &\quad+L\boldsymbol w_{k+1}^\top\vQ_{k+1}\boldsymbol w_{k+1}
 -L\boldsymbol w_k^\top\vQ_k\boldsymbol w_k.
\end{aligned}
\end{equation}
Our goal is to identify this expression with the interior case of \eqref{eq:case-fista-increment}. We first factor the quadratic remainder, then match the weighted interpolation residuals; the boundary cases are treated separately.
}
In particular, neither \(\nabla f(x_\star)\) nor the selected subgradient of
\(g\) at \(x_\star\) is assumed to vanish separately.

\medskip
\noindent\textbf{Step 1: factorization of the interior remainder.}
For \(1\leq k<N-1\), Steps 1 and 2 prove the second line of
\eqref{eq:case-fista-increment}, namely,
\[
\begin{aligned}
 V_{k+1}-V_k
 &=a_k\mathcal I_f(x_k,x_{k+1})
 +b_{k+1}\mathcal I_f(x_\star,x_{k+1})\\
 &\quad+\widetilde a_k\mathcal I_g(y_k,y_{k+1})
 +\widetilde b_{k+1}\mathcal I_g(x_\star,y_{k+1})-S_k.
\end{aligned}
\]
{Step 1 factors the quadratic form associated with a symmetric matrix; Step 2
identifies it by expanding the interpolation residuals and the potential difference.}
{Define \(\boldsymbol r_k\in\mathcal H^5\) and \(M_k\in\mathbb R^{2\times2}\) by
\[
 \boldsymbol r_k=\begin{bmatrix}Y_k\\Y_{k+1}\\x_k-y_{k+1}\\(\nabla f(x_k)-\nabla f(x_{k+1}))/L\\(\nabla f(x_{k+1})-\nabla f(x_\star))/L\end{bmatrix},\qquad
 M_k=\begin{pmatrix}0&1\\-kD_k/(k+3)&(2k+3)D_{k+1}/(k+3)\end{pmatrix}.
\]
The FISTA update gives
\[
 \boldsymbol w_{k+1}=M_k\begin{bmatrix}Y_k\\Y_{k+1}\end{bmatrix}.
\]
Define the symmetric matrix \(\vS_k\in\mathbb R^{5\times5}\) by
\begin{equation}\label{cs:fista:eq:slack-matrix}
\begin{aligned}
 \vS_k={}&\begin{pmatrix}
 \operatorname{diag}(1,D_{k+1})\vQ_k\operatorname{diag}(1,D_{k+1})-M_k^\top\vQ_{k+1}M_k&0_{2\times3}\\
 0_{3\times2}&0_{3\times3}
 \end{pmatrix}\\
 &+\frac12\begin{pmatrix}0&0&0&0&0\\0&0&-a_k&a_k&0\\0&-a_k&a_k&-a_k&0\\0&a_k&-a_k&a_k&0\\0&0&0&0&b_{k+1}\end{pmatrix}.
\end{aligned}
\end{equation}
The identities \((\vQ_k)_{12}=-\widetilde a_kD_k/2\), \((\vQ_k)_{22}=a_k/2\), and
\(\widetilde a_{k+1}D_{k+1}=a_k(D_{k+1}+1)\) give the displayed cross terms.
The claim of Step 1 is \(L\boldsymbol r_k^\top\vS_k\boldsymbol r_k=S_k\).}
For \(1\leq k<N\), the definitions imply
{%
\begin{align}
 2\vd_k^\top\boldsymbol w_k
 &=a_k(x_k-x_\star)-\widetilde a_k(y_k-x_\star)
 =\delta_kY_k+\frac{a_k(k-1)D_{k-1}}{k+2}
 \left(Y_k-Y_{k-1}\right),\notag\\
 \boldsymbol w_k^\top \vQ_k\boldsymbol w_k
 &=(D_N-\widetilde a_kD_k)\norm{Y_k}^2
 +\frac{\delta_k(k-1)D_{k-1}}{k+2}
 \csinner{Y_k}{Y_k-Y_{k-1}}\notag\\
 &\quad+\frac{a_k(k-1)^2D_{k-1}^2}{2(k+2)^2}
 \norm{Y_k-Y_{k-1}}^2.
 \label{eq:case-fista-position-term}
\end{align}
}
For \(1\leq k<N-1\), subtracting consecutive instances of
\eqref{eq:case-fista-position-term} yields
{%
\begin{equation}
\begin{aligned}
 &\boldsymbol w_{k+1}^\top \vQ_{k+1}\boldsymbol w_{k+1}
       -\boldsymbol w_k^\top \vQ_k\boldsymbol w_k\\
 &=\csinner{x_k-y_{k+1}}
       {\widetilde a_k(y_k-x_\star)
        -\widetilde a_{k+1}(y_{k+1}-x_\star)}\\
 &\quad-\frac{a_k}{2}\norm{Y_{k+1}-(x_k-y_{k+1})}^2
 -\alpha_k\norm{Y_{k+1}-Y_k}^2.
\end{aligned}
\label{eq:case-fista-position-increment}
\end{equation}
}
{Indeed, expand both sides using
\eqref{eq:case-fista-position-recurrence} and the weight identities.
The remaining coefficient comparisons reduce to
\[
\begin{aligned}
 \csinner{Y_{k+1}}{Y_{k+1}-Y_k}&:\quad
 \widetilde a_kD_k+\delta_k(D_{k+1}+1)-2D_N
 =\frac{kD_k\delta_{k+1}}{k+3},\\
 \norm{Y_{k+1}-Y_k}^2&:\quad
 \frac{a_k(D_{k+1}+1)^2}{2}-D_N
 -\frac{kD_k\delta_{k+1}}{k+3}
 -\frac{a_{k+1}k^2D_k^2}{2(k+3)^2}=\alpha_k.
\end{aligned}
\]
These are \eqref{eq:case-fista-adjacent-identity} and
\eqref{eq:case-fista-alpha}, respectively.}

{Equation~\eqref{eq:case-fista-position-increment} shows that the position terms in \({\boldsymbol r_k^\top\vS_k\boldsymbol r_k}\) equal
\[
 \frac{a_k}{2}\norm{Y_{k+1}-(x_k-y_{k+1})}^2
 +\alpha_k\norm{Y_{k+1}-Y_k}^2.
\]
The resulting matrix factorization is
\[
 \vS_k=\frac{a_k}{2}\begin{bmatrix}0\\1\\-1\\1\\0\end{bmatrix}\begin{bmatrix}0\\1\\-1\\1\\0\end{bmatrix}^{\!\top}
 +\frac{b_{k+1}}2\begin{bmatrix}0\\0\\0\\0\\1\end{bmatrix}\begin{bmatrix}0\\0\\0\\0\\1\end{bmatrix}^{\!\top}
 +\alpha_k\begin{bmatrix}-1\\1\\0\\0\\0\end{bmatrix}\begin{bmatrix}-1\\1\\0\\0\\0\end{bmatrix}^{\!\top}.
\]
Evaluating at \(\boldsymbol r_k\) gives}
{%
\begin{equation}\label{cs:fista:eq:factorization}
\begin{aligned}
 L{\boldsymbol r_k^\top\vS_k\boldsymbol r_k}
 &=\frac{La_k}{2}\norm{\frac{\nabla f(x_k)-\nabla f(x_{k+1})}{L}
 +Y_{k+1}-(x_k-y_{k+1})}^2\\
 &\quad+\frac{b_{k+1}}{2L}\norm{\nabla f(x_{k+1})-\nabla f(x_\star)}^2
 +L\alpha_k\norm{Y_{k+1}-Y_k}^2=S_k,
\end{aligned}
\end{equation}
}
{by the definition of \(s_k\).}
{Code Cell~18} verifies the position identities and this full quadratic
factorization coefficient by coefficient.

\medskip
\noindent\textbf{Step 2: verification of the interior decrement identity.}
For \(1\leq k<N-1\), it remains to prove
\(\mathcal T_k-(V_{k+1}-V_k)=L{\boldsymbol r_k^\top\vS_k\boldsymbol r_k}\).
{%
To match \eqref{cs:fista:eq:potential-difference}, compare the smooth and nonsmooth function values separately.
The identities \(b_{k+1}=a_{k+1}-a_k\) and \(\widetilde b_{k+1}=\widetilde a_{k+1}-\widetilde a_k\) give exactly the first two lines of that difference.
Expanding the residuals and using the selected subgradient \(L(x_k-y_{k+1})-\nabla f(x_k)\) gives
\begin{equation}\label{cs:fista:eq:residual-expansion}
\begin{aligned}
\mathcal T_k
 &=a_{k+1}\bigl(f(x_{k+1})-f(x_\star)\bigr)-a_k\bigl(f(x_k)-f(x_\star)\bigr)\\
 &\quad+\widetilde a_{k+1}\bigl(g(y_{k+1})-g(x_\star)\bigr)
 -\widetilde a_k\bigl(g(y_k)-g(x_\star)\bigr)\\
 &\quad+\frac{a_k}{2L}\norm{\nabla f(x_k)-\nabla f(x_{k+1})}^2
 +\frac{b_{k+1}}{2L}\norm{\nabla f(x_{k+1})-\nabla f(x_\star)}^2\\
 &\quad+\csinner{\nabla f(x_{k+1})}
 {a_k(x_k-x_\star)-a_{k+1}(x_{k+1}-x_\star)}\\
 &\quad+\csinner{L(x_k-y_{k+1})-\nabla f(x_k)}
 {\widetilde a_k(y_k-x_\star)-\widetilde a_{k+1}(y_{k+1}-x_\star)}.
\end{aligned}
\end{equation}
The function-value terms therefore cancel upon subtracting \eqref{cs:fista:eq:potential-difference}.
Using \(2\vd_k^\top\boldsymbol w_k=a_k(x_k-x_\star)-\widetilde a_k(y_k-x_\star)\) at indices \(k\) and \(k+1\), the remaining terms linear in the gradients equal
}
{%
\[
\begin{aligned}
 &\csinner{\nabla f(x_k)-\nabla f(x_{k+1})}
 {\widetilde a_{k+1}(y_{k+1}-x_\star)-a_k(x_k-x_\star)}\\
 &\qquad=a_k\csinner{\nabla f(x_k)-\nabla f(x_{k+1})}
 {Y_{k+1}-(x_k-y_{k+1})},
\end{aligned}
\]
}
where the equality uses
\(\widetilde a_{k+1}D_{k+1}=a_k(D_{k+1}+1)\).
The two smooth interpolation norm terms give the two pure-gradient
terms in \(L{\boldsymbol r_k^\top\vS_k\boldsymbol r_k}\); {the remaining terms give its position terms.} Thus
\begin{equation}\label{cs:fista:eq:interpolation-bridge}
\begin{aligned}
 &a_k\mathcal I_f(x_k,x_{k+1})
 +b_{k+1}\mathcal I_f(x_\star,x_{k+1})
 +\widetilde a_k\mathcal I_g(y_k,y_{k+1})\\
 &\qquad+\widetilde b_{k+1}\mathcal I_g(x_\star,y_{k+1})
 -(V_{k+1}-V_k)=L{\boldsymbol r_k^\top\vS_k\boldsymbol r_k}=S_k.
\end{aligned}
\end{equation}
The final equality is Step 1. Rearranging proves the second line of
\eqref{eq:case-fista-increment} on the full range \(1\leq k<N-1\).
{Code Cell~20} checks the bridge and the complete decrement in every
independent Gram, function-value, and constant coefficient. The remaining
initial and terminal cases are verified next, {with the signs needed for monotonicity established in Lemma~\ref{cs:fista:lem:signs}.}

\paragraph{Initial and terminal increments.}
{These calculations establish the first and last lines of
\eqref{eq:case-fista-increment}.}\footnote{{When \(N=2\), these are the only two transitions; Code Cells~21 and~22 verify both.}}
{Code Cell~21} checks the initial decrement and {Code Cell~22}
{checks the terminal decrement.}
For \(N\geq2\), \(x_1=y_1\) and
\(\widetilde a_1D_1=1\), so
{%
\[
 \boldsymbol w_1^\top \vQ_1\boldsymbol w_1
       =(D_N-1)\norm{Y_1}^2,\qquad
 2\vd_1^\top\boldsymbol w_1=\delta_1Y_1.
\]
}
The initial position identity is
{%
\[
\begin{aligned}
 &-\widetilde a_1\csinner{x_0-y_1}{y_1-x_\star}
 -\frac{a_0}{2}\norm{Y_1-(x_0-y_1)}^2\\
 &\qquad=(D_N-1)\norm{Y_1}^2-\tau_N\norm{x_0-x_\star}^2.
\end{aligned}
\]
}
Together with \(b_0=a_0\), this gives the stated initial increment
by the same expansion.

For \(N\geq2\), direct substitution at the terminal step in
\eqref{eq:case-fista-position-term} gives
{%
\begin{align*}
 -\boldsymbol w_{N-1}^\top \vQ_{N-1}\boldsymbol w_{N-1}
 &=\csinner{x_{N-1}-y_N}
       {\widetilde a_{N-1}(y_{N-1}-x_\star)-(y_N-x_\star)}\\
 &\quad-\frac{a_{N-1}}2\norm{Y_N-(x_{N-1}-y_N)}^2\\
 &\quad-\alpha_{N-1}\norm{Y_N-Y_{N-1}}^2.
\end{align*}
}
In \(\mathcal T_{N-1}-S_{N-1}\), all terms involving
\(\nabla f(y_N)\) cancel, using
\(a_{N-1}(D_N+1)=D_N\) and \(a_N=\widetilde a_N=1\).
Thus the terminal increment follows as well.
For \(N=1\), the values
\[
 D_1=1,\quad \tau_1=\frac14,\quad
 a_0=b_0=b_1=\frac12,\quad
 a_1=\widetilde a_1=\widetilde b_1=1,\quad \alpha_0=0
\]
give both boundary conventions in a single expansion. {Code Cell~21}
checks the literal one-step potential and
\eqref{cs:fista:eq:N1-decrement} directly, omitting \(g(y_0)\).

\paragraph{{Positivity of the coefficients.}}
{%
Lemma~\ref{cs:fista:lem:signs} proves the coefficient signs, including
the zero boundary values, and the denominator bounds.
}

Every interpolation weight and every square weight in
\eqref{eq:case-fista-increment} is therefore nonnegative.
Since \(\mathcal I_f,\mathcal I_g\leq0\),
\eqref{eq:case-fista-increment} and \eqref{cs:fista:eq:N1-decrement}
give \(V_N\leq V_{N-1}\leq\cdots\leq V_0=0\).
The endpoint normalization \eqref{cs:fista:eq:normalization}, checked in
{Code Cell~23}, proves the claimed rate.\footnote{{Code Cell~23 covers both \(N=1\) and general \(N\).}}
\end{proof}

\begin{lemma}%
\label{cs:fista:lem:signs}
\normalfont
{%
The coefficients are positive on their respective index ranges:\footnote{{The boundary values \(\widetilde a_0=\alpha_0=0\) are excluded from these strict inequalities; the unused value \(a_{-1}\) is also zero.}}
\[
 a_k>0,\qquad b_k>0,\qquad \widetilde a_k>0,\qquad \widetilde b_k>0,\qquad \alpha_k>0.
\]
}
\end{lemma}
\begin{proof}
{%
For \(1\leq k\leq N\),
\[
 D_k=D_N+\frac{(N-k)(N+k+5)}8>0,\qquad D_0=2D_N>0.
\]
Also, \(\delta_k>0\) for \(0\leq k<N\), \(D_0-D_1=1\), and
\(D_{k-1}-D_k=(k+2)/4>0\) for \(k\geq2\).
Code Cell~19 verifies the denominator identities, weight
differences, and the polynomial certificate below.
\begin{itemize}
\item \textbf{Positivity of \(a_k\).}
These coefficients are those of \eqref{cs:sfgm:eq:a}, so
Lemma~\ref{cs:sfgm:lem:signs} gives \(a_k>0\) for \(0\leq k\leq N\).

\item \textbf{Positivity of \(b_k\).}
The same lemma gives \(b_k>0\) for \(0\leq k\leq N\).

\item \textbf{Positivity of \(\widetilde a_k\).}
For \(1\leq k\leq N\), \(\widetilde a_k=(D_k+1)a_{k-1}/D_k>0\).
At index zero, \(\widetilde a_0=0\).

\item \textbf{Positivity of \(\widetilde b_k\).}
Equation~\eqref{eq:case-fista-weight-identities} gives, for \(1\leq k\leq N\),
\[
 \widetilde b_k
 =\frac{\delta_{k-1}+\widetilde a_{k-1}(D_{k-1}-D_k)}{D_k}>0.
\]

\item \textbf{Positivity of \(\alpha_k\).}
The value \(\alpha_0=0\) follows directly from \eqref{eq:case-fista-alpha}.
For \(N\geq2\), the terminal formula gives \(\alpha_{N-1}>0\).
For \(2\leq k<N\), {write \(\ell=N-k\in\mathbb Z\), with \(\ell\geq1\).}
Substitution in \eqref{eq:case-fista-alpha} gives
\begin{align*}
 &\frac{1024N(N+1)(k+2)^2(D_{k+1}+1)}{k-1}\alpha_{k-1}\\
 &\quad=8(5k+4)\ell^6+8(5k+4)(6k+11)\ell^5\\
 &\qquad+4(133k^3+599k^2+772k+296)\ell^4\\
 &\qquad+8(78k^4+491k^3+969k^2+699k+148)\ell^3\\
 &\qquad+2(223k^5+1854k^4+5153k^3+5774k^2+2408k+176)\ell^2\\
 &\qquad+2(70k^6+779k^5+2964k^4+4855k^3+3576k^2+988k+16)\ell\\
 &\qquad+k(k+1)(k+6)(k^2+5k+2)(15k^2+35k+22)>0.
\end{align*}
The multiplier of \(\alpha_{k-1}\) on the left and every term on the right
are positive, so \(\alpha_{k-1}>0\) for \(2\leq k<N\).
\end{itemize}
}
\end{proof}
\endgroup

\subsection{{Matching the original FGM conjectures}}
\label{app:fgm-conjecture-relation}

{The following lemma reconciles the \(C_k\) notation used in
Theorems~\ref{thm:case-ofgm-primary} and~\ref{thm:case-ofgm-secondary}
with the notation of the original conjectures in \citet{taylor2017smooth}.}
\begin{lemma}[Relation to the FGM conjecture notation]
\label{lem:case-fgm-conjecture-relation}
\normalfont
{Following \citet{taylor2017smooth}, define the fixed-step
coefficients \(h_{i,j}\in\mathbb R\) for \(1\leq i\leq N\) by}
{%
\begin{equation}
 x_i=x_0-\frac1L\sum_{j=0}^{i-1}h_{i,j}\nabla f(x_j),\qquad
 y_i=x_{i-1}-\frac1L\nabla f(x_{i-1}).
 \label{eq:case-fgm-h-definition}
\end{equation}
}
{Then, for \(0\leq i\leq N\), with the sum understood to be empty when \(i=0\),}
{%
\[
 \sum_{j=0}^{i-1}h_{i,j}=C_{i+1}-1.
\]
}
\end{lemma}

\begin{proof}
{Let \({\widetilde C_k}\in\mathbb R\) be given by}
\[
 {\widetilde C_k}=\sum_{j=0}^{k-1}{h_{k,j}},\qquad {\widetilde C_0}=0.
\]
For \(1\leq k\leq N\), expanding the \(y\)- and \(z\)-updates gives
\[
\begin{aligned}
 y_k&=x_0-\frac1L\left[
 \sum_{j=0}^{k-2}{h_{k-1,j}}\nabla f(x_j)
 +\nabla f(x_{k-1})\right],\\
 z_k&=x_0-\frac1L\sum_{j=0}^{k-1}\theta_j\nabla f(x_j),
\end{aligned}
\]
{with empty sums understood as zero. Substituting these expressions into
\(x_k=(1-1/\theta_k)y_k+z_k/\theta_k\) from \eqref{eq:case-ofgm} and collecting the terms involving
 each \(\nabla f(x_j)\) gives}
{%
\[
\begin{aligned}
 x_k=x_0
 &-\frac1L\sum_{j=0}^{k-2}
 \left[\left(1-\frac1{\theta_k}\right)h_{k-1,j}
 +\frac{\theta_j}{\theta_k}\right]\nabla f(x_j)
 -\frac1L\left[1-\frac1{\theta_k}
 +\frac{\theta_{k-1}}{\theta_k}\right]\nabla f(x_{k-1}).
\end{aligned}
\]
}
{Writing this expansion in the form of \eqref{eq:case-fgm-h-definition}
with \(i=k\) identifies the bracketed expressions as \(h_{k,j}\). Hence}
{%
\[
 \widetilde C_k=\sum_{j=0}^{k-1}h_{k,j}
 =\sum_{j=0}^{k-2}
 \left[\left(1-\frac1{\theta_k}\right)h_{k-1,j}
 +\frac{\theta_j}{\theta_k}\right]
 +\left[1-\frac1{\theta_k}+\frac{\theta_{k-1}}{\theta_k}\right].
\]
}
{Collecting terms and using the definition of \(\widetilde C_{k-1}\) gives}
{%
\begin{equation}
 \widetilde C_k
 =\left(1-\frac1{\theta_k}\right)
 \left(\sum_{j=0}^{k-2}h_{k-1,j}+1\right)
 +\frac1{\theta_k}\sum_{j=0}^{k-1}\theta_j
 =\left(1-\frac1{\theta_k}\right)(\widetilde C_{k-1}+1)
 +\frac1{\theta_k}\sum_{j=0}^{k-1}\theta_j.
 \label{eq:case-fgm-coefficient-sum}
\end{equation}
}
{Using \(\sum_{j=0}^{k-1}\theta_j=\theta_{k-1}^2\), obtained by telescoping
\eqref{eq:case-theta-identity}, and \(\theta_{k-1}^2=\theta_k^2-\theta_k\), we obtain}
\[
 {\widetilde C_k}+1=\theta_k+\left(1-\frac1{\theta_k}\right)({\widetilde C_{k-1}}+1).
\]
This is the recurrence \eqref{eq:case-C-recurrence} for \(C_{k+1}\).
Since \({\widetilde C_0}+1=C_1=1\), induction yields
\({\widetilde C_k}=C_{k+1}-1\) for every \(0\leq k\leq N\).

\end{proof}

{Applying Lemma~\ref{lem:case-fgm-conjecture-relation} with
\(i=N-1\) and \(i=N\) shows that the bounds in
Theorems~\ref{thm:case-ofgm-primary} and
\ref{thm:case-ofgm-secondary} are exactly the FGM expressions in
Conjectures~4 and~5 of \citet{taylor2017smooth}, respectively:}
{%
\begin{align*}
 \frac{LR^2\theta_{N-1}^2}
 {2\left(2\sum_{j=0}^{N-1}\theta_j^3+\theta_{N-1}^2\right)}
 &=\frac{LR^2}{2(2C_N+1)}
 =\frac{LR^2}{2}\frac1{
 2\sum_{j=0}^{N-2}h_{N-1,j}+3},\\[1ex]
 \frac{LR^2\theta_N^2}
 {2\left(2\sum_{j=0}^{N}\theta_j^3-\theta_N^2\right)}
 &=\frac{LR^2}{2(2C_{N+1}-1)}
 =\frac{LR^2}{2}\frac1{
 2\sum_{j=0}^{N-1}h_{N,j}+1}.
\end{align*}
}
{The first equalities follow from \eqref{eq:case-C-definition};
the second follow from Lemma~\ref{lem:case-fgm-conjecture-relation}.}

\subsubsection{{Smooth lower bounds}}

{\citet[Section~4.2, Conjectures~4--5]{taylor2017smooth}
give Huber instances for the two outputs of FGM.
For rational coefficients,
\citet[Section~4.2.2, Table~1 and the following paragraph]{TaylorHendrickxGlineur2017_exacta}
give the value at \(y_N\) and identify a Huber instance.
The \(z\) form of \eqref{eq:case-ofgm} has inertial coefficient
\((\theta_k-1)/\theta_{k+1}=k/(k+3)\) for FGM-rational, matching
that paper's FPGM1 update and \(y_N\) output.
We verify attainment for every horizon.}

\begin{lemma}\label{lem:case-smooth-lower}
{Fix \(N\geq1\), \(L>0\), and \(R\geq0\).
For \(R>0\), work on \(\mathbb R\), set \(x_0=y_0=z_0=R\)
and \(x_\star=0\), and use the convex function
\(f_\tau:\mathbb R\to\mathbb R\) given by}
\[
{f_\tau(t)=
\begin{cases}
 \dfrac{L}{2}t^2,& |t|\leq\tau,\\
 L\tau|t|-\dfrac{L}{2}\tau^2,& |t|\geq\tau.
\end{cases}}
\]
{For \eqref{eq:case-theta-recursive} at \(y_N\) and \(x_N\),
and \eqref{eq:case-theta-linear} at \(y_N\), choose, respectively,}
\[
{\tau=\frac{R}{2C_N+1},\qquad
\tau=\frac{R}{2C_{N+1}-1},\qquad
\tau=\frac{4R}{N^2+5N+6}.}
\]
{Then \eqref{eq:case-ofgm-primary-rate},
\eqref{eq:case-ofgm-secondary-rate}, and \eqref{eq:case-sfgm-rate}
hold with equality in the corresponding cases.
For \(R=0\), the zero function attains all three bounds.}
\end{lemma}

\begin{proof}
{Assume \(R>0\).
The derivative of \(f_\tau\) is continuous, nondecreasing,
and has Lipschitz constant \(L\). Thus \(f_\tau\in\mathcal F_L\),
\(x_\star=0\) is a minimizer, and \(|x_0-x_\star|=R\).}

{For each method, write the expansion of \eqref{eq:case-ofgm}
in the form \eqref{eq:case-fgm-h-definition} and set
\(\widetilde C_k=\sum_{j=0}^{k-1}h_{k,j}\in\mathbb R\), with \(\widetilde C_0=0\).
Equation~\eqref{eq:case-fgm-coefficient-sum} uses only the common updates
\eqref{eq:case-ofgm} with \(\eta=1\), without \eqref{eq:case-theta-identity},
so it applies to both parameter sequences.
Replacing \(k\) by \(k+1\) in \eqref{eq:case-fgm-coefficient-sum} gives}
{%
\[
\widetilde C_{k+1}=
\left(1-\frac1{\theta_{k+1}}\right)(\widetilde C_k+1)
+\frac1{\theta_{k+1}}\sum_{j=0}^{k}\theta_j.
\]
}
{Since \(\theta_k\geq1\), the recurrence is a convex combination of
\(\widetilde C_k+1\) and \(\sum_{j=0}^{k}\theta_j\).
Starting from \(\widetilde C_0=0\), induction gives
\(0\leq\widetilde C_k\leq\sum_{j=0}^{k-1}\theta_j\) and
\(\widetilde C_{k+1}\geq\widetilde C_k+1\), since
\(\widetilde C_k+1\leq\sum_{j=0}^{k-1}\theta_j+1\leq\sum_{j=0}^{k}\theta_j\).
We first verify identities involving only \(\widetilde C_k\) and \(\tau\):}
{%
\begin{itemize}
\item \textbf{FGM at \(y_N\) and \(x_N\).}
Lemma~\ref{lem:case-fgm-conjecture-relation} gives
\(\widetilde C_k=C_{k+1}-1\).
For \(\tau=R/(2C_N+1)\), the identity \(\widetilde C_{N-1}+1=C_N\)
gives \(\tau(\widetilde C_{N-1}+1)=(R-\tau)/2\).
For \(\tau=R/(2C_{N+1}-1)\), the identity \(\widetilde C_N=C_{N+1}-1\)
gives \(\tau\widetilde C_N=(R-\tau)/2\).

\item \textbf{FGM-rational at \(y_N\).}
Here \(\theta_k=(k+2)/2\) and \(\sum_{j=0}^{k}\theta_j=(k+1)(k+4)/4\).
Substituting these into the common recurrence proves
\(\widetilde C_k=k(k+7)/8\) by induction from \(\widetilde C_0=0\), since
\[
\widetilde C_{k+1}=
\frac{k+1}{k+3}\left(\frac{k(k+7)}8+1\right)
+\frac{(k+1)(k+4)}{2(k+3)}
=\frac{(k+1)(k+8)}8.
\]
Thus \(\widetilde C_{N-1}+1=(N^2+5N+2)/8\), so again
\(\tau(\widetilde C_{N-1}+1)=(R-\tau)/2\).
\end{itemize}
}

{These identities imply \(0<\tau\leq R\) and, by monotonicity,
\(\tau\widetilde C_k\leq(R-\tau)/2\) for \(0\leq k<N\) in each case.
We now prove by induction that \(x_k\in[\tau,\infty)\) for \(0\leq k<N\).
On this interval, the Huber function satisfies
\(f_\tau(t)=L\tau t-L\tau^2/2\) and \(\nabla f_\tau(t)=L\tau\).
}
{%
\begin{itemize}
\item Since \(x_0=R\geq\tau\), we have \(\nabla f_\tau(x_0)=L\tau\).

\item For \(1\leq k<N\), assume \(\nabla f_\tau(x_j)=L\tau\) for every \(0\leq j<k\).
The gradient expansion \eqref{eq:case-fgm-h-definition} for the chosen method gives
\[
 x_k=R-\frac1L\sum_{j=0}^{k-1}h_{k,j}\nabla f_\tau(x_j)
 =R-\tau\widetilde C_k
 \geq\frac{R+\tau}{2}\geq\tau.
\]
Thus \(x_k\) also lies in this affine interval and
\(\nabla f_\tau(x_k)=L\tau\), closing the induction.
\end{itemize}
}

{Using the same expansion, the \(y\)-update in \eqref{eq:case-ofgm},
and the respective choices of \(\tau\), we obtain}
{%
\[
 y_N=R-\tau(\widetilde C_{N-1}+1)=\frac{R+\tau}{2},\qquad
 x_N=R-\tau\widetilde C_N=\frac{R+\tau}{2}.
\]
}
{Thus each final output lies in the affine part of \(f_\tau\), and}
\[
{f_\tau\!\left(\frac{R+\tau}{2}\right)-f_\tau(0)
=L\tau\left(\frac{R+\tau}{2}-\frac{\tau}{2}\right)
=\frac{LR\tau}{2}.}
\]
{Substituting the three choices of \(\tau\) gives the claimed bounds.}
\end{proof}

\Needspace{8\baselineskip}
\subsubsection{{Composite lower bound}}

{\citet[Section~4.2.2, Table~1 and the following paragraph]{TaylorHendrickxGlineur2017_exacta}
give the constrained linear instance and its \(y_N\) value for
FPGM1. Its update and output coincide with \eqref{eq:case-fista}
after shifting the index by one. We verify its value directly.}

\begin{lemma}\label{lem:case-fista-lower}
{Fix \(N\geq1\), \(L>0\), and \(R\geq0\).
On \(\mathbb R\), set \(x_0=y_0=R\), \(x_\star=0\), and
\(c_N=4LR/(N^2+5N+2)\). Let \(f:\mathbb R\to\mathbb R\)
and \(g:\mathbb R\to\mathbb R\cup\{+\infty\}\) be}
\[
{f(t)=c_N t,\qquad
g(t)=
\begin{cases}
0,&t\geq0,\\
+\infty,&t<0.
\end{cases}}
\]
{Then \eqref{eq:case-fista} attains
\(F(y_N)-F(x_\star)=2LR^2/(N^2+5N+2)\).}
\end{lemma}

\begin{proof}
{The gradient of \(f\) is constant, so \(f\in\mathcal F_L\).
The function \(g\) is proper, closed, and convex.
The point \(0\) minimizes \(F=f+g\), satisfies
\(0\in\nabla f(0)+\partial g(0)\), and lies at distance \(R\)
from \(x_0\). If \(R=0\), then \(c_N=0\) and the claim is immediate.
Assume \(R>0\). The proximal step is projection onto
\([0,\infty)\). We prove by induction that}
\[
{x_k=R-\frac{c_N}{8L}k(k+7)\quad(0\leq k\leq N),
\qquad
y_k=R-\frac{c_N}{8L}(k^2+5k+2)\quad(1\leq k\leq N).}
\]
{The formula for \(x_0\) is immediate.
At each step \(0\leq k<N\), the inductive expression for \(x_k\)
gives}
\[
{x_k-\frac{c_N}{L}
=R-\frac{c_N}{8L}(k^2+7k+8)
\geq R-\frac{c_N}{8L}(N^2+5N+2)
=\frac R2>0.}
\]
{Thus the projection is inactive, and
\(y_{k+1}=x_k-c_N/L\) has the stated form.
For \(k=0\), the extrapolation coefficient vanishes and
\(x_1=y_1\), as required. For \(k\geq1\), the formulas give
\(y_{k+1}-y_k=-c_N(k+3)/(4L)\), so the extrapolation in
\eqref{eq:case-fista} yields the stated expression for
\(x_{k+1}\). This closes the induction.
In particular \(y_N=R/2\), and hence}
\[
{F(y_N)-F(0)=c_N\,\frac R2
=\frac{2LR^2}{N^2+5N+2}.}
\]
\end{proof}

}{%
}

\end{document}